\documentclass[11pt,oneside,english]{article}
\usepackage[T1]{fontenc}
\usepackage[utf8]{inputenc}
\usepackage{babel}
\usepackage{amstext}
\usepackage{enumitem}
\usepackage{amsthm}
\usepackage{amssymb}
\usepackage{stmaryrd}
\usepackage[unicode=true,pdfusetitle,
 bookmarks=true,bookmarksnumbered=false,bookmarksopen=false,
 breaklinks=false,pdfborder={0 0 1},backref=false,colorlinks=false]
 {hyperref}

 \usepackage{changes}
\usepackage{amsmath,mathtools,empheq}
\makeatletter
\numberwithin{equation}{section}
\numberwithin{figure}{section}
\theoremstyle{plain}
\newtheorem{thm}{\protect\theoremname}[section]
\theoremstyle{definition}
\newtheorem{defn}[thm]{\protect\definitionname}
\theoremstyle{remark}
\newtheorem{rem}[thm]{\protect\remarkname}
\theoremstyle{plain}

\theoremstyle{plain}
\newtheorem{prop}[thm]{\protect\propositionname}
\theoremstyle{plain}
\newtheorem{lem}[thm]{\protect\lemmaname}
\theoremstyle{plain}
\newtheorem{cor}[thm]{\protect\corollaryname}

\usepackage[a4paper]{geometry}
\usepackage{amsfonts}

\usepackage{times}

\allowdisplaybreaks[3]

\newcommand{\dd}{\mathrm d}
\newcommand{\E}{\mathbb E}

\newcommand{\Div}{\operatorname{div}}

\newcommand{\one}{\mathbf 1}

\makeatother

\providecommand{\assumptionname}{Assumption}
\providecommand{\corollaryname}{Corollary}
\providecommand{\definitionname}{Definition}
\providecommand{\lemmaname}{Lemma}
\providecommand{\propositionname}{Proposition}
\providecommand{\remarkname}{Remark}
\providecommand{\theoremname}{Theorem}

\begin{document}
\title{Well-posedness and large deviations for the obstacle problem of first-order stochastic conservation laws
 \thanks{
The first author is also supported by the Natural Science Foundation of Shanghai (Grant No.25ZR1402408). The second author is supported by Open Foundation of the State Key Laboratory of Mathematical Sciences (Grant No. 09), Beijing Institute of Technology Research Fund Program for Young Scholars and MIIT Key Laboratory of Mathematical Theory and Computation in Information Security.} }
\author{Ruoyang Liu\thanks{Department of Mathematics, Shanghai Normal University, Shanghai, China (Email: {\tt ryliu@shnu.edu.cn}).}
\and
Rangrang Zhang \thanks{ Corresponding author, School of Mathematics and Statistics,
Beijing Institute of Technology, Beijing 100081, China (Email:{\tt rrzhang@amss.ac.cn}).}
}
\date{}
\maketitle

\vspace{-1ex}
\begin{abstract}

This paper studies the obstacle problem for first-order  scalar conservation laws driven by multiplicative noise.
By adapting a barrier-substitution strategy to the kinetic formulation, we permit the reflection measure to be a general Radon measure and eliminate the usual obstacle-noise compatibility condition.
We establish the existence of kinetic solutions for continuous obstacles and further derive $L^1$-contraction and uniqueness under stronger spatial regularity.
Moreover, we prove a Freidlin--Wentzell large deviation principle in $L^1(0,T;L^1(\mathbb{T}^N))$.
Unlike existing large-deviation arguments based on control-uniform penalization of the obstacle, we work directly with the reflected skeleton equation and retain the possibly singular Radon reflection measure throughout, while a viscous approximation provides the spatial $H^1$-regularity required for compactness before passing to the inviscid limit.
Our results provide a hyperbolic counterpart to the large deviation theory for parabolic reflected SPDEs developed by Matoussi, Sabbagh, and Zhang (2021).

\medskip

\medskip


\end{abstract}
\tableofcontents
\section{Introduction}\label{sec:introduction}
In this paper, we study the well-posedness and a Freidlin--Wentzell large deviation principle for stochastic scalar conservation laws with a lower obstacle on the $N$-dimensional torus $\mathbb T^N$.
Let $T>0$ and set $Q_T=\mathbb T^N\times(0,T]$.
Given a lower obstacle $\psi=\psi(x,t)$, we consider the obstacle problem
\begin{equation}\label{eq:intro-reflected-model}
\left\{
\begin{aligned}
\dd u+\Div A(u)\dd t &=\Phi(u)\dd W(t)+\dd\nu, &&\text{in }Q_T,\\
u&\geq\psi, &&\text{a.e. in }Q_T,\\
u(0)&=u^0.
\end{aligned}
\right.
\end{equation}
Here $A:\mathbb R\to\mathbb R^N$ is the flux, $W$ is a cylindrical Wiener process on a (separable) Hilbert space $U$, and $\Phi$ is the multiplicative noise coefficient. The unknown consists of a pair $(u,\nu)$, where $u$ is the state variable and $\nu$ is a nonnegative Radon measure that enforces the constraint $u\geq\psi$.
We assume that $u^0\in L^\infty(\mathbb T^N)$ and  that $u^0\geq\psi(\cdot,0)$ almost everywhere, so that the initial datum is compatible with the obstacle.
The precise assumptions on $A$, $\Phi$, and $\psi$ are stated in Section~2.

In the obstacle-free case, the deterministic first-order scalar conservation laws are well studied in the PDE literature. 
Comprehensive accounts of entropy solution theory can be found in the  monograph \cite{Dafermos} and the works \cite{AWC06,O,P-V}, together with references therein.
Lions, Perthame, and Tadmor \cite{L-P-T} introduced a kinetic formulation equivalent to entropy admissibility for multidimensional scalar conservation laws.
For the stochastic conservation laws with multiplicative noise, Feng and Nualart \cite{F-N} introduced a notion of strong entropy solutions for the Cauchy problem over the whole spatial space, and they established the well-posedness of stochastic strong entropy solutions only in the one-dimensional space case.
Debussche and Vovelle \cite{DV10-publish} then established well-posedness for periodic multidimensional equations with additive or multiplicative noise by using kinetic solutions.
These results do not include an obstacle or a reflection measure.


Obstacle problems for deterministic first-order equations have been studied in several settings.
Barth\'elemy \cite{Barthelemy1988} considered a quasilinear equation with a measurable obstacle.
Levi and Vallet \cite{LeviVallet2001}, and later Levi \cite{Levi2001}, developed entropy formulations for bilateral obstacle problems in bounded domains.
Amorim, Neves, and Rodrigues \cite{AmorimNevesRodrigues2017} studied a multidimensional hyperbolic conservation law with an obstacle and a mass constraint.
These works provide the deterministic background for obstacle problems associated with first-order conservation laws.

Obstacle problems for SPDEs have also been widely studied.
Early results include \cite{haussmann1989stochastic,nualart1992white,donati1993white}, followed by further developments for quasilinear, degenerate, and more general reflected equations; see, for example, \cite{XuZhang2009,denis2014obstacle,yang2019obstacle,brzezniak2023reflection}.
Most related works formulate the problem by defining the solution as a pair $(u,\nu)$ satisfying a minimal Skorokhod condition:
\begin{equation}\label{eq:SKorokhod}
\int_{Q_T}(u-\psi)\dd\nu(x,t)=0.
\end{equation}
In this framework, the reflection measure acts only when the solution reaches the obstacle.

For stochastic conservation laws, the work of Biswas, Tahraoui, and Vallet \cite{biswas2023obstacle} is the closest to the present paper.
They studied a bilateral obstacle problem in a stochastic entropy framework and proved well-posedness by combining penalization with viscous regularization.
They also established a Lewy--Stampacchia inequality for the reflection terms.
A key assumption in their analysis is a compatibility condition between each obstacle and the stochastic forcing, which eliminates stochastic forcing over the obstacle--solution contact region.
In the notation of the present paper, this condition has the form
\[
\partial_t\left(\psi-\int_0^t\Phi(\psi(s))\,\dd W(s)\right)\in L^2(\Omega\times Q_T).
\]
This assumption is closely related to the use of the classical Skorokhod condition.
Indeed, an entropy solution has low regularity, so the pairing between $u-\psi$ and a singular reflection measure in \eqref{eq:SKorokhod} may not be well defined.
The above condition gives uniform $L^2$ estimates for the penalization terms, which ensures the limiting reflection measure $\nu$ possesses an $L^2(\Omega\times Q_T)$ density.
The Skorokhod condition can then be written in the usual form.
However, this also restricts the obstacle and the noise coefficients.
In particular, if $\psi\equiv c$ is constant, the condition requires $\Phi(c)=0$, that is, the noise must vanish at the obstacle level.
A similar assumption on the obstacle is also used by Liu and Tang \cite{liu2024obstacle} for stochastic porous-media equations.

Our approach follows a different route.
Instead of improving the regularity of the reflection measure in order to retain the Skorokhod condition, we encode the contact with the obstacle directly in the kinetic equation.
To explain this idea, introduce an additional variable $\xi\in\mathbb R$ and define
\[
f(x,t,\xi)=\one_{\{u(x,t)>\xi\}}.
\]
For each $(x,t)$, the function $f$ records all levels $\xi$ below $u(x,t)$ and therefore contains the same information as $u$.
This is the kinetic representation of the solution.

If both $u$ and $\nu$ were sufficiently regular, the contribution of the reflection measure to the kinetic equation would formally be $\delta_{\xi=u}\nu$, which means the reflection measure would be placed at the kinetic level corresponding to the value of the solution.
This expression is not well defined in our setting.
Indeed, $u$ is defined only up to sets of Lebesgue measure zero, whereas $\nu$ may be singular with respect to Lebesgue measure.
Hence, the equivalence class of $u$ does not determine its values on the support of $\nu$, and the product $\delta_{\xi=u}\nu$ cannot be defined directly.

To overcome this difficulty, we adapt the barrier-substitution idea introduced at the entropy level in \cite{du2024well} and at the kinetic
level for a diffusive Dean--Kawasaki-type equation in \cite{liu2026deankawasaki}.
The core idea is to shift the concentration from the unknown value $u$ to the known obstacle $\psi$, while introducing an extra nonnegative defect measure that quantifies the substitution cost.
The reflection contribution is therefore written as
\begin{equation*}
\delta_{\xi=\psi}\nu+\partial_\xi\lambda,
\end{equation*}
where  $\lambda$ is a nonnegative ``obstacle defect measure''.
This structure is not imposed formally; it follows from the penalized equations.
Let $u_n$ be a penalized solution and define
\[
\nu_n(\dd x,\dd t) = n(u_n-\psi)^-\dd x\dd t,
\]
where $r^-=\max\{-r,0\}$. Since $\nu_n$ is supported on the set $\{u_n<\psi\}$, one has the exact identity
 \[
\delta_{\xi=u_n}\nu_n = \delta_{\xi=\psi}\nu_n+\partial_\xi\lambda_n,
\]
where
 \[
\dd\lambda_n(x,t,\xi) = \one_{\{u_n<\xi<\psi\}}\, \dd\nu_n(x,t)\dd\xi.
 \]
Thus, $\lambda_n$ measures the kinetic interval between the penalized solution $u_n$ and the obstacle $\psi$.
The measures $\lambda_n$ are nonnegative, and the above identity is stable under weak convergence of the measures (see Remark \ref{rem:barrier-substitution-identity} for the corresponding identity for the viscous penalized approximation $u_{\alpha,n}$).
In the limit, the weak limit of $\lambda_n$ can be absorbed into the total kinetic measure $q$.
The resulting kinetic equation takes the form
\begin{equation}\label{eq:intro-kinetic-form}
\begin{aligned}
(\partial_t+a(\xi)\cdot\nabla_x)f ={}& \delta_{\xi=u}\Phi(u)\dot W +\partial_\xi q -\frac12\partial_\xi \bigl(G^2(x,\xi)\delta_{\xi=u}\bigr) +\delta_{\xi=\psi}\nu
\end{aligned}
\end{equation}
in the sense of distributions.
Here $a=A'$, $\Phi(u)e_k=g_k(\cdot,u)$, $ G^2(x,\xi)=\sum_{k\geq1}|g_k(x,\xi)|^2$, and $q$ is the total nonnegative kinetic measure, including the limiting obstacle defect.
This formulation avoids the undefined pairing between $u-\psi$ and $\nu$.
Therefore, there is no need to impose a classical Skorokhod condition or to require the reflection measure to have a density.
The measure $\nu$ may remain a general Radon measure, and the formulation covers a broader class of obstacles, including constant obstacles without the condition $\Phi(c)=0$.

Based on the kinetic formulation \eqref{eq:intro-kinetic-form}, we establish well-posedness of kinetic solutions to \eqref{eq:intro-reflected-model}.
\begin{thm}[cf. Theorems \ref{thm:uniqueness} and \ref{thm:existence-obstacle-solution}]
Assume \textbf{Hypothesis (H1)}. If $\psi$ and $u^0$ satisfy \textbf{Hypothesis (H2)}, then the lower obstacle problem \eqref{eq:intro-reflected-model} admits a kinetic solution $(u,\nu)$ in the sense of Definition \ref{def:dfn-1}. If $\psi$ satisfies stronger \textbf{Hypothesis (H2)$'$}, then any two solutions $(u_1,\nu_1)$ and $(u_2,\nu_2)$ with compatible initial data $u^0_1, u^0_2$ satisfy, for almost every $t\in[0,T]$,
\begin{equation*}
\E\Vert u_1(t)-u_2(t)\Vert _{L^1(\mathbb{T}^N)}\leq\Vert u^0_{1}-u^0_{2}\Vert _{L^1(\mathbb{T}^N)}.
\end{equation*}
In particular, if $u^0_1=u^0_2$ almost everywhere on $\mathbb{T}^N$, then $u_1=u_2$ in $L^1(\Omega\times Q_T)$ and $\nu_1=\nu_2$ as Radon measures on $\mathbb{T}^N\times[0,T)$, almost surely.
\end{thm}

Hypothesis (H1) collects the assumptions on the flux function $A$ and the coefficient $\Phi$, which remain identical to those commonly used for the obstacle-free case.
Hypothesis (H2) imposes a mild regularity condition on the barrier $\psi\in C([0,T]\times \mathbb{T}^N)$, which is sufficient for the existence of kinetic solutions to \eqref{eq:intro-reflected-model}.
To further establish the uniqueness, we strengthen this assumption to $\psi\in C^{0,2}([0,T]\times \mathbb{T}^N)$, which is stated in Hypothesis (H2)$'$.
Notably, constant obstacles $\psi\equiv c$ are admissible in our framework without requiring $\Phi(c)\equiv 0$.
Thus, our framework removes the obstacle-noise condition imposed in \cite{biswas2023obstacle,liu2024obstacle}.

A closely related recent contribution is our prior work  \cite{liu2026deankawasaki}, which develops a kinetic-level barrier-substitution principle for the generalized Dean--Kawasaki equation with correlated conservative noise and merely continuous obstacles.
Although the generalized Dean--Kawasaki equation possesses a scalar conservation law structure, the main estimates and compactness arguments in \cite{liu2026deankawasaki} fundamentally differ from the present work.
In particular, the solution concept in \cite[Definition 2.3(v)]{liu2026deankawasaki} incorporates the localized Sobolev regularity
\[
  (u\wedge K)\vee K^{-1}\in L^2\bigl(\Omega\times(0,T);H^1(\mathbb T^N)\bigr), \qquad K\in\mathbb N.
\]
This regularity, together with the associated parabolic dissipation estimates, allows the spatial-gradient terms arising in the kinetic
formulation to be controlled. 
However, no analogous $H^1$-estimate for truncations of the solution is available for the first-order stochastic conservation law considered here. 
We therefore treat the flux term by spatial integration by parts and close the estimates
using uniform $L^p(\mathbb T^N)$-bounds for $p\geq2$.

We next summarize the main steps in the proof of well-posedness.
The uniqueness of kinetic solutions is established via the doubling variables method.
A key new observation, compared with \cite{liu2026deankawasaki},  is that neither the kinetic measure nor the reflection measure has an atom at the initial time $t=0$.
This follows from the weak initial trace condition \eqref{eq:weak-initial} in Definition \ref{def:dfn-1}; see Lemma \ref{lem:no atom of q nu}.
The absence of initial atoms allows the doubling-of-variables argument to start from $t=0$.

Existence is proved through two related approximation problems: the purely penalized equation \eqref{eq:penalized-hyperbolic} and the viscous penalized equation \eqref{eq:penalized-viscous}.
We first fix the penalty parameter and let the viscosity tend to zero.
We then compare the resulting solutions for different penalty parameters, which yields a monotone sequence.
Uniform moment estimates for the solutions, kinetic measures, and penalty measures provide the compactness needed to pass to the limit as the penalty parameter tends to infinity.
Finally, we show that the limiting pair satisfies both the obstacle constraint and the kinetic formulation.
Since kinetic solutions of first-order equations do not have the time regularity available for variational solutions of parabolic obstacle problems, a separate trace argument is required to recover the initial condition.

\
\

The other purpose of the present paper is to study the small-noise problem.
From the perspective of statistical mechanics, the small-noise regime describes random fluctuations around the deterministic evolution.
In this context, establishing large deviation principles (LDP) serve as a core ingredient for refined analysis and yield deeper physical insights into the underlying evolutionary dynamics.

In the absence of an obstacle, large deviations for conservation laws have been studied in several settings.
Mariani \cite{Mariani} considered a one-dimensional viscous conservation law with simultaneously vanishing viscosity and noise.
Friz and Gess \cite{friz2016stochastic} derived support and large-deviations results for conservation laws driven by rough transport signals from continuity of the rough-path solution map.
Dong, Wu, Zhang and Zhang \cite{DWZZ20} proved a Freidlin--Wentzell large deviation principle for stochastic scalar conservation laws with small It\^o multiplicative noise by combining kinetic solutions with the weak convergence method.

The existing literature on large deviations for reflected stochastic PDEs focuses primarily on parabolic or monotone systems.
Xu and Zhang \cite{XuZhang2009} obtained well-posedness and an LDP for reflected stochastic heat equations driven by space--time white noise.
Matoussi, Sabbagh, and Zhang \cite{MSZ} established an LDP for quasilinear parabolic SPDEs with an obstacle.
More recently, Tahraoui \cite{Tahraoui2025} treated obstacle problems governed by a $T$-monotone operator using Lewy--Stampacchia inequalities to control the multiplier.
Related results for reflected Burgers-type and abstract stochastic evolution equations can be found in \cite{WangZhaiZhang2022,YangZhang2026SmallTime, BrzezniakLiZhang2024}.
These methods do not directly apply to \eqref{eq:intro-reflected-model}.
Our equation is first-order and has no viscous regularization, so the gradient estimates and spatial compactness available in parabolic or monotone problems are absent.
Moreover, in the skeleton equation, the reflection measure may be singular and depends on the control.
Therefore, both the obstacle-free argument of \cite{DWZZ20} and the parabolic techniques of \cite{MSZ,Tahraoui2025} require substantial modifications due to their reliance on gradient bounds or monotone-operator estimates.
Our second contribution is to establish a Freidlin--Wentzell LDP for multidimensional stochastic scalar conservation laws with an obstacle and a general Radon reflection measure in the kinetic framework.

For $\varepsilon>0$, let $(u^\varepsilon,\nu^\varepsilon)$ be the kinetic solution of
\begin{equation}\label{eq:intro-small-noise}
\left\{
\begin{aligned}
\dd u^\varepsilon+\Div A(u^\varepsilon)\dd t &= \sqrt{\varepsilon}\, \Phi(u^\varepsilon)\dd W(t) +\dd\nu^\varepsilon, &&\text{in }Q_T,\\
u^\varepsilon&\geq\psi, &&\text{a.e. in }Q_T,\\ u^\varepsilon(0)&=u^0.
\end{aligned}
\right.
\end{equation}
Our large-deviations result is formulated as follows.
\begin{thm}[Large deviations; cf. Theorem \ref{thm-3}]
If \textbf{Hypotheses (H1)} and \textbf{(H2)$'$} hold, then $\{u^{\varepsilon}\}_{\varepsilon>0}$ satisfies a large deviation principle on the space $L^1(0,T;L^1(\mathbb{T}^N))$, with the good rate function $I$ given by (\ref{equ-27-1}).
\end{thm}

A powerful tool for proving the Freidlin--Wentzell LDP is the weak convergence approach developed by Dupuis and Ellis \cite{DE97}.
The core of this method is to derive a variational representation for the Laplace transform of bounded continuous functionals, which establishes the equivalence between the LDP and the Laplace principle.
For Brownian functionals, elegant variational representation formulas were further obtained by Bou\'e and Dupuis \cite{BD98} and Budhiraja and Dupuis \cite{BD}.
In the infinite-dimensional setting, the method was further developed by Budhiraja, Dupuis, and Maroulas \cite{BDM2008}.
More recently, Matoussi, Sabbagh and Zhang \cite{MSZ} proposed a sufficient condition for the large deviation criteria for Brownian functionals, which is particularly well-suited for fluid-mechanical stochastic PDEs.
In this work, we adopt this criterion, which reduces the proof to two ingredients: the weak--strong continuity of the skeleton map, and the convergence of the stochastic controlled equation toward its skeleton counterpart in probability.

For reflected equations, verifying the weak--strong continuity of the skeleton map is particularly delicate.
Most existing works  \cite{MSZ,brzezniak2023reflection,WangZhaiZhang2022} employ a penalization argument, where reflection terms are replaced by penalized approximations. 
In such frameworks, compactness, which transfers the weak convergence of controls to the strong convergence of solutions, is established for the penalized skeleton equations, with associated penalization errors that are uniform over bounded sets of controls.
However, this strategy is not applicable to our kinetic setting.
The available $L^1$ estimates for penalized kinetic or entropy solutions are essentially one-sided, namely, they only yield comparison, monotonicity, and convergence of penalized solutions (see
\cite{liu2024obstacle,du2024well,liu2026deankawasaki}), but fail to provide two-sided penalization errors uniform in the control.
An alternative approach from
Tahraoui \cite{Tahraoui2025} uses the Lewy--Stampacchia inequality to control reflection terms, yet it requires an obstacle–noise compatibility condition that  rules out our obstacles.

Our main contribution to the large-deviation argument is therefore to treat the reflected skeleton equation directly, without penalizing the obstacle.
The possibly singular Radon reflection measure is retained throughout the compactness argument and is controlled in a sufficiently negative Sobolev space by its uniform mass bound. 
We add viscosity only to recover a spatial $H^1$-estimate and hence the compact embedding needed for strong compactness. 
Together with a uniform time-translation estimate, this yields compactness of the reflected viscous skeleton solutions (see Proposition \ref{prop:weak-continuity-viscous-skeleton-obstacle}). 
A doubling-of-variables argument then removes the viscosity uniformly over bounded sets of controls and gives the weak--strong continuity of the original skeleton map.

Finally, the second ingredient is achieved via the doubling variables method, together with uniform estimates for controlled solutions and their associated reflection measures.
Taken together, these two components enable us to establish an LDP for stochastic scalar conservation laws with general continuous obstacles and possibly singular Radon reflection measures, requiring no obstacle-noise compatibility condition.
Notably, our compactness strategy offers a direct proof of weak--strong continuity for skeleton obstacle problems governed by nondegenerate equations, bypassing penalization approximations.

\

The remainder of the paper is structured as follows.
Section 2 sets up the notation, states the  assumptions, and introduces the notion of kinetic solutions for our model.
Section 3 addresses the uniqueness of kinetic solutions to the stochastic obstacle conservation law.
The existence of such kinetic solutions is established in Section 4.
Section 5 lays out the large deviation framework and formulates our large deviation result.
Section 6 is devoted to an analysis of the skeleton equation, where we prove weak--strong continuity of the associated skeleton mapping.
Section 7 completes the proof of the full large deviation principle by proving the convergence of the stochastic controlled equation toward its skeleton counterpart in probability.

\
\

\noindent \textbf{Summary of Approximation Equations}

\

Throughout this work, we introduce a collection of approximation equations. To help the reader track these approximate systems, we collect them below.

In the proof of well-posedness for the obstacle problem of stochastic conservation laws \eqref{eq:intro-reflected-model}, we employ two approximation equations. The first is the viscous penalized equation
\begin{equation*}
\text{\textbf{App Equ (1)}}\quad
\left\{
\begin{aligned}
\dd u_{\alpha,n} + \Div A(u_{\alpha,n})\dd t - \alpha\Delta u_{\alpha,n}\dd t
&= \Phi_\alpha(u_{\alpha,n})\dd W(t) + b_n(u_{\alpha,n}-\psi)\dd t,\\
u_{\alpha,n}(0) &= u_\alpha^0.
\end{aligned}
\right.
\end{equation*}
and the second is the penalized equation
\begin{equation*}
\text{\textbf{App Equ (2)}}\quad
\left\{
\begin{aligned}
\dd u_n+\Div A(u_n)\dd t
&=\Phi(u_n)\dd W(t)+b_n(u_n-\psi)\dd t,\\
u_n(0) &= u^0.
\end{aligned}
\right.
\end{equation*}

To establish well-posedness for the skeleton equation associated with \eqref{eq:intro-reflected-model}, we also require two approximation equations: the viscous penalized skeleton equation
\begin{equation*}
\text{\textbf{App Equ (i)}}\quad
\left\{
\begin{aligned}
\partial_tu^h_{\alpha,n}+\Div A(u^h_{\alpha,n})-\alpha\Delta u^h_{\alpha,n}
&=\Phi_\alpha(u^h_{\alpha,n})h(t)+b_n(u^h_{\alpha,n}-\psi),\\
u^h_{\alpha,n}(0)
&=u^0_{\alpha}.
\end{aligned}
\right.
\end{equation*}
and the penalized skeleton equation
\begin{equation*}
\text{\textbf{App Equ (ii)}}\quad
\left\{
\begin{aligned}
\partial_t u_n^h+\Div A(u_n^h)
&=\sum_{k\geq1}g_k(x,u_n^h)h_k(t)+b_n(u_n^h-\psi),\\
u_n^h(0)
&=u^0.
\end{aligned}
\right.
\end{equation*}

Finally, to prove the weak–strong continuity of the skeleton equation, we make use of an approximation equation, namely the viscous skeleton equation:
\begin{equation*}
\text{\textbf{App Equ}}\ (\ast)\quad
\left\{
\begin{aligned}
\partial_t u_\alpha^h+\Div A(u_\alpha^h)-\alpha\Delta u_\alpha^h
&=\Phi(u_\alpha^h)h(t)+\nu_\alpha^h,
&&\text{in }\mathbb T^N\times(0,T),\\
u_\alpha^h
&\geq\psi,
&&\text{a.e. in }\mathbb T^N\times(0,T),\\
u_\alpha^h(0)
&=u^0.
\end{aligned}
\right.
\end{equation*}


\section{Preliminaries}
Let $(\Omega,\mathcal{F},\mathbb{P},(\mathcal{F}_t)_{t\in[0,T]},((\beta_k(t))_{t\in[0,T]})_{k\in\mathbb{N}})$ be a stochastic basis, where $(\mathcal F_t)_{t\in[0,T]}$ is the usual augmentation of the natural filtration generated by the independent standard Brownian motions $(\beta_k)_{k\in\mathbb N}$.
We use $\E$ to denote the expectation with respect to $\mathbb{P}$.
$W$ is a cylindrical Wiener process defined on a given (separable) Hilbert space $U$ (the norm is denoted by $|\cdot|_U$) with the form $W(t)=\sum_{k\geq 1}\beta_k(t) e_k$, $t\in[0,T]$, where $(e_k)_{k\geq 1}$ is a complete orthonormal basis in the Hilbert space $U$.

Let $\mathcal{L}(K_1,K_2)$ (resp. $\mathcal{L}_2(K_1,K_2)$) be the space of bounded (resp. Hilbert-Schmidt) linear operators from a Hilbert space $K_1$ to another Hilbert space $K_2$, whose norm is denoted by $\Vert \cdot\Vert _{\mathcal{L}(K_1, K_2)}$(resp. $\Vert \cdot\Vert _{\mathcal{L}_2(K_1, K_2)})$. Further, $C_b$ represents the space of bounded, continuous functions and $C^1_b$ stands for the space of bounded, continuously differentiable functions having bounded first-order derivative. $C^{0,2}\big([0,T]\times\mathbb{T}^N\big)$ denotes the space of functions continuous in time and twice continuously differentiable in space on $[0,T]\times\mathbb{T}^N$.
Let $\Vert \cdot\Vert _{L^p(\mathbb{T}^N)}$ denote the norm of the Lebesgue space $L^p(\mathbb{T}^N)$ for $p\in [1,\infty]$. 
In particular, set $H=L^2(\mathbb{T}^N)$ with the corresponding norm $\Vert \cdot\Vert _H$. 
For all $m\in\mathbb{N}_0$, let $H^m(\mathbb{T}^N)=W^{m,2}(\mathbb{T}^N)$ be the usual Sobolev space with the norm
\[
\|u\|_{H^m(\mathbb T^N)}^2:=\sum_{\substack{\beta\in\mathbb N_0^N\\|\beta|\leq m}}\|D^\beta u\|_{L^2(\mathbb T^N)}^2,\qquad|\beta|:=\beta_1+\cdots+\beta_N.
\]
We denote by $H^{-m}(\mathbb T^N)$ the topological dual of $H^m(\mathbb T^N)$, whose norm is denoted by $\Vert \cdot\Vert _{H^{-m}(\mathbb{T}^N)}$. 
Moreover, we use the brackets $\langle\cdot,\cdot\rangle$ to denote the duality between $C^{\infty}_c(\mathbb{T}^N\times \mathbb{R})$ and the space of distributions over $\mathbb{T}^N\times \mathbb{R}$.
For $1\leq p\leq\infty$, let $p'$ be the conjugate exponent of $p$, that is,
\[
\frac1p+\frac1{p'}=1,
\]
with the conventions $p'=\infty$ when $p=1$ and $p'=1$ when $p=\infty$. 
We denote
\[
\langle F, G \rangle:=\int_{\mathbb{T}^N}\int_{\mathbb{R}}F(x,\xi)G(x,\xi)\dd x\dd\xi, \quad F\in L^p(\mathbb{T}^N\times \mathbb{R}), \,G\in L^{p'}(\mathbb{T}^N\times \mathbb{R}),
\]
and also for a measure $q$ on the Borel measurable space $\mathbb{T}^N\times[0,T]\times \mathbb{R}$
\[
q(\phi):=\langle q, \phi \rangle:=\int_{\mathbb{T}^N\times[0,T]\times \mathbb{R}}\phi(x,t,\xi)\dd q(x,t,\xi), \quad  \phi\in C_b(\mathbb{T}^N\times[0,T]\times \mathbb{R}).
\]

Throughout the paper, we frequently encounter integrals over various domains $Z$, including $[0,T]\times \mathbb{T}^N$, $[0,T]\times \mathbb{T}^N\times \mathbb{R}$ and $[0,T]\times (\mathbb{T}^N)^2\times \mathbb{R}^2$, among others. For brevity, we write $\int_Z f$ in place of $\int_Z f \dd z$. We specify the kinetic and reflection measures, but omit the Lebesgue measure. Integration variables will be displayed explicitly only when necessary.

\subsection{Hypotheses}
For the flux function $A$ and the coefficient $\Phi$, we assume
\begin{description}
  \item[\textbf{Hypothesis (H1)}] The flux function $A$ belongs to $C^2(\mathbb{R};\mathbb{R}^N)$ and its derivative $a$ has at most polynomial growth. That is, there exist constants $C>0, p>1$ such that
    \begin{align*}
     |a(\xi)-a(\zeta)|\leq \Gamma(\xi,\zeta)|\xi-\zeta|, \quad \Gamma(\xi,\zeta)=C(1+|\xi|^{p-1}+|\zeta|^{p-1}).
      \end{align*}
For each $\xi\in \mathbb{R}$, the map $\Phi(\xi): U\rightarrow H$ is defined by $\Phi(\xi) e_k=g_k(\cdot, \xi)$, where each $g_k(\cdot,\xi)$ is a regular function on $\mathbb{T}^N$.
We assume that $g_k\in C(\mathbb{T}^N\times \mathbb{R})$ and there exists a constant $D_0>0$ such that, for all $x,y\in\mathbb{T}^N$ and $\xi,\zeta\in\mathbb{R}$,
\begin{align}\label{eq:assumption for g}
G^2(x,\xi)=\sum_{k\geq 1}|g_k(x,\xi)|^2&\leq D_0(1+|\xi|^2),\\
\label{eq:assumption for g Lip}
\sum_{k\geq 1}|g_k(x,\xi)-g_k(y,\zeta)|^2&\leq D_0\Big(|x-y|^2+|\xi-\zeta|^2\Big).
\end{align}
\end{description}

Since $\Vert g_k(\cdot,\xi)\Vert _{H}\leq C\Vert g_k(\cdot,\xi)\Vert _{C(\mathbb{T}^N)}$, we deduce that $\Phi(\xi)\in \mathcal{L}_2(U,H)$, for each $\xi\in \mathbb{R}$.
We remark that Hypothesis (H1) is in force in the whole paper.

Regarding the obstacle $\psi$ and the initial data of (\ref{eq:intro-reflected-model}), we impose two kinds of conditions, which are employed separately for the existence and uniqueness.
\begin{description}
  \item[\textbf{Hypothesis (H2)}]
The obstacle satisfies $\psi\in C([0,T]\times\mathbb{T}^N)$ and $u^0\in L^\infty(\mathbb{T}^N)$ with
\begin{equation*}
u^0(x)\geq\psi(x,0)\quad\text{for a.e. }x\in\mathbb{T}^N.
\end{equation*}
\end{description}
A stronger condition on the barrier $\psi$ is
\begin{description}
  \item[\textbf{Hypothesis (H2)$'$}]
The obstacle satisfies $\psi\in C^{0,2}([0,T]\times\mathbb{T}^N)$ and the initial datum $u^0\in L^\infty(\mathbb{T}^N)$ satisfies the same compatibility condition as
in \textbf{Hypothesis (H2)}.
\end{description}

\subsection{Definition of kinetic solution }
We work throughout on the stochastic basis fixed above.
\begin{defn}[Kinetic measure]\label{def:dfn-3}
 A map $q$ from $\Omega$ to the set of nonnegative, finite measures over $\mathbb{T}^N\times [0,T]\times\mathbb{R}$ is said to be a kinetic measure, if
\begin{description}
  \item[1.] $ q$ is measurable, that is, for each $\phi\in C_b(\mathbb{T}^N\times [0,T]\times \mathbb{R}), \langle q, \phi \rangle: \Omega\rightarrow \mathbb{R}$ is measurable,
  \item[2.] $q$ vanishes for large $\xi$, i.e.,
\begin{equation}\label{eq:vanishing kinetic measure}
\lim_{R\rightarrow +\infty}\E[q(\mathbb{T}^N\times [0,T]\times B^c_R)]=0,
\end{equation}
where $B^c_R:=\{\xi\in \mathbb{R}: |\xi|\geq R\}$.
  \item[3.] for every $\phi\in C_b(\mathbb{T}^N\times \mathbb{R})$, the process
\[
(\omega,t)\in\Omega\times[0,T]\mapsto \int_{\mathbb{T}^N\times [0,t]\times \mathbb{R}}\phi(x,\xi)\dd q(x,s,\xi)\in\mathbb{R}
\]
is predictable.
\end{description}
\end{defn}

\begin{defn}[Kinetic solution]\label{def:dfn-1}
A pair $(u,\nu)$ is a kinetic solution of \eqref{eq:intro-reflected-model} with initial datum $u^0$ and obstacle $\psi$ if the following conditions hold.
\begin{enumerate}[label=\textbf{\arabic*}.]
  \item The process $u$ is predictable as an $L^1(\mathbb{T}^N)$-valued process.
  \item For any $p\geq1$, there exists $C_p\geq0$ such that
\begin{equation}\label{eq:usual moment bound}
\E\Big(\underset{0\leq t\leq T}{{\rm{ess\sup}}}\ \Vert u(t)\Vert ^p_{L^p(\mathbb{T}^N)}\Big)\leq C_p,
\end{equation}
\item \textbf{Obstacle constraint.} For $\mathbb{P}\otimes\dd x\otimes\dd t$-a.e. $(\omega,x,t)$,
\begin{equation*}
u(x,t)\geq\psi(x,t),
\end{equation*}
and $\nu$ is a nonnegative predictable Radon measure on $\mathbb{T}^N\times[0,T)$ satisfying $\E\nu(\mathbb{T}^N\times[0,T))<\infty$.
\item \textbf{Weak initial trace condition.} The process $u$ almost surely satisfies
\begin{equation}
\lim_{\tau\downarrow0}\frac{1}{\tau}\int_0^\tau\Vert u(t)-u^0\Vert _{L^1(\mathbb{T}^N)}\dd t=0.\label{eq:weak-initial}
\end{equation}
\item There exists a kinetic measure $q$ such that $f:= \one_{u>\xi}$ satisfies, for all $\varphi\in C^1_c(\mathbb{T}^N\times [0,T)\times \mathbb{R})$,
\begin{align}\notag
&\int^T_0\langle f(t), \partial_t \varphi(t)\rangle \dd t+\langle f^0, \varphi(0)\rangle +\int^T_0\langle f(t), a(\xi)\cdot \nabla \varphi (t)\rangle \dd t\\
&= -\sum_{k\geq 1}\int^T_0\int_{\mathbb{T}^N} \int_{\mathbb{R}}g_k(x,\xi)\varphi (x,t,\xi)\dd\mu_{x,t}(\xi)\dd x\dd\beta_k(t) \label{eq:kinetic eqution}\\ \notag
&\quad -\frac{1}{2}\sum_{k\geq1}\int^T_0\int_{\mathbb{T}^N}\int_{\mathbb{R}}\partial_{\xi}\varphi (x,t,\xi)G^2(x,\xi)\dd\mu_{x,t}(\xi)\dd x\dd t\\
&\quad + q(\partial_{\xi} \varphi)-\int_0^T\int_{\mathbb{T}^N}\varphi(x,t,\psi(x,t))\dd\nu(x,t), \quad\ a.s. ,\nonumber
\end{align}
 where $f^0:=\one_{u^0>\xi}$, $a(\xi)=A'(\xi)$ and  $\mu_{x,t}(\dd\xi)=\delta_{u(x,t)}(\dd\xi)$.
\end{enumerate}
\end{defn}

Compared with the obstacle-free case treated in \cite{DV10-publish}, the above definition has two major differences.
First, the obstacle constraint arises from the obstacle problem and constitutes a new ingredient.
Second, we introduce a weak initial trace condition to handle the problem that $\nu$ may have atoms.
Indeed, the obstacle problem yields an additional term involving the Radon measure $\nu$.
Crucially, the integral of $\varphi$ with respect to $\nu$ fails to vanish whenever $\varphi$ is independent of $\xi$.
The atoms of the measure $\nu$ can therefore cause discontinuities in $u$ at a random (almost surely countable) set of points. 
In Lemma \ref{lem:no atom of q nu}, we show that the weak initial trace condition excludes atoms of the kinetic measure $q$ and of the reflection measure $\nu$ at time $t=0$.

The following proposition states that, almost surely, the left and right weak limits of a kinetic function $f(x,t,\xi)=\one_{u(x,t)>\xi}$ exist at every $t\in [0,T]$. 
This property derive a formulation at fixed times that is weak only with respect to $x$ and $\xi$.
The proof is unchanged except that, for each fixed kinetic test function, the kinetic measure and the reflection measure together form the finite-variation part in
time. 
We record this difference below.
\begin{prop}\label{prop:trace-obstacle-scl}
Let $(u,\nu)$ be a kinetic solution in the sense of Definition
\ref{def:dfn-1}.
Let
\[
f(x,t,\xi)=\one_{u(x,t)>\xi}.
\]
Then, almost surely, $f$ admits a right weak limit $f^{+,t}$ for every
$t\in[0,T)$ and a left weak limit $f^{-,t}$ for every $t\in(0,T]$.
These limits are kinetic functions and, for every
$\varphi\in C_c^1(\mathbb T^N\times\mathbb R)$,
\[
\langle f(t+\tau),\varphi\rangle\rightarrow\langle f^{+,t},\varphi\rangle,\qquad\langle f(t-\tau),\varphi\rangle\rightarrow\langle f^{-,t},\varphi\rangle
\]
as $\tau\downarrow0$, whenever the corresponding one-sided limit is defined. 
With the convention $f^{-,0}:=f^0$, for every $t\in[0,T)$, we have
\begin{align}\label{eq:trace-jump-obstacle-scl}
\langle f^{+,t}-f^{-,t},\varphi\rangle
&=-\int_{\mathbb{T}^N\times\{t\}\times\mathbb{R}}\partial_\xi\varphi(x,\xi)\,\dd q(x,s,\xi)  \notag\\
&\quad+\int_{\mathbb{T}^N\times\{t\}}\varphi(x,\psi(x,s))\,\dd\nu(x,s).
\end{align}
In particular, the set of $t\in[0,T)$ at which $f^{+,t}\neq f^{-,t}$ is almost surely at most countable.
We denote the corresponding predictable representatives by $f^+$ and $f^-$.
\end{prop}

\begin{proof}
We follow \cite[Proposition~10]{DV10-publish} and only explain the additional reflection term. 
Fix $\varphi\in C_c^1(\mathbb T^N\times\mathbb R)$ and let
\begin{align*}
J_\varphi(t)&:=\int_0^t\langle f(s),a(\xi)\cdot\nabla_x\varphi\rangle\dd s\\
&\quad+\sum_{k\geq1}\int_0^t\int_{\mathbb T^N}\int_{\mathbb R}g_k(x,\xi)\varphi(x,\xi)\,\dd\mu_{x,s}(\xi)\dd x\dd\beta_k(s)\\
&\quad+\frac12\int_0^t\int_{\mathbb T^N}\int_{\mathbb R}G^2(x,\xi)\partial_\xi\varphi(x,\xi)\dd\mu_{x,s}(\xi)\dd x\dd s.
\end{align*}
As in \cite[Proposition~10]{DV10-publish}, $J_\varphi$ has a continuous modification.
Define the finite signed measure on $[0,T)$ by
\begin{align*}
\mathcal R_\varphi(B)&:=\int_{\mathbb T^N\times B\times\mathbb R}\partial_\xi\varphi(x,\xi)\dd q(x,s,\xi)-\int_{\mathbb T^N\times B}\varphi(x,\psi(x,s))\,\dd\nu(x,s),\quad B\in\mathcal B([0,T)).
\end{align*}
It is finite because $q$ and $\nu$ are finite and $\psi$ is continuous. 
Testing \eqref{eq:kinetic eqution} with $\alpha(t)\varphi(x,\xi)$, $\alpha\in C_c^1([0,T))$, gives almost surely
\[
\dd\bigl(\langle f,\varphi\rangle-J_\varphi\bigr)=-\dd\mathcal R_\varphi
\]
in the sense of distributions in time. 
Thus $t\mapsto\langle f(t),\varphi\rangle$ has bounded variation up to the continuous process $J_\varphi$.

The separability and time-averaging compactness argument in \cite[Proposition~10]{DV10-publish} now applies and yields the kinetic functions $f^{+,t}$ and $f^{-,t}$. 
Since $J_\varphi$ is continuous, their jump is the atom of the finite-variation part:
\[
\langle f^{+,t}-f^{-,t},\varphi\rangle=-\mathcal R_\varphi(\{t\}),
\]
which is exactly \eqref{eq:trace-jump-obstacle-scl}. 
If $t$ is an atom of neither the time marginal of $q$ nor that of $\nu$, this jump
vanishes for every $\varphi$. 
Since finite measures have at most countably many atoms, the set of jump times is at most countable, which completes the proof.
\end{proof}

\begin{rem}\label{rem:remark for intergral of f^+}
For a countable dense family of test functions, the scalar representatives are
\[
\langle f^+(t),\varphi\rangle=\langle f^0,\varphi\rangle+J_\varphi(t)-\mathcal R_\varphi([0,t]),\qquad\langle f^-(t),\varphi\rangle=\langle f^0,\varphi\rangle+J_\varphi(t)-\mathcal R_\varphi([0,t)).
\]
The first process is adapted and c\`adl\`ag, while the second is adapted and left-continuous. 
Since the filtration is the usual augmented Brownian filtration, it is continuous and the optional and predictable $\sigma$-fields coincide. 
Hence both representatives can be chosen predictable, exactly as in \cite[Proposition 10]{DV10-publish}. 
An atom of $q$ or $\nu$ only produces the predictable jump displayed in \eqref{eq:trace-jump-obstacle-scl}.
Moreover, we have $f=f^{+}=f^{-}$ almost everywhere with respect to the Lebesgue measure on the time interval $[0,T]$.
Hence, the choice among $f$, $f^{+}$, and $f^{-}$ is immaterial in integrals with respect to the Lebesgue measure, as well as in stochastic integrals.
The same holds for the Young measure $\mu_{x,t}$ (see \cite[Definition 4]{DV10-publish}).
However, for integrals with respect to measures $\nu$ and $q$, one must explicitly specify $f^{+}$ or $f^{-}$; the two choices may lead to different values.
\end{rem}

\subsection{No initial atoms of the kinetic and reflection measures}
We show that \eqref{eq:weak-initial} in Definition \ref{def:dfn-1} guarantees that $t=0$ is not an atom of the kinetic measure $q$ and the reflection measure $\nu$. This property enables us to apply the doubling variables method on $[0,t]$ to obtain uniqueness.
\begin{lem}\label{lem:no atom of q nu} Let $(u,\nu)$ be a kinetic solution under Definition \ref{def:dfn-1}, and let $q$ be the corresponding kinetic measure. Then, almost surely
\[
q(\mathbb{T}^N\times\{0\}\times\mathbb{R})=0,\qquad\nu(\mathbb{T}^N\times\{0\})=0.
\]
\end{lem}
\begin{proof}
Let $f^0(x,\xi):=\one_{u^0(x)>\xi}$, and $f^{+,0}$ be the right weak limit of $f$ at $t=0$.
We first show that almost surely $f^{+,0}=f^0$ as a distribution on $\mathbb{T}^N\times\mathbb{R}$.
For every $\phi\in C_c^1(\mathbb{T}^N\times\mathbb{R})$, the kinetic function $f(x,t,\xi)=\one_{u(x,t)>\xi}$ satisfies
\begin{align*}
&\bigg|\frac{1}{\tau}\int_0^\tau\int_{\mathbb{T}^N}\int_\mathbb{R}\big(f(x,t,\xi)-f^0(x,\xi)\big)\phi(x,\xi)\dd \xi\dd x\dd t\bigg|\\
&\leq\Vert\phi\Vert_{L^\infty(\mathbb{T}^N\times\mathbb{R})}\frac{1}{\tau}\int_0^\tau\int_{\mathbb{T}^N}\int_\mathbb{R}|f(x,t,\xi)-f^0(x,\xi)|\dd \xi\dd x\dd t\\
&=\Vert\phi\Vert_{L^\infty(\mathbb{T}^N\times\mathbb{R})}\frac{1}{\tau}\int_0^\tau\int_{\mathbb{T}^N}|u(x,t)-u^0(x)|\dd x\dd t\\
&\rightarrow 0,
\end{align*}
almost surely, when $\tau\to0$. In the penultimate step, we employ the identity $\int_{\mathbb{R}}|\one_{a>\xi}-\one_{b>\xi}|\dd\xi=|a-b|$, while the final step relies on \eqref{eq:weak-initial} from Definition \ref{def:dfn-1}.

Therefore, by the definition of right weak limit given by Proposition \ref{prop:trace-obstacle-scl}, we have almost surely
\begin{align}\label{eq:f^+_0=f0-1}
\langle f^{+,0},\phi\rangle=\lim_{\tau\to0}\frac{1}{\tau}\int_0^\tau\langle f(t),\phi\rangle\dd t=\langle f^0,\phi\rangle,
\end{align}
which yields almost surely
\begin{equation}\label{eq:f^+_0=f0}
f^{+,0}=f^{0} \qquad\text{ as a distribution on }\mathbb{T}^N\times\mathbb{R}.
\end{equation}
Let $q^0$ be the restriction of $q$ to $\mathbb{T}^N\times\{0\}\times\mathbb{R}$ and $\nu^0$ be the restriction of $\nu$ to $\mathbb{T}^N\times\{0\}$.

Let $\Upsilon_\varepsilon\in C_c^1([0,T))$ fulfill
\[
0\leq\Upsilon_\varepsilon\leq1,\qquad\Upsilon_\varepsilon(0)=1,\qquad\Upsilon_\varepsilon(r)=0\quad\text{when }r\geq\varepsilon,
\]
and $\Upsilon_\varepsilon(r)$ converges pointwise to $\one_{\{0\}}(r)$ as $\varepsilon\to0$ on $[0,T)$.

Taking the test function  $\varphi(x,r,\xi)=\Upsilon_\varepsilon(r)\phi(x,\xi)$ for $\phi\in C_c^1(\mathbb{T}^N\times\mathbb{R})$ in \eqref{eq:kinetic eqution},
passing to the limit as  $\varepsilon\to0$ and using Burkholder-Davis-Gundy inequality together with the dominated convergence theorem, we have, almost surely
\begin{align}\notag
&-\langle f^{+,0},  \phi\rangle +\langle f^0, \phi\rangle =  \int_{\mathbb{T}^N\times\mathbb{R}}\partial_{\xi} \phi(x,\xi)\dd q^0(x,\xi)-\int_{\mathbb{T}^N}\phi(x,\psi(x,0))\dd\nu^0(x).
\end{align}
In view of \eqref{eq:f^+_0=f0-1}, we have almost surely,
\begin{equation}\label{eq:equal for qnu}
\int_{\mathbb{T}^N\times\mathbb{R}}\partial_{\xi} \phi(x,\xi)\dd q^0(x,\xi)=\int_{\mathbb{T}^N}\phi(x,\psi(x,0))\dd\nu^0(x).
\end{equation}
We first prove that $\nu^0(\mathbb{T}^N)=0$ almost surely.
For each $R>0$, let $\theta_R\in C_c^1(\mathbb{R})$ be a function such that
\[
0\leq\theta_R\leq1,\qquad\theta_R(\xi)=1\quad\text{for }|\xi|\leq R,\qquad\theta_R(\xi)=0\quad\text{for }|\xi|\geq2R,\qquad|\theta^\prime_R|\leq \frac{C}{R}.
\]
Since $\psi\in C(\mathbb{T}^N\times[0,T])$, we have $\theta_R(\psi(x,0))\equiv1$ for all $x\in\mathbb{T}^N$ when $R$ is large enough.
Taking the test function in \eqref{eq:equal for qnu} as $\phi(x,\xi)=\theta_R(\xi)$, we have almost surely
\begin{equation*}
\nu^0(\mathbb{T}^N)=\int_{\mathbb{T}^N\times\mathbb{R}}\theta^\prime_R(\xi)\dd q^0(x,\xi)\leq \frac{C}{R}\, q^0(\mathbb{T}^N\times\{R\leq|\xi|\leq2R\}).
\end{equation*}
Taking the limit $R\to\infty$ and using \eqref{eq:vanishing kinetic measure}, we have almost surely $\nu^0(\mathbb{T}^N)=0$.

Now, we prove that $q^0(\mathbb{T}^N\times\mathbb{R})=0$. Combining \eqref{eq:equal for qnu} and the nonnegativity of $\nu$, we have almost surely
\begin{equation}\label{eq:testq_0=0}
\int_{\mathbb{T}^N\times\mathbb{R}}\partial_{\xi} \phi(x,\xi)\dd q^0(x,\xi)=0,\qquad\phi\in C_c^1(\mathbb{T}^N\times\mathbb{R}).
\end{equation}
For a function $\Phi\in C_c^1(\mathbb{T}^N\times\mathbb{R})$, define
\[
c_\Phi(x):=\int_\mathbb{R}\Phi(x,\xi)\dd\xi.
\]
For each $R>0$, choose $\tilde{\theta}_R\in C_c^1(\mathbb{R})$ such that
\[
\int_{\mathbb{R}}\tilde{\theta}_R(\xi)\dd\xi=1,\qquad\text{supp}\,\tilde{\theta}_R\subset\{R\leq|\xi|\leq2R\},\qquad\Vert\tilde{\theta}_R\Vert_{L^\infty(\mathbb{R})}\leq\frac{C}{R}.
\]
Define
\[
\Phi_R(x,\xi):=\Phi(x,\xi)-c_{\Phi}(x)\tilde{\theta}_R(\xi).
\]
Since $\int_{\mathbb{R}}\Phi_{R}(x,\xi)\dd\xi=0$ for every $x\in\mathbb{T}^N$, the function
\[
\phi_R(x,\xi):=\int_{-\infty}^\xi\Phi_R(x,\zeta)\dd\zeta
\]
belongs to $C_c^1(\mathbb{T}^N\times\mathbb{R})$. Taking $\phi=\phi_R$  in \eqref{eq:testq_0=0}, we have
\[
0=\int_{\mathbb{T}^N\times\mathbb{R}}\Phi_R(x,\xi)\dd q^0(x,\xi)=\int_{\mathbb{T}^N\times{\mathbb{R}}}\Phi(x,\xi)\dd q^0(x,\xi)-\int_{\mathbb{T}^N\times\mathbb{R}}c_\Phi(x)\tilde{\theta}_R(\xi)\dd q^0(x,\xi).
\]
Therefore, based on the definition of $\tilde{\theta}_R$, we have
\[
\bigg|\int_{\mathbb{T}^N\times{\mathbb{R}}}\Phi(x,\xi)\dd q^0(x,\xi)\bigg|\leq\frac{C}{R}\Vert{c_\Phi}\Vert_{L^\infty(\mathbb{T}^N)}q^0(\mathbb{T}^N\times\{R\leq|\xi|\leq2R\}).
\]
Taking the limit $R\to\infty$ and using \eqref{eq:vanishing kinetic measure}, we have
\[
\int_{\mathbb{T}^N\times{\mathbb{R}}}\Phi(x,\xi)\dd q^0(x,\xi)=0.
\]
By the arbitrariness of $\Phi\in C_c^1(\mathbb{T}^N\times{\mathbb{R}})$, we have $q^0=0$, which completes the proof.
\end{proof}

In view of Proposition \ref{prop:trace-obstacle-scl} and Lemma \ref{lem:no atom of q nu}, we can derive a kinetic formulation at a fixed time $t$, which is weak only in $(x,\xi)$. 
Specifically, fix $t\in(0,T)$ and $\bar\epsilon\in(0,(T-t)\wedge t)$. 
Test the kinetic formulations \eqref{eq:kinetic eqution} of $f$ with functions converging to $\Upsilon_{\bar\epsilon,t}(r)\varphi(x,\xi)$, where $\varphi\in C_c^1(\mathbb{T}^N\times\mathbb{R})$, and $\Upsilon_{\bar\epsilon,t}$ is defined by
\begin{equation*}
\Upsilon_{\bar\epsilon,t}(r):=\begin{cases}
1,&\bar\epsilon\leq r\leq t,\\
r/\bar\epsilon,&0\leq r\leq \bar\epsilon,\\
1-(r-t)/\bar\epsilon,&t\leq r\leq t+\bar\epsilon,\\
0,&r\geq t+\bar\epsilon\text{ or } r\leq 0.
\end{cases}
\end{equation*}
Then, taking the limit $\bar\epsilon\to0$ and using Proposition \ref{prop:trace-obstacle-scl}, Remark \ref{rem:remark for intergral of f^+}, \eqref{eq:f^+_0=f0}, Lemma \ref{lem:no atom of q nu}, Burkholder-Davis-Gundy inequality and the dominated convergence theorem, we get for $\varphi\in C_c^1(\mathbb{T}^N\times\mathbb{R})$,
\begin{equation}
\begin{aligned}
&\langle f^+(t),\varphi\rangle -\langle f^0,\varphi\rangle
 -\int_0^t\!\langle f(r),a(\xi)\cdot\nabla_x\varphi\rangle\dd r \\
&= \sum_{k\geq1}\int_0^t\int_{\mathbb{T}^N}\int_{\mathbb{R}}g_k(x,\xi)\varphi(x,\xi)\dd \mu_{x,r}(\xi)\dd x\dd\beta_k(r)\\
&\quad +\frac12\int_0^t\int_{\mathbb{T}^N}\int_{\mathbb{R}}G^2(x,\xi)\partial_\xi\varphi(x,\xi)\dd \mu_{x,r}(\xi)\dd x\dd r\\
&\quad -\int_{\mathbb{T}^N\times[0,t]\times{\mathbb{R}}}\partial_{\xi}\varphi(x,\xi)\dd q(x,r,\xi)
+\int_{\mathbb{T}^N\times[0,t]}\varphi(x,\psi(x,r))\dd\nu(x,r).
\end{aligned}\label{eq:weak-kinetic on [0,t]}
\end{equation}
By slightly modifying the definition of $\Upsilon_{\bar\epsilon,t}$, we obtain the equation of $f^-$ with $f^-(0):=\one_{u^0>\xi}$.

\section{Uniqueness of kinetic solutions to the obstacle problem}

In this section, we employ the doubling variables method to show the uniqueness.

Let $(\rho_{\gamma})$ and $(\kappa_\delta)$ be standard  nonnegative, even mollifier sequences on $\mathbb{T}^N$ and $\mathbb{R}$, respectively.
 For each $R>0$, let $\theta_R\in C_c^\infty(\mathbb{R})$ be a nonnegative localized function such that
\begin{align}\label{theta}
0\leq\theta_R\leq1,\qquad\theta_R(\xi)=1\quad\text{for }|\xi|\leq R,\qquad\theta_R(\xi)=0\quad\text{for }|\xi|\geq R+1,\qquad|\theta^\prime_R|\leq C.
\end{align}

\begin{thm}[$L^1$-stability and uniqueness]\label{thm:uniqueness}
Let \textbf{Hypotheses (H1) and (H2)$'$} hold.
Let $(u_1,\nu_1)$ and $(u_2,\nu_2)$ be two kinetic solutions of the lower obstacle problem for \eqref{eq:intro-reflected-model} in the sense of Definition \ref{def:dfn-1},
with corresponding initial values satisfying
\[
u^0_{1},u^0_{2}\in L^\infty(\mathbb{T}^N),
\qquad
u^0_{i}\geq \psi(\cdot,0)\quad\text{a.e. on }\mathbb{T}^N.
\]
Then, for almost every $t\in[0,T]$,
\begin{equation}\label{eq:l1-contraction}
\E\Vert u_1(t)-u_2(t)\Vert _{L^1(\mathbb{T}^N)}
\leq
\Vert u^0_{1}-u^0_{2}\Vert _{L^1(\mathbb{T}^N)}.
\end{equation}
In particular, if $u^0_{1}=u^0_{2}$ a.e. on $\mathbb{T}^N$, then
\[
u_1=u_2
\quad\text{in }L^1(\Omega\times\mathbb{T}^N\times[0,T]),
\]
and the reflection measures coincide:
\[
\nu_1=\nu_2
\quad\text{as Radon measures on }\mathbb{T}^N\times[0,T),
\quad \mathbb{P}\text{-a.s.}
\]
Thus, the kinetic solution pair $(u,\nu)$ is unique.
\end{thm}
\begin{proof}
For $i=1,2$, set
\[
f_i(x,t,\xi)=\one_{u_i(x,t)>\xi},
\qquad
\bar f_i=1-f_i,\qquad \mu^i_{x,t}(\dd\xi):=\delta_{u_i(x,t)}(\dd\xi),
\]
and $f_i(0,x,\xi)=\one_{\{u^0_{i}(x)>\xi\}}$. 
The corresponding kinetic and Radon measures are denoted by $(q_i, \nu_i)$, $i=1,2$.

By Proposition \ref{prop:trace-obstacle-scl}, each $f_i$ admits left and right representatives $f_i^-$ and $f_i^+$ at every $t\in[0,T]$, and the set of jump times is at most countable.
Proceeding similarly to \eqref{eq:weak-kinetic on [0,t]}, with test function  $\varphi_1(x,\xi)$, we have
\begin{equation*}
\begin{aligned}
&\langle f_1^+(t),\varphi_1\rangle -\langle f_1(0),\varphi_1\rangle
 -\int_0^t\!\langle f_1(r),a(\xi)\cdot\nabla_x\varphi_1\rangle\dd r \\
&= \sum_{k\geq1}\int_0^t\int_{\mathbb{T}^N}g_k(x,u_1(x,r))\varphi_1(x,u_1(x,r))\dd x\dd\beta_k(r)\\
&\quad +\frac12\int_0^t\int_{\mathbb{T}^N}G^2(x,u_1(x,r))\partial_\xi\varphi_1(x,u_1(x,r))\dd x\dd r\\
&\quad -\int_{\mathbb{T}^N\times[0,t]\times{\mathbb{R}}}\partial_{\xi}\varphi_1(x,\xi)\dd q_1(x,r,\xi)
+\int_{\mathbb{T}^N\times[0,t]}\varphi_1(x,\psi(x,r))\dd\nu_1(x,r).
\end{aligned}
\end{equation*}
Similarly, for $\bar{f}_2(y,t,\zeta)$, with the test function $\varphi_2(y,\zeta)$, it follows that
\begin{equation*}
\begin{aligned}
&\langle \bar{f}_2^+(t),\varphi_2\rangle -\langle \bar{f}_2(0),\varphi_2\rangle
 -\int_0^t\!\langle \bar{f}_2(r),a(\zeta)\cdot\nabla_y\varphi_2\rangle\dd r \\
&= -\sum_{k\geq1}\int_0^t\int_{\mathbb{T}^N}g_k(y,u_2(y,r))\varphi_2(y,u_2(y,r))\dd y\dd\beta_k(r)\\
&\quad -\frac12\int_0^t\int_{\mathbb{T}^N}G^2(y,u_2(y,r))\partial_\zeta\varphi_2(y,u_2(y,r))\dd y\dd r\\
&\quad +\int_{\mathbb{T}^N\times[0,t]\times{\mathbb{R}}}\partial_{\zeta}\varphi_2(y,\zeta)\dd q_2(y,r,\zeta)
-\int_{\mathbb{T}^N\times[0,t]}\varphi_2(y,\psi(y,r))\dd\nu_2(y,r).
\end{aligned}
\end{equation*}
Then $\langle f^+_1(t),\varphi_1\rangle$ and $\langle\bar{f}_2^+(t),\varphi_2\rangle$ are c\`{a}dl\`{a}g semimartingales.
Their martingale parts are continuous, while the finite variation parts may have jumps caused by $q_i$ and $\nu_i$.
Let $\varphi_{\gamma,\delta}(x,\xi,y,\zeta):=\rho_{\gamma}(x-y)\kappa_{\delta}(\xi-\zeta)$. Using It\^o's product formula and the integration-by-parts formula for functions of finite variation, with the tensorization argument of \cite[proof of Proposition 9]{DV10-publish}, we reach
\begin{align*}
&\E\int_{(\mathbb{T}^N)^2}\int_{\mathbb{R}^2} f_1^+(t)\bar{f}_2^+(t)\rho_{\gamma}(x-y)\kappa_\delta(\xi-\zeta)\dd x\dd\xi\dd y \dd\zeta \\
&=\E\int_{(\mathbb{T}^N)^2}\int_{\mathbb{R}^2}f_1(0)\bar{f}_2(0)\rho_{\gamma}(x-y)\kappa_\delta(\xi-\zeta)\dd x\dd\xi\dd y \dd\zeta \nonumber\\
&\quad+\E\int_0^t\int_{(\mathbb{T}^N)^2}\int_{\mathbb{R}^2}f_1(x,r,\xi)\bar{f}_2(y,r,\zeta)\big(a(\xi)\cdot\nabla_x+a(\zeta)\cdot\nabla_y\big)\varphi_{\gamma,\delta}\, \dd\xi\dd\zeta\dd x\dd y\dd r\nonumber\\
&\quad+\bigg\{\frac{1}{2}\E\int_0^t\int_{(\mathbb{T}^N)^2}\int_{\mathbb{R}^2}G^2(x,\xi)\partial_\xi\varphi_{\gamma,\delta}(x,\xi,y,\zeta)\bar{f}_2(y,r,\zeta)\, \dd\zeta\dd \mu^1_{x,r}(\xi)\dd x\dd y\dd r\nonumber\\
&\quad-\frac{1}{2}\E\int_0^t\int_{(\mathbb{T}^N)^2}\int_{\mathbb{R}^2}G^2(y,\zeta)\partial_\zeta\varphi_{\gamma,\delta}(x,\xi,y,\zeta)f_1(x,r,\xi)\,\dd\xi\dd \mu^2_{y,r}(\zeta)\dd x\dd y\dd r\nonumber\\
&\quad-\E\int_0^t\int_{(\mathbb{T}^N)^2}\int_{\mathbb{R}^2}G_{1,2}(x,y,\xi,\zeta)\varphi_{\gamma,\delta}(x,\xi,y,\zeta)\dd x\dd y\dd \mu^1_{x,r}(\xi)\dd \mu^2_{y,r}(\zeta)\dd r\bigg\}\nonumber\\
&\quad+\bigg\{\E\int_{\mathbb{T}^N}\int_{\mathbb{R}}\int_{\mathbb{T}^N\times[0,t]\times\mathbb{R}}\partial_\zeta\varphi_{\gamma,\delta}(x,\xi,y,\zeta){f}_1^-(x,r,\xi)\dd\xi\dd q_2(y,r,\zeta)\dd x\nonumber\\
&\quad-\E\int_{\mathbb{T}^N}\int_{\mathbb{R}}\int_{\mathbb{T}^N\times[0,t]\times\mathbb{R}}\partial_\xi\varphi_{\gamma,\delta}(x,\xi,y,\zeta)\bar{f}_2^+(y,r,\zeta)\dd q_1(x,r,\xi)\dd\zeta\dd y\bigg\}\nonumber\\
&\quad+\bigg\{\E\int_{\mathbb{T}^N}\int_{\mathbb{R}}\int_{\mathbb{T}^N\times[0,t]}\varphi_{\gamma,\delta}(x,\psi(x,r),y,\zeta)\bar{f}_2^+(y,r,\zeta)\dd \nu_1(x,r)\dd y\dd \zeta\nonumber\\
&\quad-\E\int_{\mathbb{T}^N}\int_{\mathbb{R}}\int_{\mathbb{T}^N\times[0,t]}\varphi_{\gamma,\delta}(x,\xi,y,\psi(y,r))f_1^-(x,r,\xi)\dd \nu_2(y,r)\dd x\dd \xi\bigg\}\nonumber\\
&=: \E\int_{(\mathbb{T}^N)^2}\int_{\mathbb{R}^2}f_1(0)\bar{f}_2(0)\rho_{\gamma}(x-y)\kappa_\delta(\xi-\zeta)\dd x\dd\xi\dd y\dd\zeta+I_{\gamma,\delta}^{\rm{flux}}(t)+I_{\gamma,\delta}^{\rm{quad}}(t)+I_{\gamma,\delta}^{q}(t)+I_{\gamma,\delta}^{\nu}(t)\nonumber,
\end{align*}
where $G_{1,2}(x,y,\xi,\zeta):=\sum_{k\geq1}g_k(x,\xi)g_k(y,\zeta)$.
Due to
\begin{equation*}
(\nabla_x+\nabla_y)\varphi_{\gamma,\delta}=0,\qquad(\partial_\xi+\partial_\zeta)\varphi_{\gamma,\delta}=0,
\end{equation*}
we have
\begin{align*}
I_{\gamma,\delta}^{\rm{flux}}(t)=\E\int_0^t\int_{(\mathbb{T}^N)^2}\int_{\mathbb{R}^2}f_1(x,r,\xi)\bar{f}_2(y,r,\zeta)\big(a(\xi)-a(\zeta)\big)\cdot\nabla_x\rho_{\gamma}(x-y)\kappa_{\delta}(\xi-\zeta)\dd x\dd\xi\dd y\dd\zeta\dd r.
\end{align*}
By integration by parts w.r.t. $\xi$ and $\zeta$, using $\partial_{\zeta}\bar{f}_2(y,r,\zeta)\dd\zeta=\mu^2_{y,r}(\dd\zeta)=\delta_{u_2(y,r)}(\dd\zeta)$ and $-\partial_{\xi}f_1(x,r,\xi)\dd\xi=\mu^1_{x,r}(\dd\xi)=\delta_{u_1(x,r)}(\dd\xi)$, it follows that
\begin{align}
I_{\gamma,\delta}^{\rm{quad}}(t)=\frac{1}{2}\sum_{k\geq 1}\E\int_0^t\int_{(\mathbb{T}^N)^2}\int_{\mathbb{R}^2}
\rho_{\gamma}(x-y)\kappa_{\delta}(\xi-\zeta)|g_k(x,\xi)-g_k(y,\zeta)|^2\dd \mu^1_{x,r}\otimes\mu^2_{y,r}(\xi,\zeta)\dd x \dd y \dd r.\nonumber
\end{align}
Note that the terms $I_{\gamma,\delta}^{\rm{flux}} $ and $I_{\gamma,\delta}^{\rm{quad}}$ are the same as those in \cite[Proposition 13]{DV10}, so we deduce from the results in \cite[(31) and (35)]{DV10} that
\begin{align}\label{eq:estimates for Ia-1}
|I_{\gamma,\delta}^{\rm{flux}}(t)|&\leq
C\delta \E \int^t_0\int_{(\mathbb{T}^N)^2}|\nabla_x \rho_{\gamma}(x-y)|(1+|u_1(x,r)|^p+|u_2(y,r)|^p)\dd x\dd y\dd r\\ \notag
&\leq C\delta \gamma^{-1}\E\int^t_0(1+\|u_1\|^p_{L^p(\mathbb{T}^N)}+\|u_2\|^p_{L^p(\mathbb{T}^N)})\dd r\\
&\leq C_pT\delta\gamma^{-1},\notag
\end{align}
and
\begin{align}\label{eq:estimates for Ia-2}
|I_{\gamma,\delta}^{\rm{quad}}(t)|\leq D_0T\gamma^2\delta^{-1}+D_0T\delta.
\end{align}

In the sequel, we focus on the terms $I_{\gamma,\delta}^{\nu}(t)$ and $I_{\gamma,\delta}^{q}(t)$, associated with the reflection measure $\nu$ and the kinetic measure $q$, respectively.

Let us start with $I_{\gamma,\delta}^\nu(t)$.
Define $f_\psi(x,t,\xi):=\one_{\psi(x,t)>\xi}$.
From the obstacle constraint in Definition \ref{def:dfn-1}, we have $f_{i}\geq f_\psi$ for $\mathbb P\otimes\dd x\otimes\dd t\otimes\dd\xi$-a.e.
$(\omega,x,t,\xi)$, where $i\in\{1,2\}$.
Fix $i\in\{1,2\}$.
Note that $f^{+}_{i}=f^{-}_{i}=f_{i}$ for $\mathbb P\otimes\dd t$-a.e. $(\omega,t)$, as distributions on
$\mathbb T^N\times\mathbb R$.
Then, for all $\tau\in(0,T)$, there exist sequences $(\tau_k^{-})_{k\in{\mathbb{N}}}$ and $(\tau_k^{+})_{k\in{\mathbb{N}}}$ of positive numbers satisfying $\tau_k^{-}<\tau$, $\tau_k^{+}<T-\tau$ for every $k\in\mathbb{N}$, and $\tau_k^{-},\tau_k^{+}\to0$ as $k\to\infty$ such that $f^{\pm}_{i}(\cdot,\tau\pm \tau_k^{\pm},\cdot)=f_{i}(\cdot,\tau\pm \tau_k^{\pm},\cdot)$ as a distribution on $\mathbb{T}^N\times \mathbb{R}$ for every $k\in\mathbb{N}$.
Owing to the time-continuity of $\psi$ and the definition of $f_i^{+}$ and using Proposition \ref{prop:trace-obstacle-scl}, we obtain almost surely for all nonnegative $\phi\in C_c^\infty(\mathbb{T}^N\times \mathbb R)$,
\begin{align}\label{eq:fi geq fpsi}
&\int_{\mathbb{T}^N}\int_{\mathbb{R}}(f_{i}^{\pm}(x,\tau,\xi)-f_{\psi}(x,\tau,\xi))\phi(x,\xi)\dd x\dd\xi\\
&\geq \liminf_{k\to\infty}\int_{\mathbb{T}^N}\int_{\mathbb{R}}(f_{i}(x,\tau\pm \tau_k^{\pm},\xi)-f_{\psi}(x,\tau\pm \tau_k^{\pm},\xi))\phi(x,\xi)\dd x\dd\xi\geq 0.\nonumber
\end{align}
Using a density argument and combining the fact that $f^+_i=f_i=f^-_i$ when $t=0$, we have $f^{\pm}_{i}\geq f_\psi$ almost surely on $\Omega$ for all $t\in[0,T)$ in $L^\infty(\mathbb{T}^N\times\mathbb{R})$.
Therefore, we have
\begin{align}\label{eq:I obstacle with psi}
I_{\gamma,\delta}^\nu(t)
&\leq\E\int_{\mathbb{T}^N}\int_{\mathbb{R}}\int_{\mathbb{T}^N\times[0,t]}\varphi_{\gamma,\delta}(x,\psi(x,r),y,\zeta)\big(1-{f}_\psi(y,r,\zeta)\big)\dd \nu_1(x,r)\dd y\dd \zeta\\
&\quad-\E\int_{\mathbb{T}^N}\int_{\mathbb{R}}\int_{\mathbb{T}^N\times[0,t]}\varphi_{\gamma,\delta}(x,\xi,y,\psi(y,r))f_\psi(x,r,\xi)\dd \nu_2(y,r)\dd x\dd \xi\nonumber\\
&=\E\int_{\mathbb{T}^N}\int_{\mathbb{R}}\int_{\mathbb{T}^N\times[0,t]}\rho_{\gamma}(x-y)\kappa_\delta(\psi(x,r)-\zeta)\big(1-{f}_\psi(y,r,\zeta)\big)\dd \nu_1(x,r)\dd y\dd \zeta\nonumber\\
&\quad-\E\int_{\mathbb{T}^N}\int_{\mathbb{R}}\int_{\mathbb{T}^N\times[0,t]}\rho_{\gamma}(x-y)\kappa_{\delta}(\xi-\psi(y,r))f_\psi(x,r,\xi)\dd \nu_2(y,r)\dd x\dd \xi.\nonumber
\end{align}

For any $z\in \mathbb{R}$, define
\begin{align*}
\llbracket g\rrbracket(z):=\int_{0}^{z}g(s)\dd s,\qquad\text{for each }g\in L^1(\mathbb{R}).
\end{align*}
Then $\llbracket\kappa_\delta\rrbracket$ is a non-decreasing odd function with
\[
\kappa_\delta(\xi-\zeta)=\partial_\xi\llbracket\kappa_\delta\rrbracket(\xi-\zeta),\qquad\llbracket\kappa_\delta\rrbracket(-\infty)=-\frac{1}{2},\qquad\llbracket\kappa_\delta\rrbracket(+\infty)=\frac{1}{2},\qquad\Vert\kappa_\delta\Vert_{L^\infty(\mathbb{R})}\leq\frac{C}{\delta}.
\]
Therefore, we have
\begin{align*}
\int_{\mathbb R}\kappa_\delta(\psi(x,r)-\zeta)\big(1-{f}_\psi(y,r,\zeta)\big)\dd \zeta&=\int_{\mathbb R}-\partial_\zeta\llbracket\kappa_\delta\rrbracket(\psi(x,r)-\zeta)\big(1-{f}_\psi(y,r,\zeta)\big)\dd \zeta\\
&=\frac{1}{2}+\llbracket\kappa_\delta\rrbracket(\psi(x,r)-\psi(y,r)),
\end{align*}
and
\begin{align*}
\int_{\mathbb R}\kappa_\delta(\xi-\psi(y,r)){f}_\psi(x,r,\xi)\dd \xi&=\int_{\mathbb R}\partial_\xi\llbracket\kappa_\delta\rrbracket(\xi-\psi(y,r)){f}_\psi(x,r,\xi)\dd \xi\\
&=\frac{1}{2}+\llbracket\kappa_\delta\rrbracket(\psi(x,r)-\psi(y,r)),
\end{align*}
Combining the definition of $f_\psi$, \eqref{eq:I obstacle with psi} can be estimated as
\begin{align*}
I_{\gamma,\delta}^\nu(t)&\leq\E\int_{\mathbb{T}^N}\int_{\mathbb{T}^N\times[0,t]}\rho_{\gamma}(x-y)\llbracket\kappa_\delta\rrbracket(\psi(x,r)-\psi(y,r))\dd \nu_1(x,r)\dd y\nonumber\\
&\quad-\E\int_{\mathbb{T}^N}\int_{\mathbb{T}^N\times[0,t]}\rho_{\gamma}(x-y)\llbracket\kappa_\delta\rrbracket(\psi(x,r)-\psi(y,r))\dd \nu_2(y,r)\dd x\nonumber\\
&\quad +\frac{1}{2}\big[\E\nu_1(\mathbb{T}^N\times[0,t])-\E\nu_2(\mathbb{T}^N\times[0,t])\big].
\end{align*}
Under the assumption that $\psi\in C([0,T];C^2(\mathbb{T}^N))$, for fixed $(x,r)$, write $y=x-h$. 
Then, Taylor's formula gives
\begin{equation}\label{eq:psi-Taylor}
\psi(x,r)-\psi(x-h,r)=\nabla_x\psi(x,r)\cdot h+\mathcal R_\psi(x,h,r),\qquad|\mathcal R_\psi(x,h,r)|\leq C_\psi|h|^2.
\end{equation}
Since $\rho_\gamma$ is even and $\llbracket\kappa_\delta\rrbracket$ is odd, we have
\begin{equation*}
\int_{\mathbb T^N}\rho_\gamma(h)\llbracket\kappa_\delta\rrbracket(\nabla\psi(x,r)\cdot h)\dd h=0.
\end{equation*}
In view of $\Vert\kappa_\delta\Vert_{L^\infty(\mathbb{R})}\leq{C}/{\delta}$ and \eqref{eq:psi-Taylor}, we obtain
\begin{align*}
&\bigg|\int_{\mathbb T^N}\rho_\gamma(x-y)\llbracket\kappa_\delta\rrbracket\big(\psi(x,r)-\psi(y,r)\big)\dd y
\bigg|\leq\frac{C}{\delta}\int_{\mathbb T^N}\rho_\gamma(h)|h|^2\dd h\leq C\frac{\gamma^2}{\delta}.
\end{align*}
The same estimate holds with $x$ and $y$ interchanged. Hence
\begin{equation}\label{eq:I-nu-final-estimate}
I_{\gamma,\delta}^{\nu}(t)\leq C\frac{\gamma^2}{\delta}\E\Bigl[\nu_1(\mathbb T^N\times[0,t])+\nu_2(\mathbb T^N\times[0,t])\Bigr]+\frac{1}{2}\big[\E\nu_1(\mathbb{T}^N\times[0,t])-\E\nu_2(\mathbb{T}^N\times[0,t])\big].
\end{equation}

For the term $I^q_{\gamma,\delta}$, we first prove that $\partial_\xi{f}^{\pm}_i\leq 0$ in the sense of distribution.
Analogously to the proof of \eqref{eq:fi geq fpsi}, we will use the fact that almost surely $f_{i}^{\pm}(x,\tau,\xi)=\lim_{k\to\infty} f_{i}(x,\tau\pm \tau_k^{\pm},\xi)$ as a distribution on $\mathbb{T}^N\times\mathbb{R}$.
Since $f_i\in \{0,1\}$ and $\partial_\xi f_i\leq0$ on $\Omega\times\mathbb{T}^N\times[0,T]\times\mathbb{R}$, we have for all nonnegative functions $\varphi\in C_c^\infty(\mathbb R)$ and $\phi\in C^\infty(\mathbb{T}^N)$ and $\tau\in(0,T)$, almost surely
\begin{align*}
\int_{\mathbb{T}^N}\int_{\mathbb{R}}f_{i}^{\pm}(x,\tau,\xi)\phi(x)\varphi^\prime(\xi)\dd x \dd \xi
&=\lim_{k\to\infty}\int_{\mathbb{T}^N}\int_{\mathbb{R}}f_{i}(x,\tau\pm \tau_k^{\pm},\xi)\phi(x)\varphi^\prime(\xi)\dd x \dd \xi\\
&=\lim_{k\to\infty}\int_{\mathbb{T}^N}\int_{\mathbb{R}}\big(-\partial_{\xi}f_{i}(x,\tau\pm \tau_k^{\pm},\xi)\big)\phi(x)\varphi(\xi)\dd x \dd \xi\geq 0.
\end{align*}
Using a density argument and combining the fact that $f^+_i=f_i=f^-_i$ when $t=0$, we have for all $\tau\in[0,T)$ almost surely
\begin{align*}
\partial_\xi f^{\pm}_i\leq0,\quad\text{as a distribution on }\mathbb{T}^N\times\mathbb{R}.
\end{align*}
Therefore, we have
\begin{align}\label{eq:estimates for Iqvar}
I_{\gamma,\delta}^q(t)&=\E\int_{\mathbb{T}^N}\int_{\mathbb{R}}\int_{\mathbb{T}^N\times[0,t]\times\mathbb{R}}\partial_\zeta\varphi_{\gamma,\delta}(x,\xi,y,\zeta)\bar{f}_2^+(y,r,\zeta)\dd q_1(x,r,\xi)\dd y\dd\zeta\\
&\quad-\E\int_{\mathbb{T}^N}\int_{\mathbb{R}}\int_{\mathbb{T}^N\times[0,t]\times\mathbb{R}}\partial_\xi\varphi_{\gamma,\delta}(x,\xi,y,\zeta){f}_1^-(x,r,\xi)\dd q_2(y,r,\zeta)\dd x\dd\xi\nonumber\\
&=-\E\int_{\mathbb{T}^N}\int_{\mathbb{R}}\int_{\mathbb{T}^N\times[0,t]\times\mathbb{R}}\varphi_{\gamma,\delta}(x,\xi,y,\zeta)\partial_\zeta\bar{f}_2^+(y,r,\zeta)\dd q_1(x,r,\xi)\dd y\dd\zeta\nonumber\\
&\quad+\E\int_{\mathbb{T}^N}\int_{\mathbb{R}}\int_{\mathbb{T}^N\times[0,t]\times\mathbb{R}}\varphi_{\gamma,\delta}(x,\xi,y,\zeta)\partial_\xi{f}_1^-(x,r,\xi)\dd q_2(y,r,\zeta)\dd x\dd\xi\leq0.\nonumber
\end{align}

Choose $\delta=\gamma^{4/3}$. Then as $\gamma$ tends to 0, it follows that
\[
\frac{\delta}{\gamma}=\gamma^{1/3}\rightarrow0,\qquad\frac{\gamma^2}{\delta}
=\gamma^{2/3}\rightarrow0,\qquad\delta=\gamma^{4/3}\rightarrow0.
\]
Combining \eqref{eq:estimates for Ia-1}, \eqref{eq:estimates for Ia-2}, \eqref{eq:I-nu-final-estimate} and \eqref{eq:estimates for Iqvar}, we have
\begin{align}
&\E \int_{\mathbb{T}^N}\int_{\mathbb{R}}f_1^+(t)\bar{f}_2^+(t)\dd x \dd\xi-\int_{\mathbb{T}^N}\int_{\mathbb{R}} f_1(0)\bar{f}_2(0)\dd x \dd\xi\nonumber\\
&=\limsup_{\gamma\downarrow0}\Big[\E\int_{(\mathbb{T}^N)^2}\int_{\mathbb{R}^2} \big(f_1^+(t)\bar{f}_2^+(t)-f_1(0)\bar{f}_2(0)\big)\varphi_{\gamma,\delta}\dd x\dd\xi\dd y\dd\zeta\Big]\nonumber\\
&=
\limsup_{\gamma\downarrow0}\Big[I_{\gamma,\delta}^\nu(t)+I_{\gamma,\delta}^q(t)+
I_{\gamma,\delta}^{\rm{flux}}(t)+I_{\gamma,\delta}^{\rm{quad}}(t)\Big]\nonumber\\
&\leq\frac{1}{2}\big[\E\nu_1(\mathbb{T}^N\times[0,t])-\E\nu_2(\mathbb{T}^N\times[0,t])\big].\nonumber
\end{align}
Using Proposition \ref{prop:trace-obstacle-scl} and the joint measurability of $f^+_i$ and $f_i$, we have
\[
f^+_i=f_i\qquad\text{in }\mathcal{D}'(\mathbb{T}^N\times\mathbb{R}),\qquad\text{for }\mathbb{P}\otimes\dd t\text{-a.s.}
\]
Hence, by Fubini's theorem, there exists a set $T_0\subset(0,T)$ of full Lebesgue measure such that, for every $t\in T_0$,
\begin{equation}\label{eq:equal-for-f+f-aet}
f^+_i(t)=f_i(t),\qquad i=1,2,\qquad\mathbb{P}\text{-a.s.}
\end{equation}
Consequently, for every $t\in T_0$,
\[
\E\int_{\mathbb{T}^N\times\mathbb{R}}f^+_1(t)\bar{f}^+_2(t)\dd x\dd \xi=\E\Vert(u_1(t)-u_2(t))^+\Vert_{L^1(\mathbb{T}^N)}.
\]
Then, we have
\begin{align*}
&\E\Vert (u_1(t)-u_2(t))^+\Vert_{L^1(\mathbb{T}^N)}\\
&\leq\Vert (u^0_{1}-u^0_{2})^+\Vert_{L^1(\mathbb{T}^N)}+\frac{1}{2}\big[\E\nu_1(\mathbb{T}^N\times[0,t])-\E\nu_2(\mathbb{T}^N\times[0,t])\big],\qquad\text{a.e. }t\in(0,T).
\end{align*}
Similarly, we obtain
\begin{align*}
&\E\Vert (u_1(t)-u_2(t))^-\Vert_{L^1(\mathbb{T}^N)}\\
&\leq\Vert (u^0_{1}-u^0_{2})^-\Vert_{L^1(\mathbb{T}^N)}+\frac{1}{2}\big[\E\nu_2(\mathbb{T}^N\times[0,t])-\E\nu_1(\mathbb{T}^N\times[0,t])\big],\qquad\text{a.e. }t\in(0,T).
\end{align*}
Then \eqref{eq:l1-contraction} follows by adding the above two inequalities.

If $u^0_{1}=u^0_{2}$, we infer from \eqref{eq:l1-contraction} that $u_1=u_2$ in $L^1(\Omega\times\mathbb{T}^N)$ for almost every $t\in[0,T]$. This further yields that
 $\nu_1$ and $\nu_2$ coincide.
Indeed, we deduce from \eqref{eq:kinetic eqution} for $f_i$, $i=1,2$, that almost surely for $\varphi\in C_c^1(\mathbb{T}^N\times[0,T)\times\mathbb{R})$, 
\begin{align*}
&\int_0^T\int_{\mathbb{T}^N}\int_{\mathbb{R}}\partial_\xi\varphi(x,t,\xi)\dd q_1(x,t,\xi)-\int_0^T\int_{\mathbb{T}^N}\varphi(x,t,\psi(x,t))\dd\nu_1(x,t)\\
&=\int_0^T\int_{\mathbb{T}^N}\int_{\mathbb{R}}\partial_\xi\varphi(x,t,\xi)\dd q_2(x,t,\xi)-\int_0^T\int_{\mathbb{T}^N}\varphi(x,t,\psi(x,t))\dd\nu_2(x,t).
\end{align*}
Let $\varphi(x,t,\xi):=\theta_R(\xi)\phi(x,t)$, where $\phi\in C_c^1(\mathbb{T}^N\times[0,T))$ and $\theta_R\in C_c^\infty(\mathbb{R})$ are given by \eqref{theta}.
Using \eqref{eq:vanishing kinetic measure} and taking the limit $R\to\infty$, we have almost surely
\begin{align*}
\int_0^T\int_{\mathbb{T}^N}\phi(x,t)\dd\nu_1(x,t)=\int_0^T\int_{\mathbb{T}^N}\phi(x,t)\dd\nu_2(x,t).
\end{align*}
This completes the proof with a density argument of $\phi$ in $C_c(\mathbb{T}^N\times[0,T))$.
\end{proof}

\section{Existence of kinetic solutions to the obstacle problem}\label{sec-app}
In this section, we aim to prove the existence of a kinetic solution to the lower obstacle problem
\begin{equation}\label{eq:obstacle-problem-existence}
\left\{
\begin{aligned}
\dd u+\Div A(u)\dd t
&=
\Phi(u)\dd W(t)+\dd\nu,&&\text{in }\mathbb T^N\times(0,T),\\
u&\geq\psi,
&&\text{a.e. in }\Omega\times\mathbb T^N\times(0,T),\\
u(0)&=u^0.
\end{aligned}
\right.
\end{equation}
To achieve it, we introduce two approximation equations associated with \eqref{eq:obstacle-problem-existence}.
For the first one, we regularize the coefficient $\Phi$ and the initial data $u^0$ as described below.
For notational simplicity, in this section the same parameter $\alpha$ is used for the viscosity and the regularization of the noise coefficient; both approximations are sent to zero simultaneously.
\begin{itemize}
  \item The sequence $\{\Phi_\alpha\}_{\alpha>0}$ is a family of smooth approximations of $\Phi$ satisfying
\begin{equation*}
\sup_{x\in\mathbb{T}^N,|\xi|\leq R}\sum_{k\geq1}|g_{k,\alpha}(x,\xi)-g_k(x,\xi)|^2\rightarrow 0\quad\text{for every }R>0,
\end{equation*}
where $
g_{k,\alpha}(\cdot,r):=\Phi_\alpha(r)e_k
$. The growth and continuity properties in \textbf{Hypotheses (H1)} hold for $\Phi_\alpha$ uniformly in $\alpha$.
  \item The initial data $u^0_{\alpha}\in C^\infty(\mathbb T^N)$ satisfy
\begin{equation}\label{eq:initial-approximation-existence}
 u^0_{\alpha}\geq\psi(\cdot,0),\quad \text{a.e. on}\ \mathbb{T}^N,
\end{equation}
and
\begin{equation*}
u^0_{\alpha}\rightarrow u^0\quad\text{in }L^p(\mathbb T^N)\quad\text{for every }1\leq p<\infty,\qquad\sup_\alpha\Vert u^0_\alpha\Vert_{L^\infty(\mathbb{T}^N)}<\infty.
\end{equation*}
  \item For every $n\in\mathbb{N}$, the penalization function is defined by
\begin{equation*}
b_n(r):=nr^-=n\max\{-r,0\},\qquad r\in\mathbb R.
\end{equation*}

\end{itemize}
It follows that $b_n$ is globally Lipschitz, nonnegative, and non-increasing with respect to $r$.
Moreover, for every fixed $r<0$, it holds that $b_n(r)\rightarrow+\infty$ as $n\to\infty$.

\begin{rem}
We provide an example of $u^0_{\alpha}$ satisfying the condition \eqref{eq:initial-approximation-existence}. Let $(\rho_\alpha)_{\alpha>0}$ be a family of standard spatial mollifiers.
Define
\begin{equation*}
u^0_{\alpha}:=\rho_\alpha*u^0+\Vert\rho_\alpha*\psi(\cdot,0)-\psi(\cdot,0)\Vert_{L^\infty(\mathbb T^N)}+\alpha.
\end{equation*}
Since $u^0\geq\psi(\cdot,0)$ a.e., we have $
\rho_\alpha*u^0\geq\rho_\alpha*\psi(\cdot,0)$, a.e.,
and then \eqref{eq:initial-approximation-existence} follows.
\end{rem}

With the above $\Phi_\alpha$, $b_n$ and $u_\alpha^0$, the penalized viscous equation can be written as
\begin{equation}\label{eq:penalized-viscous}
\text{\textbf{App Equ (1)}}\quad
\left\{
\begin{aligned}
\dd u_{\alpha,n} + \Div A(u_{\alpha,n})\dd t - \alpha\Delta u_{\alpha,n}\dd t
&= \Phi_\alpha(u_{\alpha,n})\dd W(t) + b_n(u_{\alpha,n}-\psi)\dd t,\\
u_{\alpha,n}(0) &= u_\alpha^0.
\end{aligned}
\right.
\end{equation}
We note that the term $b_n(u_{\alpha,n}-\psi)=n(u_{\alpha,n}-\psi)^-$
is nonnegative and acts only on the set $\{u_{\alpha,n}<\psi\}$,
thereby pushing the approximate solution upward towards the obstacle.

The second approximating equation is the penalized equation, stated as follows:
for any $n\in\mathbb N$, 
\begin{equation}\label{eq:penalized-hyperbolic}
\text{\textbf{App Equ (2)}}\quad
\left\{
\begin{aligned}
\dd u_n+\Div A(u_n)\dd t
&=\Phi(u_n)\dd W(t)+b_n(u_n-\psi)\dd t,\\
u_n(0) &= u^0.
\end{aligned}
\right.
\end{equation}
The notation \(\textbf{App Equ}\) is shorthand for the approximation equation.

In the sequel, we first show the well-posedness of solutions to the penalized viscous equation \eqref{eq:penalized-viscous}. We then fix $n$ and let $\alpha$ tend to 0 to obtain well-posedness of kinetic solutions to the penalized equation \eqref{eq:penalized-hyperbolic}. 
Finally, we pass to the limit $n\rightarrow \infty$ to establish well-posedness of kinetic solutions to the original obstacle problem (\ref{eq:obstacle-problem-existence}).

\subsection{A priori estimates for the penalized viscous equation}
In the following, we give the definition of solutions to the penalized viscous equation \eqref{eq:penalized-viscous}.
\begin{defn}\label{def:penalized-viscous-solution}
Let $\alpha\in(0,1)$ and $n\in\mathbb N$.
An $L^p$-valued continuous $\mathcal{F}_t$-adapted random field $u_{\alpha,n}$ is called a solution of \eqref{eq:penalized-viscous} if $u_{\alpha,n}\in L^p\bigl(\Omega;C([0,T];L^p(\mathbb T^N))\bigr)$, and, for every $\phi\in C^2(\mathbb T^N)$ and every $t\in[0,T]$, almost surely
\begin{align}\label{eq:weak-penalized-viscous}
\int_{\mathbb{T}^N}u_{\alpha,n}(t)\phi\dd x&=\int_{\mathbb{T}^N}u^0_{\alpha}\phi\dd x +\int_0^t\int_{\mathbb T^N}A(u_{\alpha,n})\cdot\nabla\phi\dd x\dd s\\
&\quad+\alpha\int_0^t\int_{\mathbb T^N}u_{\alpha,n}\Delta\phi\dd x\dd s+\int_0^t\int_{\mathbb T^N}b_n(u_{\alpha,n}-\psi)\phi\dd x\dd s\nonumber\\
&\quad+\sum_{k\geq1}\int_0^t\int_{\mathbb T^N}g_{k,\alpha}(x,u_{\alpha,n})\phi\dd x\dd\beta_{k}(s).\nonumber
\end{align}
\end{defn}
It is shown in \cite[Theorem 2.1]{GR00} that \eqref{eq:penalized-viscous} has a unique $L^{p}(\mathbb{T}^N)$ solution under Definition \ref{def:penalized-viscous-solution} provided $p$ is large enough.
\begin{lem}\label{lem:priori estimates}
Let $M_\psi:=1+\Vert\psi\Vert_{L^\infty(Q_T)}$. For every $p\geq2$, there exists $C_p>0$ independent of $\alpha$ and $n\in\mathbb N$ such that
\begin{align}\label{eq:priori-p}
&\E\sup_{t\in[0,T]}\int_{\mathbb T^N}|u_{\alpha,n}(t)-M_\psi|^p\dd x+\alpha\E\int_0^T\int_{\mathbb T^N}|u_{\alpha,n}-M_\psi|^{p-2}|\nabla u_{\alpha,n}|^2\dd x\dd s\\
&+\E\int_0^T\int_{\mathbb T^N}|u_{\alpha,n}-M_\psi|^{p-2}n|(u_{\alpha,n}-\psi)^-|^2\dd x\dd s +\E\int_0^T\int_{\mathbb T^N}|u_{\alpha,n}-M_\psi|^{p-2}n(u_{\alpha,n}-\psi)^-\dd x\dd s \nonumber \\
&\leq C_p \Big(1+\int_{\mathbb T^N}|u^0_{\alpha}-M_\psi|^p\dd x\Big),\nonumber
\end{align}
\begin{align}\label{eq:priori-p_L2omega}
&\E\bigg[\bigg|\alpha\int_0^T\int_{\mathbb T^N}|u_{\alpha,n}-M_\psi|^{p-2}|\nabla u_{\alpha,n}|^2\dd x\dd s\bigg|^2\bigg]+\E\bigg[\bigg|\int_0^T\int_{\mathbb T^N}|u_{\alpha,n}-M_\psi|^{p-2}n|(u_{\alpha,n}-\psi)^-|^2\dd x\dd s\bigg|^2\bigg]\nonumber\\
&\leq C_p \Big(1+\int_{\mathbb T^N}|u^0_{\alpha}-M_\psi|^{2p}\dd x\Big),
\end{align}
and
\begin{equation}\label{eq:uniform Lp nu}
\E\bigg[\bigg|\int_0^T\int_{\mathbb{T}^N}n(u_{\alpha,n}-\psi)^-\dd x\dd s\bigg|^p\bigg]\leq C_p.
\end{equation}
\end{lem}
\begin{proof}
Applying It\^{o}'s formula (cf. \cite[Theorem 3.1]{krylov2013relatively}), we have almost surely for every $t\in[0,T]$,
\begin{align*}
&\int_{\mathbb T^N}|u_{\alpha,n}(t)-M_\psi|^p\dd x-p(p-1)\int_0^t\int_{\mathbb T^N}|u_{\alpha,n}-M_\psi|^{p-2}A(u_{\alpha,n})\cdot\nabla u_{\alpha,n}\dd x\dd s\nonumber\\
&\quad+p(p-1)\alpha\int_0^t\int_{\mathbb T^N}|u_{\alpha,n}-M_\psi|^{p-2}|\nabla u_{\alpha,n}|^2\dd x\dd s\nonumber\\
&=\int_{\mathbb T^N}|u^0_{\alpha}-M_\psi|^p\dd x+p\int_0^t\int_{\mathbb T^N}|u_{\alpha,n}-M_\psi|^{p-2}(u_{\alpha,n}-M_\psi)
b_n(u_{\alpha,n}-\psi)\dd x\dd s\nonumber\\
&\quad+p\sum_{k\geq1}\int_0^t\int_{\mathbb T^N}|u_{\alpha,n}-M_\psi|^{p-2}(u_{\alpha,n}-M_\psi)g_{k,\alpha}(x,u_{\alpha,n})\dd x\dd\beta_k(s)\nonumber\\
&\quad+\frac{p(p-1)}{2}\int_0^t\int_{\mathbb T^N}|u_{\alpha,n}-M_\psi|^{p-2}G_\alpha^2(x,u_{\alpha,n})\dd x\dd s\nonumber,
\end{align*}
where $
G_\alpha^2(x,\xi):=\sum_{k\geq1}|g_{k,\alpha}(x,\xi)|^2$.
Note that
\[
\int_{\mathbb T^N}|u_{\alpha,n}-M_\psi|^{p-2}A(u_{\alpha,n})\cdot\nabla u_{\alpha,n}\dd x=\int_{\mathbb T^N}\nabla\cdot\int_0^{u_{\alpha,n}}|\tilde{r}-M_\psi|^{p-2}A(\tilde{r})\dd \tilde{r}\dd x=0.
\]
For the  penalized term, due to the definition of $M_\psi$, we have
\begin{align*}
&\int_0^t\int_{\mathbb T^N}|u_{\alpha,n}-M_\psi|^{p-2}(u_{\alpha,n}-M_\psi)b_n(u_{\alpha,n}-\psi)\dd x\dd s\\
&\leq\int_0^t\int_{\mathbb T^N}|u_{\alpha,n}-M_\psi|^{p-2}(u_{\alpha,n}-\psi-1)b_n(u_{\alpha,n}-\psi)\dd x\dd s\\
&=-\int_0^t\int_{\mathbb T^N}|u_{\alpha,n}-M_\psi|^{p-2}n|(u_{\alpha,n}-\psi)^-|^2\dd x\dd s-\int_0^t\int_{\mathbb T^N}|u_{\alpha,n}-M_\psi|^{p-2}n(u_{\alpha,n}-\psi)^-\dd x\dd s.
\end{align*}
For the term involving $G_\alpha$, using \textbf{Hypothesis (H1)}, we have
\begin{align*}
&\int_0^t\int_{\mathbb T^N}|u_{\alpha,n}-M_\psi|^{p-2}G_\alpha^2(x,u_{\alpha,n})\dd x\dd s\leq C_p\int_0^t\int_{\mathbb T^N}[1+|u_{\alpha,n}-M_\psi|^{p}]\dd x\dd s.
\end{align*}
As a result, we reach
\begin{align}\notag
&\int_{\mathbb T^N}|u_{\alpha,n}(t)-M_\psi|^p\dd x
+\int_0^t\int_{\mathbb T^N}|u_{\alpha,n}-M_\psi|^{p-2}n|(u_{\alpha,n}-\psi)^-|^2\dd x\dd s\\
\notag
&\quad+\int_0^t\int_{\mathbb T^N}|u_{\alpha,n}-M_\psi|^{p-2}n(u_{\alpha,n}-\psi)^-\dd x\dd s
+p(p-1)\alpha\int_0^t\int_{\mathbb T^N}|u_{\alpha,n}-M_\psi|^{p-2}|\nabla u_{\alpha,n}|^2\dd x\dd s\\
\notag
&\leq\int_{\mathbb T^N}|u^0_{\alpha}-M_\psi|^p\dd x
+p\sum_{k\geq1}\int_0^t\int_{\mathbb T^N}|u_{\alpha,n}-M_\psi|^{p-2}(u_{\alpha,n}-M_\psi)g_{k,\alpha}(x,u_{\alpha,n})\dd x\dd\beta_k(s)\\
\label{estimate}
&\quad+C_p\int_0^t\int_{\mathbb T^N}[1+|u_{\alpha,n}-M_\psi|^{p}]\dd x\dd s.
\end{align}

Applying Burkholder-Davis-Gundy inequality and Young inequality, we have for $\tau\in(0,T]$ and $\bar{\epsilon}\in(0,1)$,
\begin{align*}
&\E\sup_{t\in[0,\tau]}\bigg|\sum_{k\geq1}\int_0^t\int_{\mathbb T^N}|u_{\alpha,n}-M_\psi|^{p-2}(u_{\alpha,n}-M_\psi)g_{k,\alpha}(x,u_{\alpha,n})\dd x\dd\beta_k(s)\bigg|\\
&\leq C\E\bigg[\int_0^\tau\sum_{k\geq1}\bigg|\int_{\mathbb T^N}|u_{\alpha,n}-M_\psi|^{p-1}|g_{k,\alpha}(x,u_{\alpha,n})|\dd x\bigg|^2\dd s\bigg]^{\frac{1}{2}}\\
&\leq C(D_0)\E\bigg[\int_0^\tau\Big(\int_{\mathbb T^N}|u_{\alpha,n}-M_\psi|^{p}\dd x\Big)\sum_{k\geq1}\Big(\int_{\mathbb T^N}|u_{\alpha,n}-M_\psi|^{p-2}|g_{k,\alpha}(x,u_{\alpha,n})|^2\dd x\Big)\dd s\bigg]^{\frac{1}{2}}\\
&\leq C(D_0)\E\bigg[\sup_{t\in[0,\tau]}\Big(\int_{\mathbb T^N}|u_{\alpha,n}-M_\psi|^{p}\dd x\Big)\cdot\int_0^\tau\int_{\mathbb T^N}[1+|u_{\alpha,n}-M_\psi|^{p}]\dd x\dd s\bigg]^{\frac{1}{2}}\\
&\leq \bar{\epsilon}\E\bigg[\sup_{t\in[0,\tau]}\Big(\int_{\mathbb T^N}|u_{\alpha,n}-M_\psi|^{p}\dd x\Big)\bigg]+\frac{C(D_0)}{\bar{\epsilon}}\E\bigg[\int_0^\tau\int_{\mathbb T^N}[1+|u_{\alpha,n}-M_\psi|^{p}]\dd x\dd s\bigg].
\end{align*}
Collecting all the above estimates and taking $\bar{\epsilon}$ small enough, then Gr\"onwall's inequality yields \eqref{eq:priori-p}.

To derive \eqref{eq:priori-p_L2omega}, we only retain the first term on the left-hand side of \eqref{estimate}, and square both sides. 
Note that
\begin{align*}
&\E\sup_{t\in[0,\tau]}\bigg|\sum_{k\geq1}\int_0^t\int_{\mathbb T^N}|u_{\alpha,n}-M_\psi|^{p-2}(u_{\alpha,n}-M_\psi)g_{k,\alpha}(x,u_{\alpha,n})\dd x\dd\beta_k(s)\bigg|^2\\
&\leq C\E\bigg[\sup_{t\in[0,\tau]}\Big(\int_{\mathbb T^N}|u_{\alpha,n}-M_\psi|^{p}\dd x\Big)\cdot\int_0^\tau\int_{\mathbb T^N}1+|u_{\alpha,n}-M_\psi|^{p}\dd x\dd s\bigg]\\
&\leq \bar{\epsilon}\E\sup_{t\in[0,\tau]}\Big(\int_{\mathbb T^N}|u_{\alpha,n}-M_\psi|^{p}\dd x\Big)^2+C(\bar{\epsilon},T)+C(\bar{\epsilon},T)\E \int^{\tau}_0\Big(\int_{\mathbb T^N}|u_{\alpha,n}-M_\psi|^{p}\dd x\Big)^2 \dd s,
\end{align*}
by Gr\"onwall's inequality,
we deduce that
\begin{align}\notag
\E \sup_{t\in [0,T]}\Big|\int_{\mathbb T^N}|u_{\alpha,n}(t)-M_\psi|^p\dd x\Big|^2
\leq C(p,T,D_0)\Big(1+\int_{\mathbb T^N}|u^0_{\alpha}-M_\psi|^{2p}\dd x\Big).
\end{align}
Arguing as above, we obtain \eqref{eq:priori-p_L2omega}.

 To prove \eqref{eq:uniform Lp nu}, we take the test function $\phi\equiv1$ in \eqref{eq:weak-penalized-viscous}. This yields that for any $t\in[0,T]$,
\begin{align*}
\int_0^t\int_{\mathbb T^N} n(u_{\alpha,n}-\psi)^-\dd x\dd s
&=\int_{\mathbb T^N}u_{\alpha,n}(t)\dd x-\int_{\mathbb T^N}u^0_{\alpha}\dd x
\nonumber\\
&\quad-\sum_{k\geq1}\int_0^t\int_{\mathbb T^N}g_{k,\alpha}(x,u_{\alpha,n})\dd x\dd\beta_k(s).
\end{align*}
Using Burkholder-Davis-Gundy inequality and \eqref{eq:priori-p}, we obtain \eqref{eq:uniform Lp nu}.
\end{proof}
Now, we derive the kinetic formulation for the penalized viscous equation \eqref{eq:penalized-viscous}. We need to define some measures. Let
\begin{equation}\label{rs-33}
\dd \nu_{\alpha,n}(x,t):= n(u_{\alpha,n}-\psi)^-\dd x\dd t,
\end{equation}
\begin{equation}\label{rs-34}
\dd m_{\alpha,n}(x,t,\xi):= \alpha|\nabla u_{\alpha,n}(x,t)|^2\delta_{u_{\alpha,n}(x,t)}(\dd\xi)\dd x\dd t,
\end{equation}
and
\begin{equation}\label{eq:def-lambda-eta-epsilon}
\dd\lambda_{\alpha,n}(x,t,\xi)
:=
n\bigl[(u_{\alpha,n}(x,t)-\psi(x,t))^-\bigr]^2\int_0^1\delta_{\psi(x,t)+\theta(u_{\alpha,n}(x,t)-\psi(x,t))}(\dd\xi)\dd\theta\dd x\dd t.
\end{equation}
Set
\begin{equation}\label{eq:def-q-eta-epsilon}
q_{\alpha,n}:=m_{\alpha,n}+\lambda_{\alpha,n}.
\end{equation}
\begin{rem}[The barrier-substitution identity]\label{rem:barrier-substitution-identity}
For every $\phi\in C_c^1(\mathbb R)$,
\begin{align*}
\big\langle \delta_{\xi=u_{\alpha,n}}\nu_{\alpha,n},\phi\big\rangle&=\phi(u_{\alpha,n})\nu_{\alpha,n},\\
\big\langle\delta_{\xi=\psi}\nu_{\alpha,n}+\partial_\xi\lambda_{\alpha,n},\phi\big\rangle&=\phi(\psi)\nu_{\alpha,n}-\left(\int_{u_{\alpha,n}}^{\psi}\phi'(\xi)\,\dd\xi\right)\nu_{\alpha,n}=\phi(u_{\alpha,n})\nu_{\alpha,n}.
\end{align*}
Here the equalities are understood as identities of Radon measures
in $(x,t)$. 
Thus
\[
\delta_{\xi=u_{\alpha,n}}\nu_{\alpha,n}=\delta_{\xi=\psi}\nu_{\alpha,n}+\partial_\xi\lambda_{\alpha,n}
\]
in the sense of distributions in the kinetic variable. 
The sign of the term $\partial_\xi\lambda_{\alpha,n}$ follows from the distributional convention
\[
\langle \partial_\xi\lambda_{\alpha,n},\phi\rangle=-\langle\lambda_{\alpha,n},\partial_\xi\phi\rangle.
\]
Moreover,
\[
\dd\lambda_{\alpha,n}(x,t,\xi)=\one_{\{u_{\alpha,n}(x,t)<\xi<\psi(x,t)\}}\,\dd\nu_{\alpha,n}(x,t)\dd\xi.
\]
This measure is nonnegative because $\nu_{\alpha,n}$ is supported on the set $\{(x,t):u_{\alpha,n}(x,t)<\psi(x,t)\}$.
\end{rem}

\begin{prop}[Kinetic formulation]\label{prop:kinetic-viscous-approximation} Let $u_{\alpha,n}$ be the weak solution of \eqref{eq:penalized-viscous} and $f_{\alpha,n}:=\one_{\{u_{\alpha,n}(x,t)>\xi\}}$. Then $f_{\alpha,n}$ satisfies, for every $\varphi\in C_c^2(\mathbb T^N\times[0,T)\times\mathbb R)$,
\begin{align}\label{eq:kinetic-eta-epsilon}
&\int_0^T \langle f_{\alpha,n}(t),\partial_t\varphi(t)\rangle\dd t+\langle f^0_{\alpha},\varphi(0)\rangle+\int_0^T\langle f_{\alpha,n}(t),a(\xi)\cdot\nabla_x\varphi(t)+\alpha\Delta_x\varphi(t)\rangle\dd t\\
&=-\sum_{k\geq1}\int_0^T\int_{\mathbb T^N}g_{k,\alpha}(x,u_{\alpha,n})\varphi(x,t,u_{\alpha,n})\dd x\dd\beta_k(t)\nonumber\\
&\quad-\frac{1}{2}\int_0^T\int_{\mathbb{T}^N}G_\alpha^2(x,u_{\alpha,n})\partial_\xi\varphi(x,t,u_{\alpha,n})\dd x\dd t\nonumber\\
&\quad+q_{\alpha,n}(\partial_\xi\varphi)-\int_0^T\int_{\mathbb T^N}\varphi(x,t,\psi(x,t))
\dd\nu_{\alpha,n}(x,t),\qquad\text{a.s.},\nonumber
\end{align}
where $f^0_{\alpha}(x,\xi)=\one_{\{u^0_{\alpha}(x)>\xi\}}
$, $\nu_{\alpha,n}$ is defined by \eqref{rs-33} and $q_{\alpha,n}$ is defined by \eqref{eq:def-q-eta-epsilon}.
\end{prop}
\begin{proof}
Apply It\^o's formula (cf. \cite[Theorem 3.1]{krylov2013relatively}) to $\int_{-\infty}^{u_{\alpha,n}}\varphi(x,t,\xi)\dd \xi$.
The viscous term satisfies
\begin{align*}
&\int_{\mathbb{T}^N}\varphi(x,t,u_{\alpha,n})\Delta u_{\alpha,n}\dd x\\
&=-\int_{\mathbb{T}^N}\nabla_x\varphi(x,t,u_{\alpha,n})\cdot\nabla u_{\alpha,n}\dd x-\int_{\mathbb{T}^N}\partial_\xi\varphi(x,t,u_{\alpha,n})|\nabla u_{\alpha,n}|^2\dd x\\
&=-\int_{\mathbb{T}^N}\int_{-\infty}^{u_{\alpha,n}}\nabla_x\partial_\xi\varphi(x,t,\xi)\dd \xi\cdot\nabla_x u_{\alpha,n}\dd x-\int_{\mathbb{T}^N}\partial_\xi\varphi(x,t,u_{\alpha,n})|\nabla u_{\alpha,n}|^2\dd x\\
&=\int_{\mathbb{T}^N}\int_{-\infty}^{u_{\alpha,n}}\int_{-\infty}^{r}\Delta_x\partial_\xi\varphi(x,t,\xi)\dd \xi\dd r\dd x-\int_{\mathbb{T}^N}\partial_\xi\varphi(x,t,u_{\alpha,n})|\nabla u_{\alpha,n}|^2\dd x\\
&=\int_{\mathbb{T}^N}\int_{-\infty}^{u_{\alpha,n}}\Delta_x\varphi(x,t,\xi)\dd \xi\dd x-\int_{\mathbb{T}^N}\partial_\xi\varphi(x,t,u_{\alpha,n})|\nabla u_{\alpha,n}|^2\dd x\\
&=\int_{\mathbb{T}^N}\int_{\mathbb{R}}f_{\alpha,n}(x,t,\xi)\Delta_x\varphi(x,t,\xi)\dd \xi\dd x-\int_{\mathbb{T}^N}\partial_\xi\varphi(x,t,u_{\alpha,n})|\nabla u_{\alpha,n}|^2\dd x.
\end{align*}
The second term produces $m_{\alpha,n}(\partial_\xi\varphi)$.
Note that for every $\varphi\in C_c^2(\mathbb{T}^N\times[0,T)\times\mathbb{R})$, based on the definition of $\nu_{\alpha,n}$, we have
\begin{align}\label{eq:penalty-decomposition-eta-epsilon}
&\int_0^T\int_{\mathbb T^N}\Big(\varphi(x,t,\psi(x,t))-\varphi(x,t,u_{\alpha,n})\Big)\dd\nu_{\alpha,n}(x,t)\\
&=\int_0^T\int_{\mathbb{T}^N}\int_0^1n[(u_{\alpha,n}(x,t)-\psi(x,t))^-]^2\partial_\xi\varphi(x,t,\psi+\theta(u_{\alpha,n}-\psi))\dd\theta\,\dd x\dd t\nonumber
\\
&=\lambda_{\alpha,n}(\partial_\xi\varphi).\nonumber
\end{align}
This, together with \eqref{eq:def-q-eta-epsilon}, implies the penalization term and the kinetic measure term.
Based on the definition of $f_{\alpha,n}$, we have \eqref{eq:kinetic-eta-epsilon}.
\end{proof}

\begin{cor}[Uniform bounds for the defect measures]\label{cor:uniform-bounds-defect-measure}
For every $p\geq0$, there exists $C_p>0$, independent of $\alpha\in(0,1)$ and $n\in\mathbb N$, such that
\begin{align}
&\E \bigg[\bigg|\int_{\mathbb{T}^N\times[0,T]\times\mathbb{R}}(1+|\xi|^p)\dd m_{\alpha,n}\bigg|^2\bigg]\\
&+\E\bigg[\bigg|\int_{\mathbb{T}^N\times[0,T]\times\mathbb{R}}(1+|\xi|^p)\dd \lambda_{\alpha,n}\bigg|^2\bigg]\leq C(p,D_0,M_\psi,T),\label{eq:bounds for m}
\end{align}
where $m_{\alpha,n}$ is defined by \eqref{rs-34} and $\lambda_{\alpha,n}$ is given by \eqref{eq:def-lambda-eta-epsilon}.
\end{cor}
\begin{proof}
Note that
\[
\int(1+|\xi|^p)\dd m_{\alpha,n}=\alpha\int_0^T\int_{\mathbb T^N}(1+|u_{\alpha,n}|^p)|\nabla u_{\alpha,n}|^2\dd x\dd t,
\]
and
\begin{align}
\int_{\mathbb T^N\times[0,T]\times\mathbb R}(1+|\xi|^p)\dd\lambda_{\alpha,n}
&=n\int_0^T\int_{\mathbb{T}^N}|(u_{\alpha,n}-\psi)^-|^2\int_0^1(1+|\psi+\theta(u_{\alpha,n}-\psi)|^p)\dd\theta\dd x\dd t\nonumber\\
&\quad\leq nC\int_0^T\int_{\mathbb{T}^N}|(u_{\alpha,n}-\psi)^-|^2(1+|\psi|^p+|u_{\alpha,n}-\psi|^p)\dd x\dd t\nonumber\\
&\quad\leq nC\int_0^T\int_{\mathbb{T}^N}|(u_{\alpha,n}-\psi)^-|^2(1+|u_{\alpha,n}-M_\psi|^p)\dd x\dd t.\nonumber
\end{align}
The estimate \eqref{eq:bounds for m} follows from \eqref{eq:priori-p} with exponent $p+2$.
\end{proof}

\subsection{Well-posedness of kinetic solutions to the penalized equation}

\begin{defn}[Kinetic solution of the penalized equation]
\label{def:kinetic-penalized}
A predictable process $u_n$ is called a kinetic solution of \eqref{eq:penalized-hyperbolic}, if it satisfies \eqref{eq:usual moment bound} and there exists a kinetic measure $m_n$ such that $f_n=\one_{\{u_n>\xi\}}$ satisfies, for every $\varphi\in C_c^1\big(\mathbb T^N\times[0,T)\times\mathbb R\big)$,
\begin{align}\label{eq:kinetic-penalized-direct}
&\int_0^T\langle f_n(t),\partial_t\varphi(t)\rangle\dd t+\langle f^0,\varphi(0)\rangle+\int_0^T\langle f_n(t),a(\xi)\cdot\nabla_x\varphi(t)\rangle\,\dd t\nonumber\\
&=-\sum_{k\geq1}\int_0^T\int_{\mathbb T^N}g_k(x,u_n)\varphi(x,t,u_n)\dd x\dd\beta_k(t)\nonumber\\
&\quad-\frac12\int_0^T\int_{\mathbb{T}^N}G^2(x,u_n)\partial_\xi\varphi(x,t,u_n)\dd x\dd t\nonumber\\
&\quad+m_n(\partial_\xi\varphi)-\int_0^T\int_{\mathbb T^N}b_n(u_n-\psi)\varphi(x,t,u_n)\dd x\dd t,
\qquad\text{a.s.},
\end{align}
where
$f^0=\one_{\{u^0>\xi\}}$.

\end{defn}

Define
\begin{align*}
\dd\nu_n(x,t)&:=b_n(u_n-\psi)\dd x\dd t=n(u_n-\psi)^-\dd x\dd t,\\
\dd\lambda_{n}(x,t,\xi)&:=\int_0^1n[(u_{n}(x,t)-\psi(x,t))^-]^2\delta_{\psi+\theta(u_{n}-\psi)}(\dd\xi)\dd\theta\,\dd x\dd t.\\
q_n&:=m_n+\lambda_n.\nonumber
\end{align*}
Note that \eqref{eq:penalty-decomposition-eta-epsilon} also holds for $\nu_n$ and $\lambda_{n}$, hence \eqref{eq:kinetic-penalized-direct} is equivalent to
\begin{align}\label{eq:kinetic-penalized-obstacle-form}
&\int_0^T\langle f_n(t),\partial_t\varphi(t)\rangle\dd t+\langle f^0,\varphi(0)\rangle+\int_0^T\langle f_n(t),a(\xi)\cdot\nabla_x\varphi(t)\rangle\dd t
\nonumber\\
&=-\sum_{k\geq1}\int_0^T\int_{\mathbb T^N}g_k(x,u_n)\varphi(x,t,u_n)\dd x\dd\beta_k(t)\nonumber\\
&\quad-\frac12\int_0^T\int_{\mathbb{T}^N}G^2(x,u_n)\partial_\xi\varphi(x,t,u_n)\dd x\dd t\nonumber\\
&\quad +q_n(\partial_\xi\varphi)-\int_0^T\int_{\mathbb T^N}\varphi(x,t,\psi(x,t))\dd\nu_n(x,t).
\end{align}

\begin{lem}[Vanishing viscosity limit]
\label{lem:vanishing-viscosity-fixed-epsilon}
For fixed $n\in\mathbb N$, equation \eqref{eq:penalized-hyperbolic} admits a unique kinetic solution $u_n$ in the sense of Definition \ref{def:kinetic-penalized}.
Moreover, as $\alpha\downarrow0$,
\begin{equation*}
u_{\alpha,n}\rightarrow u_n\quad\text{strongly in }L^p(\Omega\times Q_T)
\end{equation*}
for every $1\leq p<\infty$, and
\begin{align*}
\nu_{\alpha,n}&\rightarrow\nu_n \quad\text{strongly in }L^1(\Omega\times Q_T),\\
f_{\alpha,n}&\rightarrow f_n\quad\text{strongly in }L^1_{\mathrm{loc}}(\Omega\times Q_T\times\mathbb R),\nonumber
\end{align*}
and for every $\phi\in C_c^1(\mathbb T^N\times[0,T]\times\mathbb R)$, 
\begin{equation*}
\lambda_{\alpha,n}(\phi)\rightarrow\lambda_n(\phi)\quad\text{in }L^1(\Omega).
\end{equation*}
Moreover, we have
\begin{equation}
\lim_{\tau\downarrow0}\frac{1}{\tau}\int_0^\tau\Vert u_{n}(t)-u^0\Vert _{L^1(\mathbb{T}^N)}\dd t=0.\label{eq:weak-initial for u_varepsilon}
\end{equation}
\end{lem}
\begin{proof}
The proof is similar to \cite[the proof of Theorem 19]{DV10-publish}, and we only focus on the new obstacle term $b_{n}(u_{\alpha,n}-\psi)$.
Fix $n\in\mathbb N$ and let $\alpha_j\downarrow0$.
The uniform moment estimates imply that the Young measures $\delta_{u_{\alpha_j,n}(x,t)}$ are tight.
Thus, along a subsequence,
\[
f_{\alpha_j,n}\stackrel{\ast}{\rightharpoonup} f_n\quad\text{in }L^\infty(\Omega\times Q_T\times\mathbb R),
\]
where $f_n$ is a kinetic function associated with a Young measure $\mu^n_{x,t}$.
The viscous term converges to zero, while the transport, stochastic, It\^o correction, and kinetic measure term $m_n$ pass to the limit exactly as in \cite[Theorem~19]{DV10-publish}. 
The penalization term also passes to the limit, since
\[
(x,t,\xi)\longmapsto b_n(\xi-\psi(x,t))\varphi(x,t,\xi)
\]
is continuous and bounded for every compactly supported test function $\varphi$. 
Hence $f_n$ satisfies the generalized kinetic formulation associated with \eqref{eq:penalized-hyperbolic}.

We next apply the doubling-of-variables argument. 
The only additional term is generated by the penalization. 
Since $b_n$ is non-increasing and $n$-Lipschitz, we have
\begin{align*}
&\bigl[b_n(\xi-\psi(x,t))-b_n(\zeta-\psi(y,t))\bigr](\xi-\zeta)\\
&\qquad\leq n|\psi(x,t)-\psi(y,t)|\,|\xi-\zeta|.
\end{align*}
After spatial convolution, the right-hand side tends to zero as the mollification parameter tends to zero, by the uniform continuity of $\psi$. 
Therefore, the reduction argument of \cite{DV10-publish} remains valid and shows that $f_n$ is an equilibrium:
\[
\mu^n_{x,t}=\delta_{u_n(x,t)},\qquad f_n=\one_{\{u_n>\xi\}}.
\]
The same $L^1$-contraction principle gives uniqueness.
Since the limiting Young measure is a Dirac measure and the approximate solutions satisfy uniform moment estimates,
\[
u_{\alpha,n}\rightarrow u_n \quad\text{strongly in }L^p(\Omega\times Q_T)
\]
for every finite $1\leq p<\infty$.
Moreover, for the convergence of $\nu_{\alpha,n}$, note that $b_n$ is Lipschitz with constant $n$, and hence
\begin{align*}
\Vert b_n(u_{\alpha,n}-\psi)-b_n(u_n-\psi)\Vert _{L^1(\Omega\times Q_T)}\leq n\Vert u_{\alpha,n}-u_n\Vert _{L^1(\Omega\times Q_T)}\rightarrow0.
\end{align*}
This proves the convergence of the penalization measures.  
It remains to prove the convergence of the obstacle defect measures $\lambda_{\alpha,n}$.
For $\phi\in C_c^1(\mathbb T^N\times[0,T]\times\mathbb R)$, set
\[
F_\phi(x,t,r):=n\bigl[(r-\psi(x,t))^-\bigr]^2\int_0^1\phi\bigl(x,t,\psi(x,t)+\theta(r-\psi(x,t))\bigr)\,\dd\theta .
\]
Then
\[
\lambda_{\alpha,n}(\phi)=\int_{Q_T}F_\phi(x,t,u_{\alpha,n}(x,t))\,\dd x\,\dd t,\qquad\lambda_n(\phi)=\int_{Q_T}F_\phi(x,t,u_n(x,t))\,\dd x\,\dd t.
\]
The function $F_\phi$ is continuous in $r$ and has at most quadratic growth, uniformly in $(x,t)$. 
Since $u_{\alpha,n}\to u_n$ strongly in $L^2(\Omega\times Q_T)$ and the family $(u_{\alpha,n})_\alpha$ has uniformly bounded higher moments, the family $F_\phi(\cdot,\cdot,u_{\alpha,n})$ is uniformly integrable.
Vitali's theorem therefore gives
\[
\lambda_{\alpha,n}(\phi)\longrightarrow\lambda_n(\phi)\qquad\text{in }L^1(\Omega).
\]
Moreover, the weak initial trace \eqref{eq:weak-initial for u_varepsilon} is a direct consequence of \cite[Corollary 16]{DV10}, which completes the proof.
\end{proof}

We deduce from Lemma \ref{lem:priori estimates} and Lemma \ref{lem:vanishing-viscosity-fixed-epsilon} that the following estimates hold for $u_n$ uniformly with respect to $n$. For every $p\geq2$, there exists $C_p>0$ independent of $n\in\mathbb N$ such that
\begin{description}
  \item[(i)] \begin{align}\label{eq:priori-p-n}
&\E\sup_{t\in[0,T]}\int_{\mathbb T^N}|u_{n}(t)-M_\psi|^p\dd x+\E\int_0^T\int_{\mathbb T^N}|u_{n}-M_\psi|^{p-2}n|(u_{n}-\psi)^-|^2\dd x\dd s\\
&\quad+\E\int_0^T\int_{\mathbb T^N}|u_{n}-M_\psi|^{p-2}n(u_{n}-\psi)^-\dd x\dd s\leq C_p \Big(1+\int_{\mathbb T^N}|u^0-M_\psi|^p\dd x\Big),\nonumber
\end{align}
  \item[(ii)] \begin{align*}
\E\bigg[\bigg|\int_0^T\int_{\mathbb T^N}|u_{n}-M_\psi|^{p-2}n|(u_{n}-\psi)^-|^2\dd x\dd s\bigg|^2\bigg]
\leq C_p \Big(1+\int_{\mathbb T^N}|u^0-M_\psi|^{2p}\dd x\Big),
\end{align*}
  \item[(iii)] \begin{equation}\label{eq:uniform Lp nu-n}
\E\bigg[\bigg|\int_0^T\int_{\mathbb{T}^N}n(u_{n}-\psi)^-\dd x\dd t\bigg|^p\bigg]\leq C_p.
\end{equation}
\end{description}
Moreover, by the estimates obtained before passing to the limit $\alpha\downarrow0$ and by lower semicontinuity, for every $p\geq0$,
\begin{equation*}
\sup_{n\geq1}\E\bigg[\nu_n\bigl(\mathbb T^N\times[0,T]\bigr)^2+\bigg(\int_{\mathbb T^N\times[0,T]\times\mathbb R}(1+|\xi|^p)\,\dd q_n
\bigg)^2\bigg]\leq C_p.
\end{equation*}
\begin{lem}[Comparison for the penalized equation]
\label{lem:comparison-penalty-parameter}
Let $1\leq n\leq m$, and let $u_n$ and $u_m$ be kinetic solutions of the corresponding penalized equations, with initial data $u_n^0$ and $u_m^0$, respectively.
Then, for almost every $t\in(0,T)$,
\begin{align*}
\E\bigl\|(u_n(t)-u_m(t))^+\bigr\|_{L^1(\mathbb T^N)}\leq\bigl\|(u_n^0-u_m^0)^+\bigr\|_{L^1(\mathbb T^N)}.
\end{align*}
If $u_n^0=u_m^0=u^0$, then we have $u_n\leq u_m$ a.s.
Consequently, the family $(u_n)_{n\geq1}$ is non-decreasing.
\end{lem}

\begin{proof}
We carry out the proof via the doubling variables method, adopting an approach similar to that of Theorem \ref{thm:uniqueness}, based on the kinetic equation \eqref{eq:kinetic-penalized-direct}.
Unlike the aforementioned theorem, which involves the Radon reflection measure $\nu$, here we work with the penalization terms $b_n(u_n-\psi)$ and $b_m(u_m-\psi)$. Accordingly, our analysis focuses on the contributions originating from these penalization terms.

Fix $n\leq m$. 
Let $f_n(x,s,\xi)=\one_{\{\xi<u_n(x,s)\}}$, $f_m(y,s,\zeta)=\one_{\{\zeta<u_m(y,s)\}}$, and $\bar f_n:=1-f_n$, $\bar f_m:=1-f_m$.
Taking the same test function
\[
\varphi_{\gamma,\delta}(x,\xi,y,\zeta)=\rho_\gamma(x-y)\kappa_\delta(\xi-\zeta)
\]
as in Theorem \ref{thm:uniqueness}, the term associated with the penalization terms, also denoted by $I^\nu_{\gamma,\delta}(t)$ is given by
\begin{align}
&I^\nu_{\gamma,\delta}(t)\nonumber\\
&=\E\int_{\mathbb T^N}\int_{\mathbb R}\int_{\mathbb T^N\times[0,t]}\varphi_{\gamma,\delta}\bigl(x,u_n(x,s),y,\zeta\bigr)\bar f_m(y,s,\zeta)b_n\bigl(u_n(x,s)-\psi(x,s)\bigr)\dd x\dd s\dd y\dd\zeta\nonumber\\
&\quad-\E\int_{\mathbb T^N}\int_{\mathbb R}\int_{\mathbb T^N\times[0,t]}\varphi_{\gamma,\delta}\bigl(x,\xi,y,u_m(y,s)\bigr)f_n(x,s,\xi)
b_m\bigl(u_m(y,s)-\psi(y,s)\bigr)\dd y\dd s\dd x\dd\xi.\label{eq:penalty-term-nm}
\end{align}
Here, we use $f_n$ and $f_m$ instead of their left- or right-continuous representatives, since the corresponding integrations are taken with respect to the Lebesgue measure.

Since $b_k(z)=kz^-$ and $n\leq m$, we have $b_n(z)\leq b_m(z)$ for every $z\in\mathbb R$.
Moreover, all the remaining factors in the first integral of \eqref{eq:penalty-term-nm} are nonnegative.
Therefore, we reach
\begin{align}
&I^\nu_{\gamma,\delta}(t)\nonumber\\
&\leq\E\int_{\mathbb T^N}\int_{\mathbb R}\int_{\mathbb T^N\times[0,t]}\varphi_{\gamma,\delta}\bigl(x,u_n(x,s),y,\zeta\bigr)\bar f_m(y,s,\zeta)b_m\bigl(u_n(x,s)-\psi(x,s)\bigr)\dd x\dd s\dd y\dd\zeta\nonumber\\
&\quad-\E\int_{\mathbb T^N}\int_{\mathbb R}\int_{\mathbb T^N\times[0,t]}\varphi_{\gamma,\delta}\bigl(x,\xi,y,u_m(y,s)\bigr)f_n(x,s,\xi)
b_m\bigl(u_m(y,s)-\psi(y,s)\bigr)\dd y\dd s\dd x\dd\xi.\label{eq:penalty-term-nm-upper}
\end{align}

Define
\begin{equation*}
K_\delta(z):=\int_{-\infty}^{z}\kappa_\delta(r)\dd r,\qquad z\in\mathbb R.
\end{equation*}
Since $\kappa_\delta$ is even, nonnegative, and has total mass one, it follows that $0\leq K_\delta(z)\leq1$ is non-decreasing and satisfies
\[
K_\delta(-\infty)=0,\qquad K_\delta(+\infty)=1,\qquad K_\delta(0)=\frac12,\qquad K_\delta(-z)=1-K_\delta(z).
\]
Owing to the definition of $K_{\delta}$,
the velocity integrals can be written as
\begin{align}
&\int_{\mathbb R}\kappa_\delta\bigl(u_n(x,s)-\zeta\bigr)\bar f_m(y,s,\zeta)\dd\zeta
=K_\delta\bigl(u_n(x,s)-u_m(y,s)\bigr),\label{eq:first-velocity-integral-K}
\end{align}
and
\begin{align}
&\int_{\mathbb R}\kappa_\delta\bigl(\xi-u_m(y,s)\bigr)f_n(x,s,\xi)\dd\xi=K_\delta\bigl(u_n(x,s)-u_m(y,s)\bigr).\label{eq:second-velocity-integral-K}
\end{align}
It follows from \eqref{eq:penalty-term-nm-upper}, \eqref{eq:first-velocity-integral-K}, and \eqref{eq:second-velocity-integral-K} that
\begin{align}
I^\nu_{\gamma,\delta}(t)&\leq\E\int_{\mathbb T^N}\int_{\mathbb T^N\times[0,t]}\rho_\gamma(x-y)K_\delta\bigl(u_n(x,s)-u_m(y,s)\bigr)\bigl[\widetilde B_n^{(m)}(x,s)-B_m(y,s)\bigr]\dd x\dd s\dd y\nonumber\\
&=:I^B_{\gamma,\delta}(t),\label{eq:def-I-B-nm}
\end{align}
where
\[
\widetilde B_n^{(m)}(x,s):=b_m\bigl(u_n(x,s)-\psi(x,s)\bigr),\qquad B_m(x,s):=b_m\bigl(u_m(x,s)-\psi(x,s)\bigr).
\]

In the following, we aim to show that
\begin{align}\label{eq:limit-I-B-nm}
I^B_{\gamma,\delta}(t)\rightarrow I^B(t):=\E\int_0^t\int_{\mathbb T^N}\one_{\{u_n(x,s)>u_m(x,s)\}}\bigl[\widetilde B_n^{(m)}(x,s)-B_m(x,s)\bigr]\dd x\dd s
\end{align}
when $\delta=\gamma^{4/3}$ and $\gamma\downarrow0$.

Since $\int_{\mathbb T^N}\rho_\gamma(x-y)\dd y=1$,
we have
\begin{align*}
&I^B_{\gamma,\delta}(t)-I^B(t)\\
&=\E\int_0^t\int_{\mathbb T^N}\int_{\mathbb T^N}\rho_\gamma(x-y)\bigg\{K_\delta\bigl(u_n(x,s)-u_m(y,s)\bigr)\bigl[\widetilde B_n^{(m)}(x,s)-B_m(y,s)\bigr]\\
&\hspace{4cm}-\one_{\{u_n(x,s)>u_m(x,s)\}}\bigl[\widetilde B_n^{(m)}(x,s)-B_m(x,s)\bigr]\bigg\}\dd x\dd y\dd s\\
&=\E\int_0^t\int_{\mathbb T^N}\int_{\mathbb T^N}\rho_\gamma(x-y)K_\delta\bigl(u_n(x,s)-u_m(y,s)\bigr)\bigl[B_m(x,s)-B_m(y,s)\bigr]\dd x\dd y\dd s\\
&\quad+\E\int_0^t\int_{\mathbb T^N}\int_{\mathbb T^N}\rho_\gamma(x-y)\bigg[K_\delta\bigl(u_n(x,s)-u_m(y,s)\bigr)\\
&\hspace{4cm}-\one_{\{u_n(x,s)>u_m(x,s)\}}\bigg]\bigl[\widetilde B_n^{(m)}(x,s)-B_m(x,s)\bigr]\dd x\dd y\dd s\\
&=:A_{\gamma,\delta}(t)+R_{\gamma,\delta}(t).
\end{align*}

For $A_{\gamma,\delta}(t)$, due to
$
0\leq K_\delta(z)\leq1$,
we have
\begin{align}
|A_{\gamma,\delta}(t)|&\leq\E\int_0^t\int_{\mathbb T^N}\int_{\mathbb T^N}\rho_\gamma(x-y)|B_m(x,s)-B_m(y,s)|\dd x\dd y\dd s.
\label{eq:A-delta-epsilon-bound}
\end{align}
The right-hand side of \eqref{eq:A-delta-epsilon-bound} vanishes as $\gamma\downarrow0$ due to the continuity of spatial translations of $B_m$ in $L^1(\Omega\times Q_T)$.
This yields
\begin{equation}
\label{eq:A-delta-epsilon-limit}
A_{\gamma,\delta}(t)\rightarrow0\qquad\text{as }\gamma\downarrow0.
\end{equation}

Let us now focus on the term $R_{\gamma,\delta}(t)$.
Set
\begin{align*}
Q_{\gamma,\delta}(x,s):=\int_{\mathbb T^N}\rho_\gamma(x-y)\bigg|K_\delta\bigl(u_n(x,s)-u_m(y,s)\bigr)-\one_{\{u_n(x,s)>u_m(x,s)\}}\bigg|\dd y.
\end{align*}
Then
\begin{align}
|R_{\gamma,\delta}(t)|&\leq\E\int_0^t\int_{\mathbb T^N}Q_{\gamma,\delta}(x,s)|\widetilde B_n^{(m)}(x,s)-B_m(x,s)|\dd x\dd s.\label{eq:R-delta-epsilon-bound}
\end{align}

Since $u_m(\omega,\cdot,s)\in L^1(\mathbb T^N)$, we fix $(\omega,x,s)$ in a set of full measure in $\Omega\times\mathbb T^N\times[0,T]$ such that $x$ is a Lebesgue point of $u_m(\omega,\cdot,s)$.
We consider the limit of $Q_{\gamma,\delta}(x,s)$ in the following three cases.

\medskip

\noindent\textbf{Case 1:}
Suppose that $u_n(x,s)-u_m(x,s)>0$. 
We set
\[
\eta:=\frac{u_n(x,s)-u_m(x,s)}{2}>0.
\]
Then, when $|u_m(y,s)-u_m(x,s)|\leq\eta$, we have
\begin{align}
u_n(x,s)-u_m(y,s)&=u_n(x,s)-u_m(x,s)+u_m(x,s)-u_m(y,s)\geq\eta.
\label{eq:case-one-lower-bound}
\end{align}
Therefore, the monotonicity of $K_\delta$ and \eqref{eq:case-one-lower-bound} imply
\begin{align*}
Q_{\gamma,\delta}(x,s)&=\int_{\mathbb T^N}\one_{\{|u_m(y,s)-u_m(x,s)|\leq\eta\}}\rho_\gamma(x-y)\left|K_\delta(u_n(x,s)-u_m(y,s))-1\right|\dd y\\
&\quad+\int_{\mathbb T^N}\one_{\{|u_m(y,s)-u_m(x,s)|>\eta\}}\rho_\gamma(x-y)\left|K_\delta(u_n(x,s)-u_m(y,s))-1\right|\dd y\\
&\leq\int_{\mathbb T^N}\one_{\{|u_m(y,s)-u_m(x,s)|\leq\eta\}}\rho_\gamma(x-y)\left[1-K_\delta(u_n(x,s)-u_m(y,s))\right]\dd y\\
&\quad+\int_{\mathbb T^N}\one_{\{|u_m(y,s)-u_m(x,s)|>\eta\}}\rho_\gamma(x-y)\dd y\\
&\leq1-K_\delta(\eta)+\int_{\mathbb T^N}\one_{\{|u_m(y,s)-u_m(x,s)|>\eta\}}\rho_\gamma(x-y)\dd y\\
&=\int_\eta^{+\infty}\kappa_\delta(z)\dd z+\int_{\mathbb T^N}\one_{\{|u_m(y,s)-u_m(x,s)|>\eta\}}\rho_\gamma(x-y)\dd y.
\end{align*}
The first term tends to zero as $\delta\downarrow0$.
Moreover, the term
\begin{align*}
&\int_{\mathbb T^N}\one_{\{|u_m(y,s)-u_m(x,s)|>\eta\}}\rho_\gamma(x-y)\dd y\leq\frac{1}{\eta}\int_{\mathbb T^N}\rho_\gamma(x-y)|u_m(y,s)-u_m(x,s)|\dd y
\end{align*}
 tends to zero as $\gamma\downarrow0$, since $x$ is a Lebesgue point of $u_m(\omega,\cdot,s)$.
Consequently, we have
\begin{equation}\label{eq:limit-of-Qdeltaepsilon}
\lim_{\substack{ \delta=\gamma^{4/3},\gamma\downarrow0}}Q_{\gamma,\delta}(x,s)=0.
\end{equation}

\medskip

\noindent\textbf{Case 2:} Suppose that $u_n(x,s)-u_m(x,s)<0$. We set
\[
\eta:=-\frac{u_n(x,s)-u_m(x,s)}{2}=\frac{u_m(x,s)-u_n(x,s)}{2}>0.
\]
Following the procedure in Case 1, we obtain
\begin{align*}
Q_{\gamma,\delta}(x,s)\leq \int_\eta^{+\infty}\kappa_\delta(z)\dd z+\int_{\mathbb T^N}
\one_{\{|u_m(y,s)-u_m(x,s)|>\eta\}}\rho_\gamma(x-y)\dd y,
\end{align*}
where we use the evenness of $\kappa_\delta$.
Arguing as Case 1, both terms tend to zero. Therefore,
\eqref{eq:limit-of-Qdeltaepsilon} also holds in this case.

\medskip

\noindent\textbf{Case 3:}
Suppose that $u_n(x,s)=u_m(x,s)$.
Then we have
\[
\widetilde B_n^{(m)}(x,s)-B_m(x,s)=0.
\]
Moreover, since $0\leq K_\delta(z)\leq1$, we have
$
0\leq Q_{\gamma,\delta}(x,s)\leq2$.
Thus, 
\[
Q_{\gamma,\delta}(x,s)|\widetilde B_n^{(m)}(x,s)-B_m(x,s)|=0
\]

Combining the above Cases 1--3, we conclude that
\[
Q_{\gamma,\delta}(x,s)|\widetilde B_n^{(m)}(x,s)-B_m(x,s)|\rightarrow0
\]
for almost every $(\omega,x,s)$ as $\gamma\downarrow0$ and $\delta=\gamma^{4/3}$.

The a priori estimates imply that $|\widetilde B_n^{(m)}-B_m|\in L^1(\Omega\times Q_T)$.
Then the dominated convergence theorem and \eqref{eq:R-delta-epsilon-bound} yield
\begin{equation}\label{eq:R-delta-epsilon-limit}
\lim_{\substack{\gamma\downarrow0\\\delta=\gamma^{4/3}}}|R_{\gamma,\delta}(t)|=0.
\end{equation}

Combining \eqref{eq:A-delta-epsilon-limit} and
\eqref{eq:R-delta-epsilon-limit}, we obtain the desired convergence
\eqref{eq:limit-I-B-nm}:
\[
\lim_{\substack{\gamma\downarrow0\\\delta=\gamma^{4/3}}}I^B_{\gamma,\delta}(t)=I^B(t).
\]
We now estimate the limiting term $I^B(t)$.
On the set
\[
\{(x,s):u_n(x,s)>u_m(x,s)\},
\]
the fact that $b_m$ is non-increasing gives
\[
b_m\bigl(u_n(x,s)-\psi(x,s)\bigr)\leq b_m\bigl(u_m(x,s)-\psi(x,s)\bigr).
\]
Then, it follows that
\[
\widetilde B_n^{(m)}(x,s)-B_m(x,s)\leq0\qquad\text{on }\{u_n>u_m\},
\]
and consequently $I^B(t)\leq0$.
Thus, we deduce from \eqref{eq:def-I-B-nm} that
\[
\limsup_{\substack{\gamma\downarrow0\\\delta=\gamma^{4/3}}}I^\nu_{\gamma,\delta}(t)\leq0.
\]

Combining this estimate with the estimates of the flux, stochastic, and kinetic-measure terms in the standard doubling-of-variables inequality, we obtain
\[
\E\bigl\|(u_n(t)-u_m(t))^+\bigr\|_{L^1(\mathbb T^N)}\leq\bigl\|(u_n^0-u_m^0)^+\bigr\|_{L^1(\mathbb T^N)}
\]
for almost every $t\in(0,T)$. 
This completes the proof.
\end{proof}

\subsection{Proof of existence of kinetic solutions to the original obstacle problem }

\begin{thm}[Existence of a kinetic solution]
\label{thm:existence-obstacle-solution}
Let \textbf{Hypotheses (H1) and (H2)} hold. Then the lower obstacle problem \eqref{eq:obstacle-problem-existence} admits a kinetic solution $(u,\nu)$ in the sense of Definition \ref{def:dfn-1}.
\end{thm}
\begin{proof}
By Lemma \ref{lem:comparison-penalty-parameter}, the sequence $(u_n)_{n\geq1}$ is non-decreasing. 
We therefore define
\[
u:=\lim_{n\to\infty}u_n=\sup_{n\geq1}u_n\qquad\text{a.e. in }\Omega\times Q_T.
\]
Let $1\leq p<r<\infty$. 
The uniform estimate \eqref{eq:priori-p-n} gives
\[
\sup_{n\geq1}\|u_n\|_{L^r(\Omega\times Q_T)}<\infty.
\]
Hence $(|u_n|^p)_{n\geq1}$ is uniformly integrable, and Vitali's theorem yields
\begin{equation}\label{eq:strong-convergence-penalty-limit}
u_n\longrightarrow u\quad\text{strongly in }L^p(\Omega\times Q_T)
\end{equation}
for every $1\leq p<\infty$. 
Since every $u_n$ is predictable, the limit $u$ has a predictable representative. 
Moreover, Fatou's lemma and \eqref{eq:priori-p-n} give the moment bound required in Definition~\ref{def:dfn-1}.

The obstacle constraint follows from \eqref{eq:uniform Lp nu-n}. 
Indeed,
\begin{align*}
\E\int_{Q_T}(u-\psi)^-\dd x\dd t&\leq\E\int_{Q_T}|u-u_n|\dd x\dd t+\E\int_{Q_T}(u_n-\psi)^-\dd x\dd t\\
&=\E\int_{Q_T}|u-u_n|\dd x \dd t+\frac1n\E\nu_n(\mathbb T^N\times[0,T])\rightarrow0.
\end{align*}
Thus
\[
u\geq\psi\qquad\text{a.e. in }\Omega\times Q_T.
\]

We next extract the reflection and kinetic measures. 
Set
\[
K_r:=\mathbb T^N\times[0,T]\times[-r,r],\qquad\mathcal M_r:=C(K_r)^*.
\]
By \eqref{eq:uniform Lp nu-n} with $p=2$, the sequence $(\nu_n)_{n\geq1}$ is bounded in $L^2\bigl(\Omega;\mathcal M(\mathbb T^N\times[0,T])\bigr)$.
Moreover, the estimates obtained from \eqref{eq:priori-p_L2omega}, with a sufficiently large exponent and then passed through the vanishing-viscosity limit, give, for every $p\geq0$,
\begin{equation*}
\sup_{n\geq1}\E\bigg[\bigg(\int_{\mathbb T^N\times[0,T]\times\mathbb R}(1+|\xi|^p)\,\dd q_n\bigg)^2\bigg]<\infty.
\end{equation*}

By the uniform bounds for $\nu_{n_j}$ and $q_{n_j}$, the local weak-star compactness and diagonal argument of \cite[Subsection 4.1.2]{DV10-publish} yield, after extracting a subsequence, nonnegative random Radon measures $\nu$ and $q$ such
that
\[
\E\bigl[X\nu_{n_j}(\varphi)\bigr]\rightarrow\E\bigl[X\nu(\varphi)\bigr]
\]
for every $X\in L^2(\Omega)$ and $\varphi\in C(\mathbb T^N\times[0,T])$, and
\begin{equation}\label{eq:q-n-local-weak-star-limit}
\E\bigl[Xq_{n_j}(\Phi)\bigr]\rightarrow\E\bigl[Xq(\Phi)\bigr]
\end{equation}
for every $X\in L^2(\Omega)$ and $\Phi\in C_b(\mathbb T^N\times[0,T]\times\mathbb R)$.
Here, the convergence for $q_{n_j}$ follows first for test functions with compact support in $\xi$, and is extended to bounded continuous test functions by the weighted estimate, exactly as in \cite[Subsection 4.1.2]{DV10-publish}. 
The same estimate also implies that $q$ vanishes for large $\xi$.

We now verify the predictability of the limiting measures. 
We focus on $q$, because the proof for $\nu$ is identical, without the kinetic variable. 
Fix
$\phi\in C_b(\mathbb T^N\times\mathbb R)$ and set
\[
Q^\phi(t):=\int_{\mathbb T^N\times[0,t]\times\mathbb R}\phi(x,\xi)\,\dd q(x,s,\xi).
\]
We first show that $Q^\phi$ is adapted. 
Fix $t<T$ and choose $\chi_m\in C([0,T])$ such that
\[
0\leq\chi_m\leq1,\qquad\chi_m=1\ \text{on }[0,t],\qquad\chi_m=0\ \text{on }[t+m^{-1},T].
\]
Since $q_{n_j}$ is predictable, $q_{n_j}(\phi\chi_m)$ is $\mathcal F_{t+m^{-1}}$-measurable. 
The preceding weak convergence and the weak closedness of $L^2(\Omega,\mathcal F_{t+m^{-1}},\mathbb P)$ in $L^2(\Omega)$ imply that $q(\phi\chi_m)$ is also
$\mathcal F_{t+m^{-1}}$-measurable. 
Moreover,
\[
q(\phi\chi_m)\longrightarrow Q^\phi(t)\qquad\text{a.s.}
\]
by dominated convergence. 
Hence
\[
Q^\phi(t)\in\bigcap_{m\geq1}\mathcal F_{t+m^{-1}}=\mathcal F_t,
\]
where the last equality follows from the right-continuity of the filtration. 
Thus $Q^\phi$ is adapted.

For every $\omega$, the path $t\mapsto Q^\phi(t,\omega)$ is c\`adl\`ag, since it is the cumulative function of the finite signed measure $\phi q(\omega)$. 
It is therefore optional. 
Since $(\mathcal F_t)$ is the usual augmented Brownian filtration, it is a continuous filtration and its optional and predictable $\sigma$-fields coincide. 
Consequently, $Q^\phi$ is predictable.
The same argument shows that
\[
t\longmapsto\int_{\mathbb T^N\times[0,t]}\varphi(x)\,\dd\nu(x,s)
\]
is predictable for every $\varphi\in C_b(\mathbb T^N)$.
Therefore $q$ is a kinetic measure and $\nu$ is a predictable reflection measure.

It remains to pass to the limit in \eqref{eq:kinetic-penalized-obstacle-form}. 
If
\[
f_n=\one_{\{u_n>\xi\}},\qquad f=\one_{\{u>\xi\}},
\]
then \eqref{eq:strong-convergence-penalty-limit} and
\[
\int_{\mathbb R}|f_n-f|\,\dd\xi=|u_n-u|
\]
give
\[
f_n\rightarrow f\quad\text{strongly in }L^1(\Omega\times Q_T\times\mathbb R).
\]
This yields the convergence of the time and transport terms.

By \eqref{eq:assumption for g Lip},
\begin{align*}
&\E\int_0^T\|\Phi(u_n)-\Phi(u)\|_{\mathcal L_2(U;H)}^2\,\dd t\leq D_0\E\int_{Q_T}|u_n-u|^2\,\dd x\dd t\rightarrow0.
\end{align*}
Hence the stochastic integrals converge in $L^2(\Omega)$ by the It\^o isometry. 
The It\^o correction term converges by \eqref{eq:strong-convergence-penalty-limit}, the continuity and growth assumptions on $G^2$, and uniform integrability.

For the measure terms, since $\partial_\xi\varphi$ has compact support in $\xi$, \eqref{eq:q-n-local-weak-star-limit} gives
\[
q_n(\partial_\xi\varphi)\rightarrow q(\partial_\xi\varphi)\quad\text{weakly in }L^2(\Omega).
\]
Moreover, the continuity of $\psi$ implies that
\[
(x,t)\longmapsto\varphi(x,t,\psi(x,t))
\]
is continuous, and therefore
\[
\int_{\mathbb T^N\times[0,T)}\varphi(x,t,\psi(x,t))\,\dd\nu_n(x,t)
\rightarrow\int_{\mathbb T^N\times[0,T)}\varphi(x,t,\psi(x,t))\,\dd\nu(x,t)
\]
weakly in $L^2(\Omega)$.

Passing to the limit in \eqref{eq:kinetic-penalized-obstacle-form}, we obtain \eqref{eq:kinetic eqution}. 
Thus $(u,\nu)$ satisfies all the conditions of Definition~\ref{def:dfn-1}, except possibly the weak initial trace condition. 
This condition is established in Lemma~\ref{lem:weak-initial-condition}. 
Hence $(u,\nu)$ is a kinetic solution of the lower obstacle problem.
\end{proof}

\begin{lem}[Weak initial trace condition]\label{lem:weak-initial-condition}
Let $(u,\nu)$ be the pair constructed in the proof of Theorem \ref{thm:existence-obstacle-solution}.
Then, we have
\begin{equation*}
\lim_{\tau\downarrow0}\frac1{\tau}\int_0^\tau\Vert u(t)-u^0\Vert_{L^2(\mathbb T^N)}^2\dd t=0,\qquad \mathbb P\text{-a.s.,}
\end{equation*}
and therefore,
\begin{equation*}
\lim_{\tau\downarrow0}\frac1{\tau}\int_0^\tau\Vert u(t)-u^0\Vert_{L^1(\mathbb T^N)}\dd t=0,\qquad \mathbb P\text{-a.s.,}
\end{equation*}
which means that the pair $(u,\nu)$ satisfies the weak initial trace condition in Definition \ref{def:dfn-1}.
\end{lem}

\begin{proof}
We write
\[
\chi_u(x,t,\xi):=\one_{\{0<\xi<u(x,t)\}}-\one_{\{u(x,t)<\xi<0\}}=\one_{\{u(x,t)>\xi\}}-\one_{\{0>\xi\}},
\]
and
\[
\chi_{u^0}(x,\xi):=\one_{\{0<\xi<u^0(x)\}}-\one_{\{u^0(x)<\xi<0\}}.
\]

Let $\rho_{\gamma}$ be a standard nonnegative spatial mollifier on $\mathbb{T}^N$ for $\gamma\in(0,1)$.
Note that
\begin{align}\label{eq:initial-spatial-splitting}
&\frac1{\tau}\int_0^\tau\int_{\mathbb T^N}|u(x,t)-u^0(x)|^2\dd x\dd t\\
&\quad\leq\frac2{\tau}\int_0^\tau\int_{(\mathbb T^N)^2}|u(x,t)-u^0(y)|^2\rho_{\gamma}(x-y)\dd x\dd y\dd t\nonumber\\
&\qquad+2\int_{(\mathbb T^N)^2}|u^0(x)-u^0(y)|^2\rho_{\gamma}(x-y)\dd x\dd y.\nonumber
\end{align}
For the second term on the right hand side of \eqref{eq:initial-spatial-splitting}, since $u^0\in L^2(\mathbb T^N)$, the continuity of translations in $L^2(\mathbb T^N)$ yields
\begin{equation}\label{eq:translation-error-zero}
\lim_{\gamma\downarrow0}\int_{(\mathbb T^N)^2}|u^0(x)-u^0(y)|^2\rho_{\gamma}(x-y)\dd x\dd y=0.
\end{equation}

For every sufficiently small $\tau>0$, let $\Upsilon_{\tau}\in C_c^1([0,T))$ be nonnegative and non-increasing, with
\begin{equation*}
\Upsilon_{\tau}(0)=2,\qquad 0\leq\Upsilon_{\tau}\leq2\one_{[0,3\tau]},\qquad -\frac1{\tau}\leq\Upsilon_{\tau}'\leq0,
\end{equation*}
and
\begin{equation*}
\Upsilon_{\tau}'(t)=-\frac1{\tau},\qquad\text{for } t\in[0,\tau].
\end{equation*}

Using the definition of $\Upsilon_\tau$, we have
\begin{align}\label{eq:time-average-by-gamma}
&\frac1{\tau}\int_0^\tau\int_{(\mathbb T^N)^2}|u(x,t)-u^0(y)|^2\rho_{\gamma}(x-y)\dd x\dd y\dd t\\
&\quad\leq-\int_0^T\int_{(\mathbb T^N)^2}\Upsilon_{\tau}'(t)|u(x,t)-u^0(y)|^2\rho_{\gamma}(x-y)\dd x\dd y\dd t.\nonumber
\end{align}
Note that for every $r\in\mathbb R$, one has
\begin{equation*}
|u(x,t)-r|^2-|r|^2=2\int_{\mathbb R}(\xi-r)\chi_u(x,t,\xi)\dd\xi.
\end{equation*}
Then the right-hand side of \eqref{eq:time-average-by-gamma} equals 
\begin{align}\label{eq:time-average-chi}
-2\int_0^T\int_{(\mathbb T^N)^2\times\mathbb R}\Upsilon_{\tau}'(t)(\xi-u^0(y))\chi_u(x,t,\xi)\rho_{\gamma}(x-y)\dd\xi\dd x\dd y\dd t+2\Vert u^0\Vert_{L^2(\mathbb T^N)}^2.
\end{align}

Let $\theta_R\in C_c^\infty(\mathbb R)$ be an even function satisfying \eqref{theta}, that is,
\begin{equation*}
0\leq\theta_R\leq1,\qquad\theta_R(\xi)=1\quad\text{for }|\xi|\leq R,\qquad
\theta_R(\xi)=0\quad\text{for }|\xi|\geq R+1,\qquad|\theta_R'|\leq C.
\end{equation*}
Define
\begin{align*}
\mathcal J_{\tau,\gamma,R}:=-\int_0^T\int_{(\mathbb T^N)^2\times\mathbb R}\Upsilon_{\tau}'(t)\theta_R(\xi)(\xi-u^0(y))\chi_u(x,t,\xi)\rho_{\gamma}(x-y)\dd\xi\dd x\dd y\dd t.
\end{align*}
Based on \eqref{eq:time-average-by-gamma} and \eqref{eq:time-average-chi}, we have
\begin{align}\label{eq:disperse-of-u}
\frac1{\tau}\int_0^\tau\int_{(\mathbb T^N)^2}|u(x,t)-u^0(y)|^2\rho_{\gamma}(x-y)\dd x\dd y\dd t
\leq2\mathcal J_{\tau,\gamma,R}+2\Vert u^0\Vert_{L^2(\mathbb T^N)}^2+\mathcal T_{\tau,R},
\end{align}
where
\begin{align*}
\mathcal T_{\tau,R}:=2\int_0^T\int_{(\mathbb T^N)^2\times\mathbb R}|\Upsilon_{\tau}'(t)|\big(1-\theta_R(\xi)\big)|\xi-u^0(y)||\chi_u(x,t,\xi)|\rho_{\gamma}(x-y)\dd\xi\dd x\dd y\dd t.
\end{align*}
Note that using the support of $\chi_u$, we have for $p>2$,
\begin{align*}
\mathcal T_{\tau,R}&\leq C\sup_{t\in[0,T]}\int_{\mathbb T^N}\big(|u(x,t)|^2+|u(x,t)|\big)\one_{\{|u(x,t)|>R\}}\dd x\\ \notag
&\leq CR^{2-p}\big(1+\sup_{t\in[0,T]}\Vert u(t)\Vert_{L^p(\mathbb T^N)}^p\big).\nonumber
\end{align*}
Using \eqref{eq:usual moment bound}, we have
\begin{equation}\label{eq:velocity-tail-zero}
\lim_{R\to\infty}\sup_{0<\tau<T/3}\mathcal T_{\tau,R}=0,\qquad \mathbb P\text{-a.s.}
\end{equation}

Now, we focus on the estimate of $\mathcal J_{\tau,\gamma,R}$.
For fixed $y\in\mathbb T^N$, taking a smooth approximation of
\begin{equation*}
\varphi_{\tau,\gamma,R,y}(x,t,\xi):=\Upsilon_{\tau}(t)\theta_R(\xi)(\xi-u^0(y))\rho_{\gamma}(x-y)
\end{equation*}
as a test function in the kinetic formulation \eqref{eq:kinetic eqution} for $u$, and then integrating with respect to $y$, we reach
\begin{equation}\label{eq:initial-decomposition-lower}
\mathcal J_{\tau,\gamma,R}=I_0^{\gamma,R}+I_A^{\tau,\gamma,R}+I_\Phi^{\tau,\gamma,R}+I_q^{\tau,\gamma,R}+I_\nu^{\tau,\gamma,R}+I_M^{\tau,\gamma,R},
\end{equation}
where
\begin{align*}
I_0^{\gamma,R}&:=\Upsilon_{\tau}(0)\int_{(\mathbb T^N)^2\times\mathbb R}\chi_{u^0}(x,\xi)\theta_R(\xi)(\xi-u^0(y))\rho_{\gamma}(x-y)\dd\xi\dd x\dd y,\\
I_A^{\tau,\gamma,R}&:=\int_0^{3\tau}\int_{(\mathbb T^N)^2\times\mathbb R}\Upsilon_{\tau}(t)\chi_u(x,t,\xi)a(\xi)\cdot\nabla_x\rho_{\gamma}(x-y)\theta_R(\xi)(\xi-u^0(y))\dd\xi\dd x\dd y\dd t,\\
I_\Phi^{\tau,\gamma,R}&:=\frac12\int_0^{3\tau}\int_{(\mathbb T^N)^2}\Upsilon_{\tau}(t)G^2(x,u(x,t))\rho_{\gamma}(x-y)\partial_\xi\big[\theta_R(\xi)(\xi-u^0(y))\big]\big|_{\xi=u(x,t)}\dd x\dd y\dd t,\\
I_q^{\tau,\gamma,R}&:=-\int_{\mathbb T^N\times[0,3\tau]\times\mathbb R}\int_{\mathbb T^N}\Upsilon_{\tau}(t)\rho_{\gamma}(x-y)\partial_\xi\big[\theta_R(\xi)(\xi-u^0(y))\big]\dd y\dd q(x,t,\xi),\\
I_\nu^{\tau,\gamma,R}&:=\int_{\mathbb T^N\times[0,3\tau]}\int_{\mathbb T^N}\Upsilon_{\tau}(t)\theta_R(\psi(x,t))\big(\psi(x,t)-u^0(y)\big)\rho_{\gamma}(x-y)\dd y\dd\nu(x,t),
\end{align*}
and
\begin{align*}
I_M^{\tau,\gamma,R}&:=\sum_{k\geq1}\int_0^{3\tau}\int_{(\mathbb T^N)^2}\Upsilon_{\tau}(t)\theta_R(u(x,t))\big(u(x,t)-u^0(y)\big)\\
&\hspace{9em}\times\rho_{\gamma}(x-y)g_k(x,u(x,t))\dd x\dd y\dd\beta_k(t).
\end{align*}
We will treat each term on the right-hand side of \eqref{eq:initial-decomposition-lower}.
Fix $\gamma\in(0,1)$ and $R>1$. Since $\theta_R$ is compactly supported, we have almost surely
\begin{equation}\label{eq:IA-lower-zero}
|I_A^{\tau,\gamma,R}|\leq C_{\gamma,R}\int_0^{3\tau}\int_{\mathbb T^N}1\,\dd x\dd t\rightarrow0,\text{ as }\tau\downarrow0.
\end{equation}
Similarly, the growth assumption on $G^2$ in \textbf{Hypothesis (H1)} gives
\begin{equation*}
\lim_{\tau\downarrow0}I_\Phi^{\tau,\gamma,R}=0,\qquad \mathbb P\text{-a.s.}
\end{equation*}
For $I_M^{\tau,\gamma,R}$, fix $\gamma$ and $R$ and define a continuous process
\begin{align*}
M_{\gamma,R}(t)&:=\sum_{k\geq1}\int_0^t\int_{(\mathbb T^N)^2}\theta_R(u(x,s))\big(u(x,s)-u^0(y)\big)\rho_{\gamma}(x-y)g_k(x,u(x,s))\dd x\dd y\dd\beta_k(s),
\end{align*}
which satisfies $M_{\gamma,R}(0)=0$. Then, integration by parts in time gives
\[
I_M^{\tau,\gamma,R}=\int_0^{3\tau}\Upsilon_{\tau}(t)\dd M_{\gamma,R}(t)=-\int_0^{3\tau}\Upsilon_{\tau}'(t)M_{\gamma,R}(t)\dd t.
\]
Therefore, we have almost surely
\begin{equation*}
|I_M^{\tau,\gamma,R}|\leq C\sup_{0\leq t\leq3\tau}|M_{\gamma,R}(t)|\to0,\qquad\text{as }\tau\downarrow0.
\end{equation*}

For $I_\nu^{\tau,\gamma,R}$, based on the fact that $u^0(\cdot)\geq\psi(\cdot,0)$ a.e. on $\mathbb{T}^N$, we have
\begin{align*}
\psi(x,t)-u^0(y)\leq\psi(x,t)-\psi(y,0)\leq|\psi(x,t)-\psi(y,0)|.
\end{align*}
Then, on the support of $\rho_{\gamma}(x-y)$ and $\Upsilon_{\tau}(t)$, one has
\[
\psi(x,t)-u^0(y)\leq\omega_\psi(C\gamma+3\tau), \quad\text{for a.e. }y\in\mathbb{T}^N
\]
where $\omega_\psi$ is a modulus of continuity of $\psi$ on $\mathbb T^N\times[0,T]$.
Combining the boundedness of $\Upsilon_{\tau}$ and $\theta_R$, we obtain almost surely
\begin{align}\label{eq:lower-obstacle-initial-limsup}
I_\nu^{\tau,\gamma,R}\leq C\omega_\psi(C\gamma+3\tau)\nu\big(\mathbb T^N\times[0,3\tau]\big)\leq C\omega_\psi(C\gamma+3\tau)\nu\big(\mathbb T^N\times[0,T)\big).
\end{align}

For $I_q^{\tau,\gamma,R}$, we write
\begin{equation*}
I_q^{\tau,\gamma,R}=I_{q,1}^{\tau,\gamma,R}+I_{q,2}^{\tau,\gamma,R},
\end{equation*}
where
\begin{align*}
I_{q,1}^{\tau,\gamma,R}&:=-\int_{\mathbb T^N\times[0,3\tau]\times\mathbb R}\int_{\mathbb T^N}\Upsilon_{\tau}(t)\theta_R'(\xi)(\xi-u^0(y))\rho_{\gamma}(x-y)\dd y\dd q(x,t,\xi),\\
I_{q,2}^{\tau,\gamma,R}&:=-\int_{\mathbb T^N\times[0,3\tau]\times\mathbb R}\int_{\mathbb T^N}\Upsilon_{\tau}(t)\theta_R(\xi)\rho_{\gamma}(x-y)\dd y\dd q(x,t,\xi).
\end{align*}
Since $q$, $\Upsilon_{\tau}$, $\theta_R$, and $\rho_{\gamma}$ are nonnegative, we have
\begin{equation}\label{eq:Iq2-lower-negative}
I_{q,2}^{\tau,\gamma,R}\leq0.
\end{equation}
On the other hand, set
\[
S_R:=\{\xi\in\mathbb R:R\leq|\xi|\leq R+1\}.
\]
Since $\theta_R'$ is supported in $S_R$ and $u^0\in L^\infty(\mathbb T^N)$, we obtain
\begin{align*}
|I_{q,1}^{\tau,\gamma,R}|&\leq C\big(R+1+\Vert u^0\Vert_{L^\infty(\mathbb T^N)}\big)q\big(\mathbb T^N\times[0,T]\times S_R\big)\\
&\leq R^{1-p}\int_{\mathbb T^N\times[0,T]\times S_R}|\xi|^p\dd q(x,t,\xi)+C(1+\Vert u^0\Vert_{L^\infty(\mathbb T^N)})q\big(\mathbb T^N\times[0,T]\times S_R\big)\\
&\rightarrow0, \text{  as }R\to\infty,
\end{align*}
This, together with \eqref{eq:Iq2-lower-negative}, implies that
\begin{equation}\label{eq:Iq-lower-limsup}
\limsup_{\gamma\downarrow0}\limsup_{R\to\infty}\limsup_{\tau\downarrow0}I_q^{\tau,\gamma,R}\leq0.
\end{equation}

For $I_0^{\gamma,R}$, using the dominated convergence theorem, we have
\begin{align*}
\lim_{R\to\infty}2 I_0^{\gamma,R}
&=4\lim_{R\to\infty}\int_{(\mathbb T^N)^2\times\mathbb R}\chi_{u^0}(x,\xi)\theta_R(\xi)(\xi-u^0(y))\rho_{\gamma}(x-y)\dd\xi\dd x\dd y\\
&=2\int_{(\mathbb T^N)^2}\Big(|u^0(x)-u^0(y)|^2-|u^0(y)|^2\Big)\rho_{\gamma}(x-y)\dd x\dd y,
\end{align*}
which indicates
\begin{align}\label{eq:I0-lower-cancellation}
&\lim_{R\to\infty}\Big[2I_0^{\gamma,R}+2\Vert u^0\Vert_{L^2(\mathbb T^N)}^2\Big]=2\int_{(\mathbb T^N)^2}|u^0(x)-u^0(y)|^2\rho_{\gamma}(x-y)\dd x\dd y.
\end{align}

Combining \eqref{eq:disperse-of-u} and \eqref{eq:initial-decomposition-lower}, by using \eqref{eq:velocity-tail-zero}, \eqref{eq:IA-lower-zero}-\eqref{eq:lower-obstacle-initial-limsup}, \eqref{eq:Iq-lower-limsup}, and \eqref{eq:I0-lower-cancellation}, we obtain almost surely
\begin{align}\label{eq:cross-initial-lower-final}
&\limsup_{\tau\downarrow0}\frac1{\tau}\int_0^\tau\int_{(\mathbb T^N)^2}|u(x,t)-u^0(y)|^2\rho_{\gamma}(x-y)\dd x\dd y\dd t\\
&\leq C\int_{(\mathbb T^N)^2}|u^0(x)-u^0(y)|^2\rho_{\gamma}(x-y)\dd x\dd y+C\omega_\psi(C\gamma)\nu\big(\mathbb T^N\times[0,T)\big)\nonumber
\end{align}
where we have first let $\tau\downarrow0$ for fixed $\gamma$ and $R$, and then $R\to\infty$.

Substituting \eqref{eq:cross-initial-lower-final} into \eqref{eq:initial-spatial-splitting}, we obtain almost surely
\begin{align*}
&\limsup_{\tau\downarrow0}\frac1{\tau}\int_0^\tau\Vert u(t)-u^0\Vert_{L^2(\mathbb T^N)}^2\dd t\\
&\leq C\int_{(\mathbb T^N)^2}|u^0(x)-u^0(y)|^2\rho_{\gamma}(x-y)\dd x\dd y+C\omega_\psi(C\gamma)\nu\big(\mathbb T^N\times[0,T)\big).
\end{align*}
Letting $\gamma\downarrow0$, using \eqref{eq:translation-error-zero}, the continuity of $\psi$, and the finiteness of $\nu$, we have almost surely
\[
\lim_{\tau\downarrow0}\frac1{\tau}\int_0^\tau\Vert u(t)-u^0\Vert_{L^2(\mathbb T^N)}^2\dd t=0,
\]
which completes the proof.
\end{proof}
Combining Theorems \ref{thm:uniqueness} and \ref{thm:existence-obstacle-solution}, we conclude the well-posedness of the lower obstacle problem for \eqref{eq:intro-reflected-model}, in the sense of Definition \ref{def:dfn-1}. It reads as follows.
\begin{cor}\label{cor:well-posedness}
Let \textbf{Hypotheses (H1) and (H2)$'$} hold.
Then the lower obstacle problem \eqref{eq:obstacle-problem-existence} admits a unique kinetic solution $(u,\nu)$ in the sense of Definition \ref{def:dfn-1}.
\end{cor}

\begin{rem}[Regularity of the obstacle]
\label{rem:obstacle-regularity}
The different assumptions on the obstacle in the existence and uniqueness results arise from different parts of the argument.
For existence, the continuity of $\psi$ is sufficient: on the compact set $[0,T]\times\mathbb T^N$, it provides the uniform continuity needed to pass to the limit in the penalization and reflection terms.

By contrast, the stronger spatial regularity in \textbf{Hypothesis (H2)$'$} is used in the doubling-of-variables argument for uniqueness. 
More precisely, the evenness of $\rho_\gamma$ and the oddness of $\llbracket\kappa_\delta\rrbracket$ cancel the linear part of the Taylor expansion of $\psi$, while the second-order remainder gives
\[
\left|\int_{\mathbb T^N}\rho_\gamma(x-y)\llbracket\kappa_\delta\rrbracket
\bigl(\psi(x,t)-\psi(y,t)\bigr)\dd y\right|\leq C\|D_x^2\psi\|_{L^\infty}\frac{\gamma^2}{\delta}.
\]
In the present proof, this is the only step that requires second spatial derivatives of the obstacle. 
It would therefore be interesting to determine whether the uniqueness result remains valid under weaker spatial regularity by replacing this cancellation estimate with a different argument.
\end{rem}

\section{Large deviation principle for the obstacle problem}
We begin with a brief review of large deviation theory.
Let $(X^\varepsilon)_{\varepsilon>0}$ be a family of random variables defined on a given probability space $(\Omega, \mathcal{F}, \mathbb{P})$ taking values in some Polish space $\mathcal{E}$.
\begin{defn}[Rate function] A function $I: \mathcal{E}\rightarrow [0,\infty]$ is called a rate function if $I$ is lower semicontinuous.
A rate function $I$ is called a good rate function if the level set $\{x\in \mathcal{E}: I(x)\leq M\}$ is compact for each $M<\infty$.
\end{defn}
\begin{defn}[Large deviation principle]
The family $\{X^\varepsilon\}$ is said to satisfy the large deviation principle with rate function $I$ if for each Borel subset $A$ of $\mathcal{E}$
      \[
      -\inf_{x\in A^o}I(x)\leq \liminf_{\varepsilon\rightarrow 0}\varepsilon \log \mathbb{P}(X^\varepsilon\in A)\leq \limsup_{\varepsilon\rightarrow 0}\varepsilon \log \mathbb{P}(X^\varepsilon\in A)\leq -\inf_{x\in \bar{A}}I(x),
      \]
      where $A^o$ and $\bar{A}$ denote the interior and closure of $A$ in $\mathcal{E}$, respectively.
\end{defn}

Suppose $W(t)$ is a cylindrical Wiener process on a Hilbert space $U$ defined on a filtered probability space $(\Omega, \mathcal{F},(\mathcal{F}_t)_{t\in [0,T]}, \mathbb{P})$ (that is, the paths of $W$ take values in $C([0,T];\mathcal{U})$, where $\mathcal{U}$ is another Hilbert space such that the embedding $U\subset \mathcal{U}$ is Hilbert-Schmidt).

Now we define
\begin{align*}
&\mathcal{A}:=\{\phi: \phi\ {\rm{is\ a\ } } U\text{-}{\rm{valued}}\ \{\mathcal{F}_t\}\text{-}{\rm{predictable\ process\ such\ that}} \ \int^T_0 |\phi(s)|^2_U\,\dd s<\infty\ \mathbb{P}\text{-}a.s.\};\\
&S_M:=\{ h\in L^2([0,T];U): \int^T_0 |h(s)|^2_U\,\dd s\leq M\};\\
&\mathcal{A}_M:=\{\phi\in \mathcal{A}: \phi(\omega)\in S_M,\ \mathbb{P}\text{-}a.s.\}.
\end{align*}
Here and throughout the remainder of the paper, we will always refer to the weak topology on the set $S_M$.

Suppose for each $\varepsilon>0$, $\mathcal{G}^{\varepsilon}: C([0,T];\mathcal{U})\rightarrow \mathcal{E}$ is a measurable map and let $X^{\varepsilon}:=\mathcal{G}^{\varepsilon}(W)$. 
To establish a large deviation principle for the family $(X^{\varepsilon})$ as $\varepsilon\rightarrow 0$,
a new sufficient condition was proposed by  Matoussi, Sabbagh and Zhang in \cite{MSZ}, which can be used to verify assumptions in Budhiraja and Dupuis \cite{BD} (hence the large deviation principle). 
It turns out this new sufficient condition is suitable for establishing the large deviation principle for the scalar conservation laws.
\begin{description}
  \item[\textbf{Condition A} ] There exists a measurable map $\mathcal{G}^0: C([0,T];\mathcal{U})\rightarrow \mathcal{E}$ such that the following two items hold:
\end{description}
\begin{description}
  \item[(i)] For every $M<+\infty$, and for any family $\{h^{\varepsilon}; \varepsilon>0\}$ $\subset \mathcal{A}_M$ and any $\delta>0$,
      \[
      \lim_{\varepsilon\rightarrow 0}\mathbb{P}\Big(\rho(Y^\varepsilon, Z^\varepsilon)>\delta\Big)=0,
      \]
     where $Y^\varepsilon:=\mathcal{G}^{\varepsilon}\left(W(\cdot)+\frac{1}{\sqrt{\varepsilon}}\int^{\cdot}_{0}h^\varepsilon(s)\dd s\right)$, $Z^\varepsilon:=\mathcal{G}^0\left(\int^{\cdot}_{0}h^\varepsilon(s)\dd s\right)$,
     and $\rho(\cdot,\cdot)$ stands for the metric in the space $\mathcal{E}$.
  \item[(ii)] For every $M<+\infty$ and any family $\{h^\varepsilon; \varepsilon>0\}\subset S_M$ that weakly converges to some element $h$ as $\varepsilon\rightarrow 0$,
      $\mathcal{G}^0\left(\int^{\cdot}_{0}h^\varepsilon(s)\dd s\right)$ converges to $\mathcal{G}^0\left(\int^{\cdot}_{0}h(s)\dd s\right)$ in the space $\mathcal{E}$.
\end{description}

We also refer to Condition (ii) as the weak--strong continuity of the skeleton equation.

The following result is due to Budhiraja and Dupuis in \cite{BD}.
\begin{thm}\label{thm-7}
If $\{\mathcal{G}^{\varepsilon}\}$ satisfies {Condition A}, then $X^{\varepsilon}$ satisfies the large deviation principle on $\mathcal{E}$ with the
following good rate function $I$ defined by
\begin{equation}\label{equ-27-1}
I(f)=\inf_{\{h\in L^2([0,T];U): f= \mathcal{G}^0(\int^{\cdot}_{0}h(s)\dd s)\}}\Big\{\frac{1}{2}\int^T_0|h(s)|^2_{U}\dd s\Big\},\ \ \forall f\in\mathcal{E}.
\end{equation}
By convention, $I(f)=\infty$, if  $\Big\{h\in L^2([0,T];U): f= \mathcal{G}^0(\int^{\cdot}_{0}h(s)\dd s)\Big\}=\emptyset.$
\end{thm}

\subsection{Statement of large-deviations result}
We are concerned with the obstacle problem for stochastic conservation laws driven by small noise \eqref{eq:intro-small-noise} for $\varepsilon>0$ and $u^0\in L^\infty(\mathbb{T}^N)$. 
Under \textbf{Hypotheses (H1) and (H2)$'$}, by Corollary \ref{cor:well-posedness}, there exists a unique kinetic solution $u^\varepsilon$ and a Radon measure $\nu^{\varepsilon}$ to \eqref{eq:intro-small-noise} with initial data $u^0$ such that
\begin{equation*}
u^{\varepsilon}(x,t)\geq\psi(x,t),\quad \text{for }\mathbb{P}\otimes\dd x\otimes\dd t-a.s. (\omega,x,t),
\end{equation*}
\begin{align*}
\E\Big(\underset{0\leq t\leq T}{{\rm{ess\sup}}}\ \Vert u^\varepsilon(t)\Vert ^p_{L^p(\mathbb{T}^N)}\Big)\leq C_p(T,D_0),
\end{align*}
\begin{equation*}
\lim_{\tau\downarrow0}\frac{1}{\tau}\int_0^\tau\Vert u^{\varepsilon}(t)-u^0\Vert _{L^1(\mathbb{T}^N)}\dd t=0, \quad \text{a.s.},
\end{equation*}
and there exists a kinetic measure $q^{\varepsilon}$ such that $f:=\one_{u^{\varepsilon}>\xi}$ fulfills, for all $\varphi\in C^1_c(\mathbb{T}^N\times [0,T)\times \mathbb{R})$,
\begin{align*}\notag
&\int^T_0\langle f(t), \partial_t \varphi(t)\rangle \dd t+\langle f^0, \varphi(0)\rangle +\int^T_0\langle f(t), a(\xi)\cdot \nabla \varphi (t)\rangle \dd t\\
&= -\sqrt{\varepsilon}\sum_{k\geq 1}\int^T_0\int_{\mathbb{T}^N} \int_{\mathbb{R}}g_k(x,\xi)\varphi (x,t,\xi)\dd\mu^{\varepsilon}_{x,t}(\xi)\dd x\dd\beta_k(t) \\ \notag
&\quad -\frac{\varepsilon}{2}\sum_{k\geq1}\int^T_0\int_{\mathbb{T}^N}\int_{\mathbb{R}}\partial_{\xi}\varphi (x,t,\xi)G^2(x,\xi)\dd\mu^{\varepsilon}_{x,t}(\xi)\dd x\dd t\\
&\quad + q^\varepsilon(\partial_{\xi} \varphi)-\int_0^T\int_{\mathbb{T}^N}\varphi(x,t,\psi(x,t))\dd\nu^{\varepsilon}(x,t), \quad\ \text{a.s.} ,\nonumber
\end{align*}
where $\mu^{\varepsilon}_{x,t}(\dd\xi)=\delta_{u^{\varepsilon}(x,t)}(\dd\xi)$.
Therefore, there exists a Borel-measurable function
\[
\mathcal{G}^{\varepsilon}: C([0,T];\mathcal{U})\rightarrow L^1(0,T;L^1(\mathbb{T}^N))
\]
such that $u^{\varepsilon}(\cdot)=\mathcal{G}^{\varepsilon}(W(\cdot))$.

For $h\in L^2(0,T;U)$, the skeleton equation corresponding to \eqref{eq:intro-small-noise} is given by
\begin{equation}\label{eq:skeleton-obstacle-problem}
\left\{
\begin{aligned}
\partial_t u^h+\Div A(u^h)&=\Phi(u^h)h(t)+\nu^h,&&\text{in }\mathbb T^N\times(0,T),\\
u^h&\geq\psi,&&\text{a.e. in }\mathbb T^N\times(0,T),\\
u^h(0)&=u^0.
\end{aligned}
\right.
\end{equation}
Motivated by Definition \ref{def:dfn-1}, we propose the definition of kinetic solution to \eqref{eq:skeleton-obstacle-problem} as follows.
\begin{defn}\label{def:dfn-2}
A pair $(u^h,\nu^h)$ is called a kinetic solution of \eqref{eq:skeleton-obstacle-problem} with initial data $u^0$ and obstacle $\psi$ if the following conditions hold.
\begin{enumerate}[label=\textbf{\arabic*}.]
  \item For every $p\geq1$, $u^h\in L^\infty(0,T;L^p(\mathbb T^N))$ and
\begin{align*}
\underset{0\leq t\leq T}{{\rm{ess\,sup}}}\ \Vert u^h(t)\Vert ^p_{L^p(\mathbb{T}^N)} \leq C(p,\Vert h\Vert_{L^2(0,T;U)},T,D_0),
\end{align*}
\item  We have 
\begin{equation*}
u^h(x,t)\geq\psi(x,t),\qquad\text{for } \dd x\otimes\dd t\text{-a.e. } (x,t),
\end{equation*}
and $\nu^h$ is a nonnegative finite Radon measure on $\mathbb{T}^N\times[0,T)$.
\item  The function $u^h$ satisfies the weak initial trace condition
\begin{equation*}
\lim_{\tau\downarrow0}\frac{1}{\tau}\int_0^\tau\Vert u^h(t)-u^0\Vert _{L^1(\mathbb{T}^N)}\dd t=0.
\end{equation*}
\item There exists a (deterministic) kinetic measure $q^h$ such that, with $f:= \one_{\{u^h>\xi\}}$ and $\mu^h:=\delta_{u^h}$, for all $\varphi\in C^1_c(\mathbb{T}^N\times [0,T)\times \mathbb{R})$,
\begin{align}\notag
&\int^T_0\langle f(t), \partial_t \varphi(t)\rangle \dd t+\langle f^0, \varphi(0)\rangle +\int^T_0\langle f(t), a(\xi)\cdot \nabla \varphi (t)\rangle \dd t\\
&= -\sum_{k\geq 1}\int^T_0\int_{\mathbb{T}^N} \int_{\mathbb{R}}g_k(x,\xi)\varphi (x,t,\xi)h_k(t)\dd \mu^h_{x,t}(\xi)\dd x\dd t \label{eq:kinetic-skeleton-obstacle}\\
&\quad + q^h(\partial_{\xi} \varphi)-\int_0^T\int_{\mathbb{T}^N}\varphi(x,t,\psi(x,t))\dd\nu^h(x,t),\nonumber
\end{align}
 where  $\mu^h_{x,t}(\dd\xi)=\delta_{u^h(x,t)}(\dd\xi)$.
\end{enumerate}
\end{defn}

Well-posedness of kinetic solutions to \eqref{eq:skeleton-obstacle-problem} is proved by the Theorem \ref{thm:existence-uniqueness-skeleton-obstacle} below. 
Consequently, the solution $u^h$ induces a measurable mapping $\mathcal{G}^0: C([0,T];\mathcal{U})\rightarrow L^1(0,T;L^1(\mathbb{T}^N))$ so that  $\mathcal{G}^0\Big(\int^{\cdot}_0 h(s)\dd s\Big):=u^h(\cdot)$.

\smallskip

We are now ready to state our large-deviations result.

\begin{thm}\label{thm-3}
 Under \textbf{Hypotheses (H1) and (H2)$'$}, the family $\{u^{\varepsilon}\}_{\varepsilon>0}$  satisfies the large deviation principle on the space $L^1(0,T;L^1(\mathbb{T}^N))$, with the good rate function $I$ given by (\ref{equ-27-1}).
\end{thm}
\begin{proof}
  Based on  Theorem \ref{thm-7}, the large deviation principle follows from Propositions \ref{prp-3} and \ref{prp-1} below.
\end{proof}

\section{Weak--strong continuity of the skeleton equation }
The main purpose of this section is to prove the weak--strong continuity of the skeleton
equation  (\ref{eq:skeleton-obstacle-problem}), which is used to verify (ii) in condition A for large deviations. In the sequel, we fix $h\in S_M$ and write
\[
h(t)=\sum_{k\geq1}h_k(t)e_k,\qquad h_k(t):=\langle h(t),e_k\rangle_U .
\]

\subsection{Well-posedness of the skeleton equation}\label{sec-skeleton}

Similarly to the proof of the
existence of kinetic solutions to \eqref{eq:obstacle-problem-existence}, to prove  the well-posedness of the skeleton equation \eqref{eq:skeleton-obstacle-problem}, we need to introduce two auxiliary approximate equations.

Recall that $b_n(r):=nr^-$. For $\alpha\in(0,1)$ and $n\in\mathbb N$, consider the viscous penalized skeleton equation
\begin{equation}\label{P-3}
\text{\textbf{App Equ (i)}}\quad
\left\{
\begin{aligned}
\partial_t u^h_{\alpha,n}+\Div A(u^h_{\alpha,n})-\alpha\Delta u^h_{\alpha,n}
&=\Phi_\alpha(u^h_{\alpha,n})h(t)+b_n(u^h_{\alpha,n}-\psi),\\
u^h_{\alpha,n}(0)
&=u^0_{\alpha}.
\end{aligned}
\right.
\end{equation}
where $u^0_{\alpha}$, $\Phi_\alpha$ are the same as those in \eqref{eq:penalized-viscous}.
The weak solution to (\ref{P-3}) can be defined similarly to Definition \ref{def:penalized-viscous-solution}. 
Since $u^0\in L^\infty(\mathbb{T}^N)$, referring to \cite[Theorem 2.1]{GR00} again, we deduce that \eqref{P-3} has a unique $L^{p}(\mathbb{T}^N)$ solution.

Moreover, to obtain estimates uniform in $\alpha$ and $n$, it remains to control the term involving $h$.
Recall that  $M_\psi:=1+\Vert\psi\Vert_{L^\infty(Q_T)}$.
Under \textbf{Hypothesis (H1)}, by H\"{o}lder inequality and Young inequality, for every $p\geq 2$, this term can be estimated as follows:
\begin{align}\notag
&p\int_0^t\int_{\mathbb{T}^N}|u_{\alpha,n}^h-M_\psi|^{p-2}(u_{\alpha,n}^h-M_\psi)\sum_{k\geq1}g_{k,\alpha}(x,u^h_{\alpha,n})h_k(s)\dd x\dd s\\ \notag
&\leq p\int_0^t\int_{\mathbb{T}^N}|u_{\alpha,n}^h-M_\psi|^{p-1}\Big(\sum_{k\geq1}g^2_{k,\alpha}(x,u^h_{\alpha,n})\Big)^{\frac{1}{2}}\Big(\sum_{k\geq1}h^2_k(s)\Big)^{\frac{1}{2}}\dd x\dd s\\ \notag
&\leq p\int_0^t|h(s)|_U\int_{\mathbb{T}^N}|u_{\alpha,n}^h-M_\psi|^{p-1}(1+|u_{\alpha,n}^h|)\dd x\dd s\\ \notag
&\leq C(p)\int_0^t|h(s)|_U\|u_{\alpha,n}^h-M_\psi\|^{p}_{L^p(\mathbb{T}^N)}\dd s
+p(1+M_\psi)\int_0^t|h(s)|_U\|u_{\alpha,n}^h-M_\psi\|^{p-1}_{L^p(\mathbb{T}^N)}\dd s\\
\label{control term}
&\leq   C(p,M_\psi)\int_0^t|h(s)|_U\|u_{\alpha,n}^h-M_\psi\|^{p}_{L^p(\mathbb{T}^N)}\dd s+C(p,M_\psi)\int_0^t|h(s)|_U \dd s.
\end{align}
By Gr\"onwall's inequality,
we deduce that the results of Lemma \ref{lem:priori estimates} hold for $u^h_{\alpha,n}$ with the constant $C$ depending also on $\Vert h\Vert_{L^2(0,T;U)}$. That is, for every $p\geq2$, there exists $C_p>0$ independent of $\alpha\in(0,1)$ and $n\in\mathbb N$ such that
\begin{align}\label{eq:priori-p-1}
&\sup_{t\in[0,T]}\int_{\mathbb T^N}|u^h_{\alpha,n}(t)-M_\psi|^p\dd x+\alpha\int_0^T\int_{\mathbb T^N}|u^h_{\alpha,n}-M_\psi|^{p-2}|\nabla u^h_{\alpha,n}|^2\dd x\dd s\\
&\quad+\int_0^T\int_{\mathbb T^N}|u^h_{\alpha,n}-M_\psi|^{p-2}n|(u^h_{\alpha,n}-\psi)^-|^2\dd x\dd s+\int_0^T\int_{\mathbb T^N}|u^h_{\alpha,n}-M_\psi|^{p-2}n(u^h_{\alpha,n}-\psi)^-\dd x\dd s\nonumber\\
&\quad\leq C(p,D_0,M,T) \Big(1+\int_{\mathbb T^N}|u^0_{\alpha}-M_\psi|^p\dd x\Big),\nonumber
\end{align}
The weak solution $u^h_{\alpha,n}$ has a kinetic formulation.
Precisely, define
\begin{align*}
\dd\nu_{\alpha,n}^h(x,t)&:=n(u_{\alpha,n}^h-\psi)^-\dd x\dd t,\\
\dd\lambda_{\alpha,n}^h(x,t,\xi)&:=\int_0^1 n[(u_{\alpha,n}^h(x,t)-\psi(x,t))^-]^2\delta_{\psi+\theta(u_{\alpha,n}^h-\psi)}(\dd\xi)\dd\theta\,\dd x\dd t\nonumber\\
q_{\alpha,n}^h&:=m_{\alpha,n}^h+\lambda_{\alpha,n}^h,\nonumber
\end{align*}
where $\dd m_{\alpha,n}^h=\alpha|\nabla u_{\alpha,n}^h|^2\delta_{u_{\alpha,n}^h}(\dd\xi)\dd x\dd t$.
Then
$
f(x,t,\xi):=\one_{\{u^h_{\alpha,n}(x,t)>\xi\}}$ satisfies, for any $\varphi\in C^\infty_c(\mathbb{T}^N\times [0,T)\times \mathbb{R})$,
\begin{align*}
&\int_0^T \langle f(t),\partial_t\varphi(t)\rangle\dd t+\langle f_\alpha^0,\varphi(0)\rangle+\int_0^T\langle f(t),a(\xi)\cdot\nabla_x\varphi(t)+\alpha\Delta_x\varphi(t)\rangle\dd t\\
&=-\sum_{k\geq1}\int_0^T\int_{\mathbb T^N}g_{k,\alpha}(x,u^h_{\alpha,n})h_k(t)\varphi(x,t,u^h_{\alpha,n})\dd x\dd t\nonumber\\
&\quad+q^h_{\alpha,n}(\partial_\xi\varphi)-\int_0^T\int_{\mathbb T^N}\varphi(x,t,\psi(x,t))
\dd\nu^h_{\alpha,n}(x,t),\nonumber
\end{align*}
where $f_\alpha^0=\one_{\{u^0_\alpha(x)>\xi\}}$

With the aid of the above uniform estimates \eqref{eq:priori-p-1}, we have the following result.
\begin{cor}[Uniform bounds for the defect measures of $u^h_{\alpha,n}$]\label{cor:uniform-bounds-defect-measure-2}
For every $p\geq0$, there exists $C>0$ independent of $\alpha\in(0,1)$ and $n\in\mathbb N$, such that
\begin{align*}
\int_{\mathbb{T}^N\times[0,T]\times\mathbb{R}}(1+|\xi|^p)\dd q^h_{\alpha,n}\leq C(p,D_0,M,T).
\end{align*}
\end{cor}

By an argument similar to the proof of Lemma \ref{lem:vanishing-viscosity-fixed-epsilon} (but simpler), when the viscosity parameter $\alpha\rightarrow 0$, the solution $u^h_{\alpha,n}$ converges to the kinetic solution $u_n^h$ of
 the penalized skeleton equation:
 \begin{equation}\label{eq:penalized-skeleton}
\text{\textbf{App Equ (ii)}}\quad
\left\{
\begin{aligned}
\partial_t u_n^h+\Div A(u_n^h)
&=\sum_{k\geq1}g_k(x,u_n^h)h_k(t)+b_n(u_n^h-\psi),\\
u_n^h(0)
&=u^0.
\end{aligned}
\right.
\end{equation}
Moreover, let $m_n^h$ denote the kinetic measure obtained as the vanishing-viscosity limit of $m_{\alpha,n}^h$.
Then $f_n^h(x,t,\xi):=\one_{\{u_n^h(x,t)>\xi\}}$ fulfills, for any $\varphi\in C^1_c(\mathbb{T}^N\times [0,T)\times \mathbb{R})$,
\begin{align}\label{eq:kinetic-penalized-skeleton-obstacle-form}
&\int_0^T\langle f_n^h(t),\partial_t\varphi(t)\rangle\dd t+\langle f^0,\varphi(0)\rangle+\int_0^T\langle f_n^h(t),a(\xi)\cdot\nabla_x\varphi(t)\rangle\dd t\nonumber\\
&=-\sum_{k\geq1}\int_0^T\int_{\mathbb T^N}g_k(x,u_n^h)h_k(t)\varphi(x,t,u_n^h)\dd x\dd t\nonumber\\
&\quad+q_n^h(\partial_\xi\varphi)-\int_0^T\int_{\mathbb T^N}\varphi(x,t,\psi(x,t))\dd\nu_n^h(x,t),
\end{align}
where the measures are given by
\begin{align}
\dd\nu_n^h(x,t)&:=n(u_n^h-\psi)^-\dd x\dd t,\nonumber\\
\dd\lambda_n^h(x,t,\xi)&:=\int_0^1 n[(u_n^h(x,t)-\psi(x,t))^-]^2\delta_{\psi+\theta(u_n^h-\psi)}(\dd\xi)\dd\theta\,\dd x\dd t,\nonumber\\
q_n^h&:=m_n^h+\lambda_n^h.\nonumber
\nonumber
\end{align}
Furthermore, for every $p\geq2$, there exists $C_p>0$ independent of $n\in\mathbb N$ such that
\begin{align}\label{eq:priori-p-2}
&\sup_{t\in[0,T]}\int_{\mathbb T^N}|u^h_{n}(t)-M_\psi|^p\dd x+\int_0^T\int_{\mathbb T^N}|u^h_{n}-M_\psi|^{p-2}n|(u^h_{n}-\psi)^-|^2\dd x\dd s\\
&\quad+\int_0^T\int_{\mathbb T^N}|u^h_{n}-M_\psi|^{p-2}n(u^h_{n}-\psi)^-\dd x\dd s\nonumber\\
&\quad\leq C(p,M,D_0,T) \Big(1+\int_{\mathbb T^N}|u^0-M_\psi|^p\dd x\Big),\nonumber
\end{align}
Due to the uniform estimates \eqref{eq:priori-p-2}, we have the following result.
\begin{cor}[Uniform bounds for the defect measures of $u^h_{n}$]\label{cor:uniform-bounds-defect-measure-1}
For every $p\geq0$, there exists $C>0$ independent of $n\in\mathbb N$, such that
\begin{align}\label{eq:bounds for m-2}
 \int_{\mathbb{T}^N\times[0,T]\times\mathbb{R}}(1+|\xi|^p)\dd q^h_{n}\leq C(p,M,D_0,T).
\end{align}
\end{cor}

\begin{thm}[Existence and uniqueness of kinetic solutions to the skeleton obstacle problem]\label{thm:existence-uniqueness-skeleton-obstacle} Let \textbf{Hypotheses (H1) and (H2)$'$} hold.
Then, for every $h\in S_M$, the skeleton equation \eqref{eq:skeleton-obstacle-problem} admits a unique kinetic solution $(u^h,\nu^h)$ in the sense of Definition \ref{def:dfn-2}.

Furthermore, for every $p\geq2$, there exists a positive constant $C$ depending only on $p$, $T$, $D_0$, $M$, $\Vert u^0\Vert_{L^p(\mathbb T^N)}$ and $\Vert\psi\Vert_{L^\infty(Q_T)}$, such that
\begin{equation}\label{eq:skeleton-moment-bound}
\sup_{t\in[0,T]}\Vert u^h(t)\Vert_{L^p(\mathbb T^N)}^p+\int_{\mathbb T^N\times[0,T]\times\mathbb R}(1+|\xi|^p)\dd q^h(x,t,\xi)
+\nu^h(\mathbb T^N\times[0,T))\leq C(p,M,D_0,T).
\end{equation}
\end{thm}

\begin{proof}

To find a solution of \eqref{eq:skeleton-obstacle-problem}, we prove the comparison principle for two kinetic solutions $u^h_{n}$ and $u^h_{m}$ of the penalized equation \eqref{eq:penalized-skeleton}. 
For $n\leq m$, similar to Lemma \ref{lem:comparison-penalty-parameter}, we have
\begin{align}\label{eq:comparison-epsilon-local_b_epsilon-h}
\big\Vert(u^h_{n}(t)-u^h_{m}(t))^+\big\Vert_{L^1(\mathbb T^N)}&\leq\int_0^t\int_{\mathbb T^N}\one_{\{u^h_{n}\geq u^h_{m}\}}\big[b_{m}(u^h_{n}-\psi)-b_{m}(u^h_{m}-\psi)\big]\dd x\dd s\\
&\quad+\int_0^t\int_{\mathbb T^N}\one_{\{u^h_{n}\geq u^h_{m}\}}\sum_{k\geq1}\big[g_{k}(x,u^h_{n})-g_{k}(x,u^h_{m})\big]h_k(s)\dd x\dd s.\nonumber
\end{align}
The first term on the right-hand side of \eqref{eq:comparison-epsilon-local_b_epsilon-h} is nonpositive. Thus, we only need to handle the term involving $h$. Due to \eqref{eq:assumption for g Lip}, we have
\begin{align*}
&\int_0^t\int_{\mathbb T^N}\one_{\{u^h_{n}\geq u^h_{m}\}}\sum_{k\geq1}\big[g_{k}(x,u^h_{n})-g_{k}(x,u^h_{m})\big]h_k(s)\dd x\dd s\\
&\leq C\int_0^t|h(s)|_U\Vert (u_n^h(s)-u_m^h(s))^+\Vert_{L^1(\mathbb T^N)}\dd s .
\end{align*}
By employing Gr\"onwall's inequality, we reach
\[
\Vert (u_n^h(t)-u_m^h(t))^+\Vert_{L^1(\mathbb T^N)}=0.
\]
Thus, $(u_n^h)_n$ is increasing.
We define
\begin{equation*}
u^h(x,t):=\lim_{n\to\infty}u_n^h(x,t).
\end{equation*}
By the uniform a priori estimates, Vitali's theorem implies
\begin{equation}\label{eq:strong-convergence-un-skeleton}
u_n^h\rightarrow u^h\quad\text{strongly in }L^p(Q_T)
\end{equation}
for every finite $p\geq1$.
Moreover, from
\[
\int_0^T\int_{\mathbb T^N}(u_n^h-\psi)^-\dd x\dd t\leq \frac{C_h}{n},
\]
we obtain
\[
(u^h-\psi)^-=0,\qquad\text{a.e. in }Q_T,
\]
and hence $u^h\geq\psi$ a.e.

As in \cite[Subsection 4.1.2]{DV10-publish}, using the Banach--Alaoglu theorem, a diagonal argument, and the uniform a priori estimates, there exist nonnegative measures $\nu^h$ and $q^h$ such that, up to a subsequence,
\begin{align*}
\nu_n^h\stackrel{\ast}{\rightharpoonup}\nu^h\quad\text{in }\mathcal M(\mathbb T^N\times[0,T)),\quad q_n^h\stackrel{\ast}{\rightharpoonup}q^h\quad\text{in }\mathcal M(\mathbb T^N\times[0,T]\times\mathbb R).
\end{align*}
The uniform a priori estimates also imply that $q^h$ has the vanishing property at infinity, and therefore $q^h$ is a kinetic measure.

It remains to pass to the limit in \eqref{eq:kinetic-penalized-skeleton-obstacle-form}.
The strong convergence \eqref{eq:strong-convergence-un-skeleton} gives
\[
\int_{Q_T\times\mathbb R}|f_n^h-f^h|\dd x\dd t\dd\xi
=
\int_{Q_T}|u_n^h-u^h|\dd x\dd t
\rightarrow0.
\]
Thus the time derivative term and the transport term converge.
Moreover, by \textbf{Hypothesis (H1)} and $h\in L^2(0,T;U)$,
\[
\sum_{k\geq1}g_k(x,u_n^h)h_k(t)\rightarrow\sum_{k\geq1}g_k(x,u^h)h_k(t)\quad\text{in }L^1(Q_T),
\]
and hence the term involving $h$ also converges.
Passing to the limit in the measure terms gives \eqref{eq:kinetic-skeleton-obstacle}.
The weak initial condition follows by the same argument as Lemma \ref{lem:weak-initial-condition}.
Therefore $(u^h,\nu^h)$ is a kinetic solution to \eqref{eq:skeleton-obstacle-problem}. The estimate \eqref{eq:skeleton-moment-bound} is a consequence of \eqref{eq:bounds for m-2} and \eqref{eq:priori-p-2}.

We finally prove uniqueness. 
Let $(u^h,\nu^h)$ and $(\widetilde u^h,\widetilde\nu^h)$ be two kinetic solutions corresponding to the same control $h$.
Repeating the doubling-of-variables argument used in the proof of Theorem \ref{thm:uniqueness}, and estimating the obstacle terms by the Taylor-expansion argument leading to \eqref{eq:I-nu-final-estimate}, we obtain, for a.e. $t\in(0,T)$,
\begin{align}
\bigl\|(u^h(t)-\widetilde u^h(t))^+\bigr\|_{L^1(\mathbb T^N)}&\leq C\int_0^t\bigl(1+|h(s)|_U\bigr)\|u^h(s)-\widetilde u^h(s)\|_{L^1(\mathbb T^N)}\dd s \notag\\
&\quad+\frac12\Bigl[\nu^h(\mathbb T^N\times[0,t])-\widetilde\nu^h(\mathbb T^N\times[0,t])\Bigr].\notag
\end{align}
Interchanging the two solutions and adding the two inequalities eliminates the difference of the total masses of the reflection measures. 
Therefore,
\begin{align}
\|u^h(t)-\widetilde u^h(t)\|_{L^1(\mathbb T^N)}&\leq C\int_0^t\bigl(1+|h(s)|_U\bigr) \|u^h(s)-\widetilde u^h(s)\|_{L^1(\mathbb T^N)}\dd s.\label{eq:skeleton-l1-contraction}
\end{align}
Gr\"onwall's inequality yields $u^h=\widetilde u^h$ a.e. in $Q_T$ when the initial data coincide.
The proof of $\nu^h=\tilde{\nu}^h$ follows by the same argument as in the proof of Theorem \ref{thm:uniqueness}.
We complete the proof.
\end{proof}

\subsection{The viscous skeleton equation and its approximation convergence}\label{subsection:viscous skeleton approximation}
For any $\alpha>0$, $h\in S_M$, we consider the viscous skeleton equation
\begin{equation}\label{eq:skeleton-obstacle-problem-viscous}
\mathbf{App\ Equ}\ (\ast)\quad
\left\{
\begin{aligned}
\partial_t u_\alpha^h+\Div A(u_\alpha^h)-\alpha\Delta u_\alpha^h
&=\Phi(u_\alpha^h)h(t)+\nu_\alpha^h,
&&\text{in }\mathbb T^N\times(0,T),\\
u_\alpha^h
&\geq\psi,
&&\text{a.e. in }\mathbb T^N\times(0,T),\\
u_\alpha^h(0)
&=u^0.
\end{aligned}
\right.
\end{equation}
Similar to Definition \ref{def:dfn-2}, we propose the following definition.
\begin{defn}\label{def:dfn-2-viscous}
A pair $(u_\alpha^h,\nu_\alpha^h)$ is a kinetic solution of \eqref{eq:skeleton-obstacle-problem-viscous} with initial data $u^0$ and obstacle $\psi$ if the following conditions are satisfied.
\begin{enumerate}[label=\textbf{\arabic*}.]
  \item For any $p\geq1$, $u_\alpha^h\in L^\infty(0,T;L^p(\mathbb T^N))\cap L^2(0,T;H^1(\mathbb T^N)),$ and
\begin{align}\label{P-6-vicous}
\underset{0\leq t\leq T}{{\rm{ess\sup}}}\ \Vert u_\alpha^h(t)\Vert ^p_{L^p(\mathbb{T}^N)}+\alpha \int_0^T\int_{\mathbb T^N}|\nabla u^h_{\alpha}|^2\dd x\dd s\leq C(p,M,T,D_0),
\end{align}
\item  We have 
\begin{equation*}
u^h_{\alpha}(x,t)\geq\psi(x,t),\quad\text{for }\dd x\otimes\dd t\text{-a.e. }(x,t),
\end{equation*}
and $\nu_\alpha^h$ is a nonnegative finite Radon measure on $\mathbb{T}^N\times[0,T)$.
\item  The function $u^h_{\alpha}$ satisfies
\begin{equation*}
\lim_{\tau\downarrow0}\frac{1}{\tau}\int_0^\tau\Vert u^h_{\alpha}(t)-u^0\Vert _{L^1(\mathbb{T}^N)}\dd t=0.
\end{equation*}
\item Let $\dd o^h_{\alpha}:=\alpha|\nabla u^h_{\alpha}|^2\delta_{u^h_{\alpha}(x,t)}(\dd\xi)\dd x\dd t$ be the parabolic defect measure. There exists a (deterministic) kinetic measure $q^h_{\alpha}\geq o^h_{\alpha}$ such that $f:= \one_{u^h_{\alpha}>\xi}$ satisfies for all $\varphi\in C^\infty_c(\mathbb{T}^N\times [0,T)\times \mathbb{R})$,
\begin{align}\notag
&\int^T_0\langle f(t), \partial_t \varphi(t)\rangle \dd t+\langle f^0, \varphi(0)\rangle +\int^T_0\langle f(t), a(\xi)\cdot \nabla_x \varphi (t)\rangle \dd t+\alpha\int^T_0\langle f(t), \Delta_x \varphi (t)\rangle \dd t \\
&= -\sum_{k\geq 1}\int^T_0\int_{\mathbb{T}^N} \int_{\mathbb{R}}g_k(x,\xi)\varphi (x,t,\xi)h_k(t)\dd \mu^{\alpha,h}_{x,t}(\xi)\dd x\dd t \label{eq:kinetic-skeleton-obstacle-viscous}\\
&\quad + q^h_{\alpha}(\partial_{\xi} \varphi)-\int_0^T\int_{\mathbb{T}^N}\varphi(x,t,\psi(x,t))\dd\nu^h_{\alpha}(x,t),\nonumber
\end{align}
 where  $\mu^{\alpha,h}_{x,t}(\dd\xi)=\delta_{u^h_{\alpha}(x,t)}(\dd\xi)$.
\end{enumerate}
\end{defn}

Compared with equation \eqref{eq:skeleton-obstacle-problem}, only the extra viscous term $\alpha \Delta u^h_{\alpha}$ has been added in \eqref{eq:skeleton-obstacle-problem-viscous}. As a result,
the well-posedness of \eqref{eq:skeleton-obstacle-problem-viscous} is closely analogous to that of  \eqref{eq:skeleton-obstacle-problem}, so we only outline the key steps of the proof.
We first address the existence of solutions to \eqref{eq:skeleton-obstacle-problem-viscous}.
As shown in Section \ref{sec-skeleton}, for all $\alpha>0$ and $n\geq 1$, \eqref{P-3} possesses a unique weak solution $u^h_{\alpha,n}$ satisfying estimates \eqref{eq:priori-p-1} uniformly in $n$. 
In fact, one intermediate step requires replacing the coefficient $\Phi_{\alpha}$ with $\Phi_{\alpha'}$ and subsequently taking  $\alpha'\rightarrow 0$, which can be carried out following the proof in \cite[Section 4.1]{DV10}. 
With this preparation in place, we can construct a solution to  \eqref{eq:skeleton-obstacle-problem-viscous} by passing to the limit $n\rightarrow \infty$ via an argument analogous to that for Theorem \ref{thm:existence-uniqueness-skeleton-obstacle}. 
Specifically, the comparison estimate \eqref{eq:comparison-epsilon-local_b_epsilon-h} remains valid, due to the fact that the contributions arising from the viscous term $\alpha \Delta u^h_{\alpha,n}$ and the parabolic defect measure $o^h_{\alpha,n}$ sum to a nonpositive quantity, see, for instance, \cite[Theorem 3.3]{DHV16} for details. 
Consequently, the sequence $(u^h_{\alpha,n})_{n\geq 1}$ is increasing, and we can define
\begin{equation*}
u^h_{\alpha}(x,t):=\lim_{n\to\infty}u^h_{\alpha,n}(x,t).
\end{equation*}

Note that estimate \eqref{eq:priori-p-1} indicates, after extracting a subsequence, the weak convergence of $u_{\alpha,n}^h$ to $u^h_\alpha$ in $L^2(0,T;H^1(\mathbb{T}^N))$.
By the uniform measure estimates, there exist a subsequence $(n_j)_{j\geq1}$ and nonnegative finite measures $q_\alpha^h$ and $\nu_\alpha^h$ such that
\[
q_{\alpha,n_j}^h\overset{\ast}{\rightharpoonup}q_\alpha^h\quad\text{locally on }\mathbb T^N\times[0,T]\times\mathbb R,
\]
and
\[
\nu_{\alpha,n_j}^h\overset{\ast}{\rightharpoonup}\nu_\alpha^h\quad\text{in }C_0\bigl(\mathbb T^N\times[0,T)\bigr)^*.
\]
We now check that for every nonnegative $\phi\in C_c(\mathbb{T}^N\times[0,T)\times\mathbb{R})$,
\[
\alpha\int_{Q_T}\phi(x,t,u_\alpha^h)|\nabla u_{\alpha}^h|^2\dd x\dd t\leq\liminf_{n\to\infty}\alpha\int_{Q_T}\phi(x,t,u^h_{\alpha,n})|\nabla u^h_{\alpha,n}|^2\dd x\dd t.
\]
Set
\[
\phi_n(x,t):=\sqrt{\phi(x,t,u_{\alpha,n}^h(x,t))},\qquad
\phi_{\infty}(x,t):=\sqrt{\phi(x,t,u_\alpha^h(x,t))}.
\]
Since $u_{\alpha,n}^h\to u_\alpha^h$ a.s. and $\phi$ is bounded and
continuous, for every $\Psi\in L^2(Q_T;\mathbb R^N)$, the dominated
convergence theorem gives
\[
\phi_n\Psi\to \phi_{\infty}\Psi\qquad\text{strongly in }L^2(Q_T;\mathbb R^N).
\]
Together with
$\nabla u_{\alpha,n}^h\rightharpoonup\nabla u_\alpha^h$ in $L^2$,
this yields
\[
\phi_n\nabla u_{\alpha,n}^h
\rightharpoonup
\phi_{\infty}\nabla u_\alpha^h
\qquad\text{weakly in }L^2(Q_T;\mathbb R^N).
\]
Hence, by the weak lower semicontinuity of the $L^2$ norm,
\begin{equation}\label{eq:parabolic-defect-liminf}
\int_{Q_T}\phi(x,t,u_\alpha^h)|\nabla u_\alpha^h|^2\dd x\dd t
\leq\liminf_{n\to\infty}
\int_{Q_T}\phi(x,t,u_{\alpha,n}^h)
|\nabla u_{\alpha,n}^h|^2\dd x\dd t,
\end{equation}
which indicates $q_\alpha^h\geq o_\alpha^h$.
Proceeding exactly as in Theorem \ref{thm:existence-uniqueness-skeleton-obstacle}, $(u^h_{\alpha},\nu^h_{\alpha})$ is a solution to \eqref{eq:skeleton-obstacle-problem-viscous} in the sense of Definition \ref{def:dfn-2-viscous}. 
As for uniqueness, it follows from estimate \eqref{eq:skeleton-l1-contraction} and the fact that the combined contributions of the viscous term $\alpha \Delta u^h_{\alpha}$ and
the parabolic defect measure $o^h_{\alpha}$ are nonpositive.

Similar to \eqref{eq:skeleton-moment-bound}, we have for any $p\geq 2$,
\begin{equation}\label{eq:skeleton-moment-bound-viscous}
\int_{\mathbb T^N\times[0,T]\times\mathbb R}(1+|\xi|^p)\dd q^h_{\alpha}(x,t,\xi)
+\nu^h_{\alpha}(\mathbb T^N\times[0,T))\leq C(p,M,D_0,T).
\end{equation}

In the following, we are devoted to proving the convergence of viscous approximations to the skeleton obstacle problem \eqref{eq:skeleton-obstacle-problem}.

\begin{lem} \label{lem-1}
The kinetic solution $u^{h}_{\alpha}$ of \eqref{eq:skeleton-obstacle-problem-viscous} converges to the kinetic solution $u^{h}$ of \eqref{eq:skeleton-obstacle-problem} uniformly on $h\in S_M$. That is,
  \begin{align}\label{rs-18}
  \lim_{\alpha\rightarrow 0}\sup_{h\in S_M}\Vert u^{h}_{\alpha}-u^{h}\Vert _{L^1(Q_T)}=0.
\end{align}
\end{lem}
\begin{proof}
Let  $f_1=\one_{u^{h}_{\alpha}>\xi}$, $f_2=\one_{u^{h}>\xi}$. 
The corresponding kinetic and reflection measures are denoted by $(q^h_{\alpha}, \nu^h_{\alpha})$ and $(q^h, \nu^h)$. 
Set $\mu^1_{x,s}(\dd\xi)=\delta_{u^{h}_{\alpha}}(\dd\xi)$ and $\mu^2_{x,s}(\dd\xi)=\delta_{u^{h}}(\dd\xi)$.

According to Lemma \ref{lem:no atom of q nu}, $t=0$ is not an atom of kinetic measures and reflection measures. Moreover, by the same method as in the derivation of \eqref{eq:weak-kinetic on [0,t]}, the weak formulations \eqref{eq:kinetic-skeleton-obstacle} and
\eqref{eq:kinetic-skeleton-obstacle-viscous} can be strengthened to be weak only with respect to $(x,\xi)$. This means that for test function  $\varphi_1(x,\xi)$, 
\begin{equation*}
\begin{aligned}
&\langle f_1^+(t),\varphi_1\rangle -\langle f_1(0),\varphi_1\rangle
 -\int_0^t\!\langle f_1(s),a(\xi)\cdot\nabla_x\varphi_1\rangle\dd s \\
&= \alpha\int_0^t\!\langle f_1(s),\Delta_x\varphi_1\rangle\dd s+
\sum_{k\geq1}\int_0^t\int_{\mathbb{T}^N}\int_{\mathbb{R}}g_k(x,\xi)\varphi_1(x,\xi)h_k(s)\dd \mu^1_{x,s}(\xi)\dd x\dd s\\
&\quad -\int_{\mathbb{T}^N\times[0,t]\times{\mathbb{R}}}\partial_{\xi}\varphi_1(x,\xi)\dd q^h_{\alpha}(x,s,\xi)
+\int_{\mathbb{T}^N\times[0,t]}\varphi_1(x,\psi(x,s))\dd\nu^h_{\alpha}(x,s).
\end{aligned}
\end{equation*}
Similarly, for $\bar{f}_2(y,t,\zeta)$, with the test function $\varphi_2(y,\zeta)$, it follows that
\begin{equation}
\begin{aligned}
&\langle \bar{f}_2^+(t),\varphi_2\rangle -\langle \bar{f}_2(0),\varphi_2\rangle
 -\int_0^t\!\langle \bar{f}_2(s),a(\zeta)\cdot\nabla_y\varphi_2\rangle\dd s \\
&= -\sum_{k\geq1}\int_0^t\int_{\mathbb{T}^N}g_k(y,\zeta)\varphi_2(y,\zeta)h_k(s)\dd \mu^2_{y,s}(\zeta)\dd y\dd s\\
&\quad +\int_{\mathbb{T}^N\times[0,t]\times{\mathbb{R}}}\partial_{\zeta}\varphi_2(y,\zeta)\dd q^h(y,s,\zeta)
-\int_{\mathbb{T}^N\times[0,t]}\varphi_2(y,\psi(y,s))\dd\nu^h(y,s).
\end{aligned}\label{eq:weak-kinetic-2}
\end{equation}
Proceeding similarly to Theorem \ref{thm:uniqueness}, with the tensorization argument of \cite[proof of Proposition 9]{DV10-publish}, taking $\varphi_{\gamma,\delta}(x,\xi,y,\zeta)=\rho_{\gamma}(x-y)\kappa_\delta(\xi-\zeta)$, we reach
\begin{align*}
\mathcal{R}(t)&:=\int_{(\mathbb{T}^N)^2}\int_{\mathbb{R}^2} (f^+_1(t)\bar{f}^+_2(t)+\bar{f}^+_1(t)f^+_2(t))\rho_{\gamma}(x-y)\kappa_\delta(\xi-\zeta)\dd x\dd\xi\dd y\dd\zeta \\
\notag
&\leq \int_{(\mathbb{T}^N)^2}\int_{\mathbb{R}^2}(f_1(0)\bar{f}_2(0)+\bar{f}_1(0)f_2(0))\rho_{\gamma}(x-y)\kappa_\delta(\xi-\zeta)\dd x\dd\xi\dd y\dd\zeta \\
&\quad+J_{\gamma,\delta}^{\rm{vis}}(t)+J_{\gamma,\delta}^{\rm{flux}}(t)+J_{\gamma,\delta}^h(t)+
J_{\gamma,\delta}^{q}(t)+J_{\gamma,\delta}^{\nu}(t),\notag
\end{align*}
where $J_{\gamma,\delta}^{\rm{flux}}$ and $J_{\gamma,\delta}^{\nu}$ are the same as those in Theorem \ref{thm:uniqueness}, the viscous-related term is given by
\begin{align*}
 J_{\gamma,\delta}^{\rm{vis}}(t):=\alpha \int^t_0 \int_{(\mathbb{T}^N)^2}\int_{\mathbb{R}^2} (f_1(s)\bar{f}_2(s)+\bar{f}_1(s)f_2(s))\Delta_x\rho_{\gamma}(x-y)\kappa_\delta(\xi-\zeta)\dd x\dd\xi\dd y\dd\zeta\dd s,
\end{align*}
the control-related term is given by
\begin{align*}
J_{\gamma,\delta}^h(t)&:=\sum_{k\geq1}\int_0^t\int_{(\mathbb{T}^N)^2}\int_{\mathbb{R}^2}(\bar{f}_2-f_2)g_k(x,\xi)h_k(s)\rho_{\gamma}(x-y)\kappa_\delta(\xi-\zeta)\dd \mu^1_{x,s}(\xi)\dd \zeta \dd x \dd y \dd s\\
&\quad-\sum_{k\geq1}\int_0^t\int_{(\mathbb{T}^N)^2}\int_{\mathbb{R}^2}(f_1-\bar{f}_1)g_k(y,\zeta)h_k(s)\rho_{\gamma}(x-y)\kappa_\delta(\xi-\zeta)\dd \mu^2_{y,s}(\zeta)\dd \xi\dd x \dd y\dd s,
\end{align*}
and the term generated by the kinetic measure is expressed as
\begin{align*}
J_{\gamma,\delta}^{q}(t)&:=\int_{\mathbb{T}^N}\int_{\mathbb{R}}\int_{\mathbb{T}^N\times[0,t]\times\mathbb{R}}\rho_{\gamma}(x-y)\partial_\zeta\kappa_\delta(\xi-\zeta)({f}^{-}_1-\bar{f}^-_1)\dd\xi\dd q^h(y,s,\zeta)\dd x\\
&\quad-\int_{\mathbb{T}^N}\int_{\mathbb{R}}\int_{\mathbb{T}^N\times[0,t]\times\mathbb{R}}\rho_{\gamma}(x-y)\partial_\xi\kappa_\delta(\xi-\zeta)(\bar{f}_2^+-f_2^+)\dd q^h_{\alpha}(x,s,\xi)\dd\zeta\dd y.
\end{align*}
Referring to \eqref{eq:estimates for Ia-1}, we have
\begin{align*}
|J_{\gamma,\delta}^{\rm{flux}}(t)|\leq C_pT\delta\gamma^{-1}.
\end{align*}
Owing to \eqref{eq:I-nu-final-estimate}, by using \eqref{eq:skeleton-moment-bound} and \eqref{eq:skeleton-moment-bound-viscous}, we obtain
\begin{equation*}
J_{\gamma,\delta}^{\nu}(t)\leq C(M,D_0,T)\frac{\gamma^2}{\delta}.
\end{equation*}
Similarly to \eqref{eq:estimates for Iqvar}, by the nonnegativity of $q^h$ and $q^h_{\alpha}$, we have
\begin{equation*}
J_{\gamma,\delta}^{q}(t)\leq 0.
\end{equation*}
For the viscous term $J_{\gamma,\delta}^{\rm{vis}}$, define
\begin{align*}
  l_{\delta}(\xi,\zeta)=\int^{\infty}_{\zeta}\int^{\xi}_{-\infty} \kappa_{\delta}(\xi'-\zeta')\dd \xi' \dd \zeta'.
\end{align*}
By integration by parts, we have
\begin{align*}
 J_{\gamma,\delta}^{\rm{vis}}(t)&=\alpha \int^t_0 \int_{(\mathbb{T}^N)^2}\Delta_x\rho_{\gamma}(x-y)\Big[\int_{\mathbb{R}^2} (f_1(s)\bar{f}_2(s)+\bar{f}_1(s)f_2(s))\kappa_\delta(\xi-\zeta)\dd \xi \dd \zeta \Big] \dd x \dd y \dd s\\
 &= \alpha \int^t_0 \int_{(\mathbb{T}^N)^2}\Delta_x\rho_{\gamma}(x-y)\int_{\mathbb{R}^2}\Big[ l_{\delta}(\xi,\zeta)+ l_{\delta}(\zeta,\xi)\Big]\dd \mu^1_{x,s}\otimes  \mu^2_{y,s}(\xi,\zeta) \dd x \dd y \dd s.
\end{align*}
Setting $\xi''=\xi'-\zeta'$, we have
\begin{align*}
l_{\delta}(\xi, \zeta)=\int^{\infty}_{\zeta}\int^{\xi-\zeta'}_{-\infty} \kappa_{\delta}(\xi'')\dd \xi'' \dd \zeta'.
\end{align*}
Since {\rm{supp}}$\kappa_{\delta}\subset (-\delta,\delta)$, we have
\begin{align*}
  \int^{\xi-\zeta'}_{-\infty} \kappa_{\delta}(\xi'')\dd \xi''=0,
\end{align*}
if $\xi-\zeta'<-\delta$, that is, $\zeta'>\xi+\delta$. Therefore the upper limit of integration can be restricted to $\xi+\delta$. 
If $\zeta\ge\xi+\delta$, the integration interval is empty and $l_\delta(\xi,\zeta)=0$. Otherwise, 
there exists a constant $C>0$ only depending on $\|\kappa_{1}\|_{L^{\infty}(\mathbb{R})}$ such that
\begin{align*}
l_{\delta}(\xi, \zeta)&\leq  
\int^{\xi+\delta}_{\zeta}\Big(\int^{\xi-\zeta'}_{-\infty} \kappa_{\delta}(\xi'')\dd \xi''\Big)\dd \zeta'\\
&\leq  \int^{\xi+\delta}_{\zeta}2\delta \|\kappa_{\delta}\|_{L^{\infty}(\mathbb{R})}\dd\zeta'\\
&\leq  C(|\xi|+|\zeta|+\delta).
\end{align*}
The same upper bound holds for $l_{\delta}( \zeta,\xi)$. Then, we obtain
\begin{align*}
\int_{\mathbb{R}^2}\Big[l_{\delta}(\xi, \zeta)+ l_{\delta}(\zeta,\xi)\Big]\dd\mu^{1}_{x,s}\otimes\mu^{2}_{y,s}(\xi,\zeta)\leq 2C(|u^h_{\alpha}(x,s)|+|u^h(y,s)|+\delta).
\end{align*}
As a result, we reach
\begin{align*}
 J_{\gamma,\delta}^{\rm{vis}}(t)&\leq 2\alpha C \int^T_0 \int_{(\mathbb{T}^N)^2}|\Delta_x\rho_{\gamma}(x-y)|\big(|u^h_{\alpha}(x,s)|+|u^h(y,s)|+\delta\big) \dd x \dd y \dd s\\
 &\leq C(M,D_0,T,\|\kappa_1\|_{L^{\infty}(\mathbb{R})})\alpha\gamma^{-2}(1+\delta),
\end{align*}
where we have used \eqref{eq:skeleton-moment-bound} and \eqref{P-6-vicous}.

With regard to $J_{\gamma,\delta}^h(t)$, we introduce
\begin{align*}
  \Lambda_1(\xi,\zeta):=\int^{\xi}_{-\infty}\kappa_{\delta}(\xi'-\zeta)\dd\xi', \quad   \Lambda_2(\zeta,\xi):=\int^{+\infty}_{\zeta}\kappa_{\delta}(\xi-\zeta')\dd\zeta'.
\end{align*}
It holds that $\Lambda_1(\xi,\zeta)=\Lambda_2(\zeta,\xi)=\int^{\xi-\zeta}_{-\infty}\kappa_{\delta}(y) \dd y$ and $\Lambda_1(\xi,\zeta)\leq 1$.
Similarly, define
\begin{align*}
  \tilde{\Lambda}_1(\xi,\zeta)=\int^{+\infty}_{\xi}\kappa_{\delta}(\xi'-\zeta)\dd\xi',\quad \tilde{\Lambda}_2(\zeta,\xi)=\int^{\zeta}_{-\infty}\kappa_{\delta}(\xi-\zeta')\dd\zeta',
\end{align*}
we have $\tilde{\Lambda}_1(\xi,\zeta)=\tilde{\Lambda}_2(\zeta,\xi)=\int^{+\infty}_{\xi-\zeta}\kappa_{\delta}(y)\dd y$ and $\tilde{\Lambda}_1(\xi,\zeta)\leq 1$.

With the aid of the functions defined above, integration by parts yields
\begin{align*}
 J_{\gamma,\delta}^{h}(t)&\leq \sum_{k\geq 1}\int^t_0\int_{(\mathbb{T}^N)^2}\rho_{\gamma}(x-y)\int_{\mathbb{R}^2}\Lambda_1(\xi,\zeta)(g_{k}(x,\xi)-g_{k}(y,\zeta))h_k(s)\dd \mu^{1}_{x,s}\otimes \mu^{2}_{y,s}(\xi,\zeta) \dd x\dd y\dd s\\
 &+\quad \sum_{k\geq 1}\int^t_0\int_{(\mathbb{T}^N)^2}\rho_{\gamma}(x-y)\int_{\mathbb{R}^2}\tilde{\Lambda}_1(\xi,\zeta)(g_{k}(x,\xi)-g_{k}(y,\zeta))h_k(s)\dd \mu^{1}_{x,s}\otimes \mu^{2}_{y,s}(\xi,\zeta) \dd x\dd y\dd s\\
&\leq 2\int^t_0|h(s)|_U\int_{(\mathbb{T}^N)^2}\rho_{\gamma}(x-y)\\
&\qquad\qquad\qquad\times\int_{\mathbb{R}^2}\Big(\sum_{k\geq 1}|g_{k}(x,\xi)-g_{k}(y,\zeta)|^2\Big)^{\frac{1}{2}}\dd \mu^{1}_{x,s}\otimes \mu^{2}_{y,s}(\xi,\zeta) \dd x\dd y\dd s.
\end{align*}
In view of (\ref{eq:assumption for g Lip}), we have
\begin{align*}
 J_{\gamma,\delta}^{h}(t)
&\leq 2\sqrt{D_0}\int^t_0|h(s)|_U\int_{(\mathbb{T}^N)^2}\rho_{\gamma}(x-y)\int_{\mathbb{R}^2}|x-y|\dd \mu^{1}_{x,s}\otimes \mu^{2}_{y,s}(\xi,\zeta) \dd x\dd y\dd s\\
&\quad+ 2\sqrt{D_0}\int^t_0|h(s)|_U\int_{(\mathbb{T}^N)^2}\rho_{\gamma}(x-y)\int_{\mathbb{R}^2}|\xi-\zeta|\dd \mu^{1}_{x,s}\otimes \mu^{2}_{y,s}(\xi,\zeta) \dd x\dd y\dd s\\
&\leq 2\sqrt{D_0}\gamma \int^t_0|h(s)|_U \dd s+2\sqrt{D_0}\int^t_0|h(s)|_U\\
&\qquad\times\int_{(\mathbb{T}^N)^2}\int_{\mathbb{R}}(f^+_1(x,s,\xi)\bar{f}^+_2(y,s,\xi)+\bar{f}^+_1(x,s,\xi)f^+_2(y,s,\xi))\rho_{\gamma}(x-y)\dd\xi \dd x\dd y \dd s.
\end{align*}
Referring to \cite[(3.14)]{DV18} or \cite[(5.16)-(5.17)]{DWZZ20}, we have for all $s\in(0,T)$,
\begin{align}\label{appro-1}
&  \Big|\int_{(\mathbb{T}^N)^2}\int_{\mathbb{R}}\rho_{\gamma}(x-y)(f^+_1(x,s,\xi)\bar{f}^+_2(y,s,\xi)+\bar{f}^+_1(x,s,\xi)f^+_2(y,s,\xi))\dd\xi \dd x\dd y\\ \notag
& \ -\int_{(\mathbb{T}^N)^2}\int_{\mathbb{R}^2}(f^+_1(x,s,\xi)\bar{f}^+_2(y,s,\zeta)+\bar{f}^+_1(x,s,\xi)f^+_2(y,s,\zeta))\rho_{\gamma}(x-y)\kappa_{\delta}(\xi-\zeta)\dd x\dd y\dd\xi \dd\zeta\Big|\\
&\leq 2\delta.\notag
\end{align}
Consequently, the term $J_{\gamma,\delta}^{h}(t)$ can be further estimated as
\begin{align} \label{rs-20}
 J_{\gamma,\delta}^{h}(t)
 &\leq 2\sqrt{D_0}\gamma \int^t_0|h(s)|_U \dd s+ 4 \delta \sqrt{D_0}\int^t_0|h(s)|_U \dd s+2\sqrt{D_0}\int^t_0|h(s)|_U\mathcal{R}(s)\dd s\\
\notag
&\leq 2\sqrt{D_0}\gamma \sqrt{M}\sqrt{T}+4 \delta \sqrt{D_0}\sqrt{M}\sqrt{T}
+ 2\sqrt{D_0}\int^t_0|h(s)|_U\mathcal{R}(s)\dd s.
\end{align}
Collecting all the previous estimates, we arrive at
\begin{align*}\notag
\mathcal{R}(t)
&\leq \int_{(\mathbb{T}^N)^2}\int_{\mathbb{R}^2}(f_1(0)\bar{f}_2(0)+\bar{f}_1(0)f_2(0))\rho_{\gamma}(x-y)\kappa_\delta(\xi-\zeta)\dd x\dd\xi\dd y\dd\zeta \\
&\quad+ C(M,D_0,T,\|\kappa_1\|_{L^{\infty}(\mathbb{R})})\alpha\gamma^{-2}(1+\delta)+R(\gamma,\delta)
+ 2\sqrt{D_0}\int^t_0|h(s)|_U\mathcal{R}(s)\dd s,
\end{align*}
where
\begin{equation}\label{eq:esti for errorR}
R(\gamma,\delta):=C(M,D_0,T)\frac{\gamma^2}{\delta}+C_pT\delta\gamma^{-1}+C(M,D_0,T)(\gamma+\delta).
\end{equation}

Employing Gr\"onwall's inequality, we deduce that
\begin{align}\label{eq:rs-13}
&\int_{(\mathbb{T}^N)^2}\int_{\mathbb{R}^2} (f^+_1(t)\bar{f}^+_2(t)+\bar{f}^+_1(t)f^+_2(t))\rho_{\gamma}(x-y)\kappa_\delta(\xi-\zeta) \dd x\dd\xi\dd y\dd\zeta\\ \notag
&\leq \exp\{2\sqrt{D_0}\sqrt{M}\sqrt{T}\}\bigg[\int_{(\mathbb{T}^N)^2}\int_{\mathbb{R}^2}(f_1(0)\bar{f}_2(0)+\bar{f}_1(0)f_2(0))\rho_{\gamma}(x-y)\kappa_\delta(\xi-\zeta)\dd x\dd\xi \dd y\dd\zeta\\
\notag
&\quad+ C(M,D_0,T,\|\kappa_1\|_{L^{\infty}(\mathbb{R})})\alpha\gamma^{-2}(1+\delta)+R(\gamma,\delta)\bigg].
\end{align}
For any $t\in [0,T]$, define the error term
\begin{align*}
&\mathcal{E}_t(\alpha,\gamma,\delta)\\
&:=\int_{\mathbb{T}^N}\int_{\mathbb{R}}(f^+_1(x,t,\xi)\bar{f}^+_2(x,t,\xi)+\bar{f}^{+}_1(x,t,\xi){f}^+_2(x,t,\xi))\dd x\dd\xi\\
&\quad-\int_{(\mathbb{T}^N)^2}\int_{\mathbb{R}^2}\rho_\gamma (x-y)\kappa_{\delta}(\xi-\zeta)(f^{+}_1(x,t,\xi)\bar{f}^{+}_2(y,t,\zeta)+\bar{f}^{+}_1(x,t,\xi){f}^{+}_2(y,t,\zeta))\dd\xi \dd\zeta \dd x\dd y\\
&= \Big[\int_{\mathbb{T}^N}\int_{\mathbb{R}}(f^{+}_1(x,t,\xi)\bar{f}^{+}_2(x,t,\xi)+\bar{f}^{+}_1(x,t,\xi){f}^{+}_2(x,t,\xi))\dd x\dd\xi\\
& \quad-\int_{(\mathbb{T}^N)^2}\int_{\mathbb{R}}\rho_{\gamma}(x-y)(f^{+}_1(x,t,\xi)\bar{f}^{+}_2(y,t,\xi)+\bar{f}^{+}_1(x,t,\xi){f}^{+}_2(y,t,\xi))\dd\xi \dd x\dd y\Big]\\
& \quad+\Big[\int_{(\mathbb{T}^N)^2}\int_{\mathbb{R}}\rho_{\gamma}(x-y)(f^{+}_1(x,t,\xi)\bar{f}^{+}_2(y,t,\xi)+\bar{f}^{+}_1(x,t,\xi){f}^{+}_2(y,t,\xi))\dd\xi \dd x\dd y\\
& \quad-\int_{(\mathbb{T}^N)^2}\int_{\mathbb{R}^2}\rho_\gamma (x-y)\kappa_{\delta}(\xi-\zeta)(f^{+}_1(x,t,\xi)\bar{f}^{+}_2(y,t,\zeta)+\bar{f}^{+}_1(x,t,\xi){f}^{+}_2(y,t,\zeta))\dd\xi \dd\zeta \dd x\dd y\Big]\\
&=: H_1+H_2.
\end{align*}
The estimate \eqref{appro-1} yields 
\begin{align}\label{eq:rs-14}
|H_2|\leq 2\delta.
\end{align}

For $H_1$, with the definition of $\bar{f}^+_i$, we have
\begin{align*}
H_1&=\int_{(\mathbb{T}^N)^2}\int_{\mathbb{R}}\rho_\gamma(x-y)\big(1-2f^+_1(x,t,\xi)\big)\big[{f}^+_2(x,t,\xi)-{f}^+_2(y,t,\xi)\big]\dd\xi\dd x\dd y.
\end{align*}
Based on the facts
\[
|a-b|\leq a(1-b)+(1-a)b,\qquad \text{for }a,b\in[0,1],
\]
and $0\leq|f_i^+|\leq1$, which derived from $0\leq|f_i|\leq1$ and Proposition \ref{prop:trace-obstacle-scl}, we have
\begin{align*}
|H_1|&\leq\int_{(\mathbb{T}^N)^2}\int_{\mathbb{R}}\rho_\gamma(x-y)\big|{f}^+_2(x,t,\xi)-{f}^+_2(y,t,\xi)\big|\dd\xi\dd x\dd y\\
&\leq\int_{(\mathbb{T}^N)^2}\int_{\mathbb{R}}\rho_\gamma(x-y)\big[{f}^+_2(x,t,\xi)\bar{f}^+_2(y,t,\xi)+\bar{f}^+_2(x,t,\xi){f}^+_2(y,t,\xi)\big]\dd\xi\dd x\dd y.
\end{align*}
Combining a similar proof of \eqref{appro-1}, we obtain
\begin{align*}\notag
|H_1|&\leq 2\delta+ \int_{(\mathbb{T}^N)^2}\int_{\mathbb{R}^2}\rho_{\gamma}(x-y)\kappa_{\delta}(\xi-\zeta)(f^+_2(x,t,\xi)\bar{f}^+_2(y,t,\zeta)+\bar{f}^+_2(x,t,\xi)f^+_2(y,t,\zeta))\dd\xi \dd\zeta \dd x\dd y.
\end{align*}
Similar to \eqref{eq:rs-13} ($\alpha=0$), we have
\begin{align}\notag
&\int_{(\mathbb{T}^N)^2}\int_{\mathbb{R}^2} (f^+_2(t)\bar{f}^+_2(t)+\bar{f}^+_2(t)f^+_2(t))\rho_{\gamma}(x-y)\kappa_\delta(\xi-\zeta)\dd x\dd\xi\dd y\dd\zeta \\ \notag
&\leq \exp\{2\sqrt{D_0}\sqrt{M}\sqrt{T}\}\Big(\mathcal{E}_0(\gamma,\delta)
+ R(\gamma,\delta)\Big),
\end{align}
where
\[\mathcal{E}_0(\gamma,\delta):=
\int_{(\mathbb{T}^N)^2}\int_{\mathbb{R}^2}(f_1(0)\bar{f}_2(0)+\bar{f}_1(0)f_2(0))\rho_{\gamma}(x-y)\kappa_\delta(\xi-\zeta)\dd x\dd\xi\dd y\dd\zeta,
\]
which will vanish as $\gamma,\delta\downarrow0$ using the continuity of translation in $L^1(\mathbb{T}^N)$.
Hence, we reach
\begin{align}\label{rs-16}
|H_1|&\leq 2\delta+ \exp\{2\sqrt{D_0}\sqrt{M}\sqrt{T}\}\Big(\mathcal{E}_0(\gamma,\delta)
+R(\gamma,\delta)\Big).
\end{align}
Combining \eqref{eq:rs-14} and \eqref{rs-16}, we have
\begin{align}\label{eq:rs-17}
\mathcal{E}_t(\alpha,\gamma,\delta)&\leq 4\delta+ \exp\{2\sqrt{D_0}\sqrt{M}\sqrt{T}\}\Big(\mathcal{E}_0(\gamma,\delta)
+ R(\gamma,\delta)\Big).
\end{align}

Returning to \eqref{eq:rs-13} and using Proposition \ref{prop:trace-obstacle-scl}, we have
\begin{align*}\notag
&\|u^{h}_{\alpha}(t)-u^{h}(t)\|_{L^1(\mathbb{T}^N)}\\
&\leq \mathcal{E}_t(\alpha,\gamma,\delta)+ \int_{(\mathbb{T}^N)^2}\int_{\mathbb{R}^2} (f^+_1(t)\bar{f}^+_2(t)+\bar{f}^+_1(t)f^+_2(t))\rho_{\gamma}(x-y)\kappa_\delta(\xi-\zeta)\dd x\dd\xi\dd y\dd\zeta \\ \notag
&\leq 4\delta+2\exp\{2\sqrt{D_0}\sqrt{M}\sqrt{T}\} \Big(\mathcal{E}_0(\gamma,\delta)+C(M,D_0,T,\|\kappa_1\|_{L^{\infty}(\mathbb{R})})\alpha\gamma^{-2}(1+\delta)+R(\gamma,\delta)\Big).
\end{align*}
Note that the right hand side is uniformly bounded on $h\in S_M$.  
With the definition of $R(\gamma,\delta)$ in \eqref{eq:esti for errorR}, let $\gamma=\alpha^{\frac{1}{3}}$ and $\delta=\gamma^{\frac{3}{2}}$. 
Taking $\alpha\rightarrow 0$, we get
\begin{align*}
\lim_{\alpha\rightarrow 0}\sup_{h\in S_M}\int^T_0\|u^{h}_{\alpha}(t)-u^{h}(t)\|_{L^1(\mathbb{T}^N)}\dd t=0,
\end{align*}
which is the desired result (\ref{rs-18}).
\end{proof}

\subsection{The continuity of the viscous skeleton equation}
This section concerns the weak--strong continuity of the viscous skeleton equation \eqref{eq:skeleton-obstacle-problem-viscous}. In the sequel, we fix $\alpha>0$.
Our strategy relies on the following Aubin--Lions--Simon compactness criterion, whose proof can be found in \cite{Simon1987}.

\begin{lem}[Aubin--Lions--Simon compactness criterion]
\label{lem:simon-compactness}
Let
\[
B_0\subset B\hookrightarrow B_1
\]
be three Banach spaces, where the first embedding is compact and the second one is continuous.
For $1\leq p<\infty$, let $\mathcal{G}$ be a bounded subset of $L^p(0,T;B_0)$ such that
\[
\lim_{\tau\downarrow0}\sup_{u\in \mathcal{G}}\Vert u(\cdot+\tau)-u(\cdot)\Vert _{L^p(0,T-\tau;B_1)}=0.
\]
Then, $\mathcal{G}$ is relatively compact in $L^p(0,T;B)$.
\end{lem}

The following result establishes the weak--strong continuity of the viscous skeleton equation.
\begin{prop}
\label{prop:weak-continuity-viscous-skeleton-obstacle} Assume \textbf{Hypotheses (H1)} and \textbf{(H2)$'$} hold.
Let $\{h^\varepsilon\}_{\varepsilon>0}\subset S_M$ and $h\in S_M$ satisfy
\[
h^\varepsilon\rightharpoonup h
\qquad
\text{weakly\ in } L^2(0,T;U),\quad \text{as }\ \varepsilon\rightarrow 0.
\]
Then, for any $\alpha>0$, the kinetic solutions $u_\alpha^{h^\varepsilon}$ and $u_\alpha^h$ of \eqref{eq:skeleton-obstacle-problem-viscous} with controls $h^{\varepsilon}$ and $h$, and with
initial data $u^0$ satisfy that, as $\varepsilon\rightarrow 0$,
\[
u_\alpha^{h^\varepsilon}\rightarrow u_\alpha^h\qquad\text{strongly in }L^2(Q_T).
\]
\end{prop}

\begin{proof}
Since $\alpha>0$ is fixed, for simplicity, we set
\[
u_\varepsilon:=u_\alpha^{h^\varepsilon},\qquad\nu_\varepsilon:=\nu_\alpha^{h^\varepsilon},\qquad q_\varepsilon:=q_\alpha^{h^\varepsilon}.
\]
Our goal is to show that for any sequence $\varepsilon_j\downarrow0$, there exists a subsequence (still denoted by $\varepsilon$), such that
\[
u_\varepsilon\to u_\alpha^h\qquad\text{strongly in }L^2(Q_T).
\]
Our proof proceeds in two steps. 
First, we establish that the family $\{u_\varepsilon\}_{\varepsilon>0}$ is relatively compact in $L^2(0,T;L^2(\mathbb T^N))$. 
Second, we verify that the limit of any such subsequence coincides precisely with $u^h_{\alpha}$.

\noindent \textbf{Step 1: }\quad Let $\chi_\varepsilon$ be the modification of the kinetic function $\one_{\{\xi<u_{\varepsilon}(x,t)\}}$, that is,
\[
\chi_\varepsilon(x,t,\xi):=\one_{\{\xi<u_{\varepsilon}(x,t)\}}-\one_{\{\xi<0\}}=\one_{\{0<\xi<u_\varepsilon(x,t)\}}-\one_{\{u_\varepsilon(x,t)<\xi<0\}}.
\]
For any $h\in S_M$, define
\[
G_h(x,t,\xi):=\sum_{k\geq1}g_k(x,\xi)h_k(t).
\]
From the kinetic formulation \eqref{eq:kinetic-skeleton-obstacle-viscous} for $u_\varepsilon$, we have for every $\varphi\in C_c^\infty(\mathbb T^N\times[0,T)\times\mathbb R)$,
\[
\begin{aligned}
&\int_0^T
\langle \chi_\varepsilon(t),\partial_t\varphi(t)\rangle\dd t+\langle \chi^0,\varphi(0)\rangle+\int_0^T\langle \chi_\varepsilon(t),a(\xi)\cdot\nabla_x\varphi(t)\rangle\dd t
+\alpha\int_0^T\langle \chi_\varepsilon(t),\Delta_x\varphi(t)\rangle\dd t\\
&=-\int_0^T\int_{\mathbb T^N}G_{h^\varepsilon}(x,t,u_\varepsilon(x,t))\varphi(x,t,u_\varepsilon(x,t))\dd x\dd t\\
&\quad+q_\varepsilon(\partial_\xi\varphi)-\int_{Q_T}\varphi(x,t,\psi(x,t))\dd\nu_\varepsilon(x,t).
\end{aligned}
\]

Choose $\varrho\in C_c^\infty(\mathbb R)$ such that
\[
0\leq\varrho\leq1,\qquad\varrho(\xi)=1\ \text{for }|\xi|\leq1,\qquad\varrho(\xi)=0\ \text{for }|\xi|\geq2,
\]
and define
\[
\varrho_R(\xi):=\varrho(\xi/R),\qquad R>1.
\]
Taking the special test function $\varphi_R(x,t,\xi)=\theta(x)\zeta(t)\varrho_R(\xi)$ with $\theta\in C^\infty(\mathbb T^N)$ and $\zeta\in C_c^1([0,T))$ in the above kinetic formulation, we deduce that
\[
\begin{aligned}
&\int_0^T\int_{\mathbb T^N}\int_{\mathbb R}\chi_\varepsilon(x,t,\xi)\theta(x)\zeta'(t)\varrho_R(\xi)\dd\xi\dd x\dd t+\int_{\mathbb T^N}\int_{\mathbb R}\chi^0(x,\xi)\theta(x)\zeta(0)\varrho_R(\xi)\dd\xi\dd x\\
&\quad+\int_0^T\int_{\mathbb T^N}\int_{\mathbb R}\chi_\varepsilon(x,t,\xi)a(\xi)\cdot\nabla_x\theta(x)\zeta(t)\varrho_R(\xi)\dd\xi\dd x\dd t\\
&\quad+\alpha\int_0^T\int_{\mathbb T^N}\int_{\mathbb R}\chi_\varepsilon(x,t,\xi)\Delta_x\theta(x)\zeta(t)\varrho_R(\xi)\dd\xi\dd x\dd t\\
&=-\int_0^T\int_{\mathbb T^N}G_{h^\varepsilon}(x,t,u_\varepsilon(x,t))\theta(x)\zeta(t)\varrho_R(u_\varepsilon(x,t))\dd x\dd t\\
&\quad+q_\varepsilon\left(\theta(x)\zeta(t)\varrho_R'(\xi)\right)-\int_{Q_T}\theta(x)\zeta(t)\varrho_R(\psi(x,t))\dd\nu_\varepsilon(x,t).
\end{aligned}
\]
Letting $R\to\infty$, and then using the dominated convergence theorem and vanishing property at infinity of $q_\varepsilon$, we obtain
\[
\begin{aligned}
&\int_0^T\bigg(\int_{\mathbb T^N}u_\varepsilon(x,t)\theta(x)\dd x\bigg)\zeta'(t)\dd t+\bigg(\int_{\mathbb T^N}u^0(x)\theta(x)\dd x\bigg)\zeta(0)\\
&\quad+\int_0^T\bigg(\int_{\mathbb T^N}A(u_\varepsilon(x,t))\cdot\nabla_x\theta(x)\dd x\bigg)\zeta(t)\dd t\\
&\quad+\alpha\int_0^T\bigg(\int_{\mathbb T^N}u_\varepsilon(x,t)\Delta_x\theta(x)\dd x\bigg)\zeta(t)\dd t\\
&=-\int_0^T\bigg(\int_{\mathbb T^N}G_{h^\varepsilon}(x,t,u_\varepsilon(x,t))\theta(x)\dd x\bigg)\zeta(t)\dd t-\int_{Q_T}\theta(x)\zeta(t)\dd\nu_\varepsilon(x,t).
\end{aligned}
\]

Define the scalar function
\[
F_\varepsilon^\theta(t):=\int_{\mathbb T^N}u_\varepsilon(x,t)\theta(x)\dd x,
\]
and the finite signed measure $\tilde\mu_\varepsilon^\theta$ on $[0,T]$ by
\[
\begin{aligned}
\tilde\mu_\varepsilon^\theta(E)&:=\int_E\int_{\mathbb T^N}A(u_\varepsilon(x,t))\cdot\nabla_x\theta(x)\dd x\dd t+\alpha\int_E\int_{\mathbb T^N}u_\varepsilon(x,t)\Delta_x\theta(x)\dd x\dd t\\
&\quad+\int_E\int_{\mathbb T^N}G_{h^\varepsilon}(x,t,u_\varepsilon(x,t))\theta(x)\dd x\dd t+\int_{\mathbb T^N\times E}\theta(x)\dd\nu_\varepsilon(x,t).
\end{aligned}
\]
For notational convenience, we extend $\nu_\varepsilon$ to $\mathbb T^N\times[0,T]$ by setting
\[
\nu_\varepsilon(\mathbb T^N\times\{T\})=0.
\]
Then, we have
\[
\int_0^T F_\varepsilon^\theta(t)\zeta'(t)\dd t+F_\varepsilon^\theta(0)\zeta(0)=-\int_{[0,T)}\zeta(t)\dd\tilde\mu_\varepsilon^\theta(t).
\]
Thus, in the sense of distributions on $[0,T)$, we have
\[
\dd F_\varepsilon^\theta=\dd\tilde\mu_\varepsilon^\theta.
\]
Hence, $F_\varepsilon^\theta$ admits a right-continuous representative $F_{\varepsilon}^{\theta,+}$ on $[0,T)$ and a left-continuous representative $F_{\varepsilon}^{\theta,-}$ on $(0,T]$ such that, for every $0\leq s<t\leq T$,
\begin{align}\label{rs-27}
 F_{\varepsilon}^{\theta,+}(t)-F_{\varepsilon}^{\theta,+}(s)=\tilde\mu_\varepsilon^\theta((s,t]),\quad\text{and}\quad F_{\varepsilon}^{\theta,-}(t)-F_{\varepsilon}^{\theta,-}(s)=\tilde\mu_\varepsilon^\theta([s,t)). \end{align}

With the density of $\theta$ in $H^r(\mathbb{T}^N)$ for $r>N/2+2$, the solution $u_\varepsilon$ has right-continuous and left-continuous representatives in the space
\[
Y:=H^{-r}(\mathbb T^N),\qquad r>N/2+2.
\]
In the following estimate, we use the right-continuous representative and still denote it by $u_\varepsilon$.
Let $\theta\in H^r(\mathbb T^N)$ with $\Vert \theta\Vert _{H^r(\mathbb T^N)}\leq1$.
Since $r>N/2+2$, we have $\Vert \theta\Vert _{C^2(\mathbb T^N)}\leq C$.

From the first identity of \eqref{rs-27}, for $0\leq s<t\leq T$, we have
\begin{align}\notag
|\langle u_\varepsilon(t)-u_\varepsilon(s),\theta\rangle|&\leq C\int_s^t\Vert A(u_\varepsilon(\ell))\Vert _{L^1(\mathbb T^N)}\dd\ell+\alpha C\int_s^t\Vert u_\varepsilon(\ell)\Vert _{L^1(\mathbb T^N)}\dd\ell\\
\label{rs-32}
&\quad+C\int_s^t\Vert G_{h^\varepsilon}(\ell,u_\varepsilon(\ell))\Vert _{L^1(\mathbb T^N)}\dd\ell+C\nu_\varepsilon(\mathbb T^N\times(s,t]).
\end{align}

Define the finite positive measure $\Lambda_\varepsilon$ on $[0,T]$ by
\[d
\begin{aligned}
\Lambda_\varepsilon(E)&:=\int_E\Big(1+\Vert A(u_\varepsilon(t))\Vert _{L^1(\mathbb T^N)}+\alpha\Vert u_\varepsilon(t)\Vert _{L^1(\mathbb T^N)}\Big)\dd t\\
&\quad+\int_E\Vert G_{h^\varepsilon}(t,u_\varepsilon(t))\Vert _{L^1(\mathbb T^N)}\dd t+\nu_\varepsilon(\mathbb T^N\times E).
\end{aligned}
\]
Taking the supremum over all $\theta\in H^r(\mathbb{T}^N)$ in \eqref{rs-32}, we have
\begin{align}\label{rs-28}
\Vert u_\varepsilon(t)-u_\varepsilon(s)\Vert _Y\leq C\Lambda_\varepsilon((s,t]).
\end{align}

By \eqref{eq:assumption for g} and \eqref{P-6-vicous}, we have
\[
\begin{aligned}
&\int_0^T\Vert G_{h^\varepsilon}(t,u_\varepsilon(t))\Vert _{L^1(\mathbb T^N)}\dd t\\
&\leq C\int_0^T|h^\varepsilon(t)|_U\big(1+\Vert u_\varepsilon(t)\Vert _{L^1(\mathbb T^N)}\big)\dd t\\
&\leq C\Vert h^\varepsilon\Vert _{L^2(0,T;U)}\big(1+\Vert u_\varepsilon\Vert _{L^2(Q_T)}\big)\\
&\leq C_{\alpha,M}.
\end{aligned}
\]
This, together with \eqref{P-6-vicous} and \eqref{eq:skeleton-moment-bound-viscous}, yields
\begin{align}\label{rs-29}
\sup_\varepsilon\Lambda_\varepsilon([0,T])\leq C_{\alpha,M}.
\end{align}

For any $0<\tau<T$,
integrating $t$ over $[0,T-\tau]$, by \eqref{rs-28}, we have
\begin{align*}
\int_0^{T-\tau}\Vert u_\varepsilon(t+\tau)-u_\varepsilon(t)\Vert _Y^2\dd t&\leq  C\int_0^{T-\tau}\left[\Lambda_\varepsilon((t,t+\tau])\right]^2\dd t\\
&\leq C\Lambda_\varepsilon([0,T])\int_0^{T-\tau}\Lambda_\varepsilon((t,t+\tau])\dd t.
\end{align*}
By Fubini's theorem, we have
\[
\begin{aligned}
\int_0^{T-\tau}\Lambda_\varepsilon((t,t+\tau])\dd t&=\int_0^{T-\tau}\int_{[0,T]}\one_{\{t<r\leq t+\tau\}}\dd\Lambda_\varepsilon(r)\dd t\\
&=\int_{[0,T]}\bigg(\int_0^{T-\tau}\one_{\{t<r\leq t+\tau\}}\dd t\bigg)\dd\Lambda_\varepsilon(r)\\
&=\int_{[0,T]}\bigg(\int_0^{T-\tau}\one_{\{r-\tau\leq t\leq r\}}\dd t\bigg)\dd\Lambda_\varepsilon(r)\\
&\leq\tau\Lambda_\varepsilon([0,T]).
\end{aligned}
\]
Then, we reach
\[
\int_0^{T-\tau}\Vert u_\varepsilon(t+\tau)-u_\varepsilon(t)\Vert _Y^2\dd t\leq C\tau[\Lambda_\varepsilon([0,T])]^2.
\]
Thus, it follows from \eqref{rs-29} that
\[
\sup_\varepsilon\Vert u_\varepsilon(\cdot+\tau)-u_\varepsilon(\cdot)\Vert _{L^2(0,T-\tau;Y)}\leq C_{\alpha,M}\tau^{1/2}\rightarrow0,\qquad\text{as }\tau\downarrow0.
\]

Now, we can apply Lemma \ref{lem:simon-compactness} with
\[
B_0=H^1(\mathbb T^N),\qquad B=L^2(\mathbb T^N), \qquad B_1=H^{-r}(\mathbb T^N).
\]
Since
\[
H^1(\mathbb T^N)\subset\subset L^2(\mathbb T^N)  \hookrightarrow H^{-r}(\mathbb T^N),
\]
and
\[
\sup_\varepsilon\Vert u_\varepsilon\Vert _{L^2(0,T;H^1(\mathbb T^N))}\leq C_{\alpha,M},
\]
Lemma \ref{lem:simon-compactness} implies that the family $\{u_\varepsilon\}_{\varepsilon>0}$ is relatively compact in $L^2(0,T;L^2(\mathbb T^N))$.

\noindent \textbf{Step 2:}\quad
From Step 1, it follows that, up to a subsequence, there exists $v\in L^2(0,T;H^1(\mathbb{T}^N))$ such that
\begin{align}\label{rs-30}
  u_\varepsilon\to v\qquad\text{strongly in }L^2(Q_T).
\end{align}
Since
\[
u_\varepsilon\geq\psi\qquad\text{a.e. in }Q_T,
\]
we also have
\[
v\geq\psi\qquad\text{a.e. in }Q_T.
\]

As in \cite[Subsection 4.1.2]{DV10-publish}, using the Banach--Alaoglu theorem, a diagonal argument, and the uniform estimate \eqref{eq:skeleton-moment-bound-viscous} on the obstacle measures, there exists a nonnegative finite Radon measure $\nu$ on $\mathbb{T}^N\times[0,T)$ such that, up to a subsequence,
\[
\nu_\varepsilon\overset{\ast}{\rightharpoonup}\nu\qquad\text{in }\mathcal M(\mathbb{T}^N\times[0,T)).
\]
Similarly, there exists a kinetic measure $q$ such that, up to a subsequence,
\[
q_\varepsilon\overset{\ast}{\rightharpoonup} q\qquad\text{in }\mathcal M(\mathbb T^N\times[0,T]\times\mathbb R).
\]
In what follows, we will verify that
 the limit $(v,\nu,q)$ is a kinetic solution of \eqref{eq:skeleton-obstacle-problem-viscous} with control $h$.
 
The weak initial trace follows by applying the argument of Lemma \ref{lem:weak-initial-condition} with the deterministic control term included. Moreover, the weak convergence $\nabla u_\varepsilon\rightharpoonup\nabla v$ in $L^2(Q_T)$ and the lower-semicontinuity argument in \eqref{eq:parabolic-defect-liminf} yield
\[
q\geq\alpha|\nabla v|^2\delta_v(\dd\xi)\,\dd x\,\dd t.
\]
Moreover, 
from \eqref{eq:kinetic-skeleton-obstacle-viscous}, the kinetic function $f_\varepsilon(x,t,\xi)=\one_{\{u_\varepsilon(x,t)>\xi\}}$ satisfies,
for any $\varphi\in C_c^\infty(\mathbb T^N\times[0,T)\times\mathbb R)$,
\[
\begin{aligned}
&\int_0^T
\langle f_\varepsilon(t),\partial_t\varphi(t)\rangle\dd t+\langle f^0,\varphi(0)\rangle+\int_0^T\langle f_\varepsilon(t),a(\xi)\cdot\nabla_x\varphi(t)\rangle\dd t\\
&\quad+\alpha\int_0^T\langle f_\varepsilon(t),\Delta_x\varphi(t)\rangle\dd t\\
&=-\int_0^T\int_{\mathbb T^N}G_{h^\varepsilon}(x,t,u_\varepsilon(x,t))\varphi(x,t,u_\varepsilon(x,t))\dd x\dd t\\
&\quad+q_\varepsilon(\partial_\xi\varphi)-\int_{Q_T}\varphi(x,t,\psi(x,t))\dd\nu_\varepsilon(x,t).
\end{aligned}
\]
Define $f:=\one_{\{v(x,t)>\xi\}}$. Note that
\[
\int_{\mathbb{R}}|f_\varepsilon(x,t,\xi)-f(x,t,\xi)|\dd\xi\leq|u_\varepsilon(x,t)-v(x,t)|.
\]
Therefore, \eqref{rs-30} yields
\[
f_\varepsilon\to f \qquad\text{strongly in }L^1(Q_T\times\mathbb{R}).
\]
Then, the terms containing $f_\varepsilon$ converge to the corresponding terms containing $f$.
Moreover, from the convergence of $\nu_\varepsilon$ and $q_\varepsilon$, we have
\[
q_\varepsilon(\partial_\xi\varphi)\to q(\partial_\xi\varphi),
\]
and
\[
\int_{Q_T}\varphi(x,t,\psi(x,t))\dd\nu_\varepsilon(x,t)\to\int_{Q_T}\varphi(x,t,\psi(x,t))\dd\nu(x,t).
\]
It remains to pass to the limit in the control term.
Define
\[
\Gamma_\varepsilon(t):=\bigg(\int_{\mathbb T^N}g_k(x,u_\varepsilon(x,t))\varphi(x,t,u_\varepsilon(x,t))\dd x\bigg)_{k\geq1},
\]
and
\[
\Gamma(t):=\bigg(\int_{\mathbb T^N}g_k(x,v(x,t))\varphi(x,t,v(x,t))\dd x\bigg)_{k\geq1}.
\]
Then
\[
\int_0^T\int_{\mathbb T^N}G_{h^\varepsilon}(x,t,u_\varepsilon)\varphi(x,t,u_\varepsilon)\dd x\dd t=\int_0^T\langle h^\varepsilon(t),\Gamma_\varepsilon(t)\rangle_U\dd t.
\]

We now prove that
\[
\Gamma_\varepsilon\to\Gamma\qquad\text{strongly in }L^2(0,T;U).
\]
Since $\varphi$ is smooth and compactly supported in $\xi$, by (\ref{eq:assumption for g Lip}),
there exists $C_\varphi>0$ such that
\[
\Bigg(\sum_{k\geq1}|g_k(x,\xi)\varphi(x,t,\xi)-g_k(x,\zeta)\varphi(x,t,\zeta)|^2\Bigg)^{1/2}\leq C_\varphi|\xi-\zeta|.
\]
Therefore, we have
\[
|\Gamma_\varepsilon(t)-\Gamma(t)|_U\leq C_\varphi\int_{\mathbb T^N}|u_\varepsilon(x,t)-v(x,t)|\dd x.
\]
Hence, by \eqref{rs-30}, we obtain
\[
\begin{aligned}
\Vert \Gamma_\varepsilon-\Gamma\Vert _{L^2(0,T;U)}^2\leq C_\varphi\int_0^T\int_{\mathbb T^N}|u_\varepsilon(x,t)-v(x,t)|^2\dd x\dd t\to0.
\end{aligned}
\]
This, together with the weak convergence of $h^\varepsilon$, implies that, as $\varepsilon\rightarrow 0$,
\[
\begin{aligned}
&\int_0^T\langle h^\varepsilon(t),\Gamma_\varepsilon(t)\rangle_U\dd t-\int_0^T\langle h(t),\Gamma(t)\rangle_U\dd t\\
&=\int_0^T\langle h^\varepsilon(t),\Gamma_\varepsilon(t)-\Gamma(t)\rangle_U\dd t+\int_0^T\langle h^\varepsilon(t)-h(t),\Gamma(t)\rangle_U\dd t\\
&\rightarrow0.
\end{aligned}
\]
Hence, we arrive at
\[
\begin{aligned}
\int_0^T\int_{\mathbb T^N}G_{h^\varepsilon}(x,t,u_\varepsilon)\varphi(x,t,u_\varepsilon)\dd x\dd t\to\int_0^T\int_{\mathbb T^N}G_h(x,t,v)\varphi(x,t,v)\dd x\dd t.
\end{aligned}
\]
These show that $(v,\nu)$ is a kinetic solution of the viscous skeleton obstacle problem \eqref{eq:skeleton-obstacle-problem-viscous} with control $h$.

By uniqueness of kinetic solutions for \eqref{eq:skeleton-obstacle-problem-viscous} , we have
$
v=u_\alpha^h.
$
Therefore, for every sequence $\varepsilon_j\downarrow0$, there exists  a subsequence such that
\[
u_{\varepsilon_j}\to u_\alpha^h\qquad\text{strongly in }L^2(Q_T).
\]
This implies the convergence of the whole family:
\[
u_\alpha^{h^\varepsilon}\to u_\alpha^h\qquad\text{strongly in }L^2(Q_T).
\]
The proof is complete.
\end{proof}

We are ready to show the weak--strong continuity of the skeleton equation.
\begin{prop}\label{prp-3}
Assume \textbf{Hypotheses (H1)} and \textbf{(H2)$'$} hold.
Let $\{h^\varepsilon\}_{\varepsilon>0}\subset S_M$ and $h\in S_M$ satisfy
\[
h^\varepsilon\rightharpoonup h
\qquad
\text{weakly in }L^2(0,T;U)\quad \text{as}\ \varepsilon\rightarrow 0.
\]
Then, the kinetic solutions $u^{h^\varepsilon}$ and $u^h$ of \eqref{eq:skeleton-obstacle-problem} with controls $h^{\varepsilon}$ and $h$, and with
initial data $u^0$ satisfy that
\begin{align*}
  \lim_{\varepsilon\rightarrow 0}\Vert u^{h^{\varepsilon}}-u^{h}\Vert _{L^1(Q_T)}=0.
\end{align*}
\end{prop}
\begin{proof}
Note that for any $\alpha>0$, we have
\begin{align}\label{eq:rs-19}
  \Vert u^{h^{\varepsilon}}-u^{h}\Vert _{L^1(Q_T)}
  \leq \Vert u^{h^{\varepsilon}}-u^{h^{\varepsilon}}_{\alpha}\Vert _{L^1(Q_T)}
  +\Vert u^{h^{\varepsilon}}_{\alpha}-u^{h}_{\alpha}\Vert _{L^1(Q_T)}
  +\Vert u^{h}_{\alpha}-u^{h}\Vert _{L^1(Q_T)}.
\end{align}

For any $\iota>0$, in view of Lemma \ref{lem-1}, there exists $\alpha_0$ such that for all $\varepsilon>0$,
\begin{align*}
  \Vert u^{h^{\varepsilon}}-u^{h^{\varepsilon}}_{\alpha_0}\Vert _{L^1(Q_T)}\leq \frac{\iota}{2},\quad \Vert u^{h}_{\alpha_0}-u^{h}\Vert _{L^1(Q_T)}\leq \frac{\iota}{2}.
\end{align*}
Taking $\alpha=\alpha_0$ in \eqref{eq:rs-19}, we obtain
\begin{align*}
  \Vert u^{h^{\varepsilon}}-u^{h}\Vert _{L^1(Q_T)}
  \leq \iota
  +\Vert u^{h^{\varepsilon}}_{\alpha_0}-u^{h}_{\alpha_0}\Vert _{L^1(Q_T)}.
\end{align*}
Further, it follows from Proposition
\ref{prop:weak-continuity-viscous-skeleton-obstacle} that
\begin{align*}
  \limsup_{\varepsilon\rightarrow0}\Vert u^{h^{\varepsilon}}-u^{h}\Vert _{L^1(Q_T)}
  \leq \iota.
\end{align*}
Since the constant $\iota$ is arbitrary, we obtain the desired result.
\end{proof}

\section{Proof of large deviations for the obstacle problem}
For any sequence $\{h^{\varepsilon}; 0<\varepsilon<1\}\subset \mathcal{A}_M$, let $h^{\varepsilon}=\sum_{k\geq 1}h^{\varepsilon,k}e_k$, where $(e_k)_{k\geq 1}$ is an orthonormal basis of $U$. By the definition of $\mathcal{A}_M$, we have
\begin{align}\label{rr-2}
\int_0^T|h^{\varepsilon}(t)|_U^2\dd t\leq M,\qquad\mathbb P\text{-a.s.},\quad \varepsilon\in(0,1).
\end{align}

Consider the following obstacle problem of stochastic conservation laws with control
\begin{equation}\label{rs-1}
\left\{
  \begin{aligned}
  \dd\bar{u}^{\varepsilon}(x,t)+\Div A(\bar{u}^{\varepsilon}(x,t))\dd t&=\Phi(\bar{u}^{\varepsilon}(x,t)) h^{\varepsilon}(t)\dd t+\sqrt{\varepsilon}\Phi(\bar{u}^{\varepsilon}(x,t)) \dd W(t)+\bar{\nu}^{\varepsilon}(\dd x,\dd t),\\
  \bar{u}^{\varepsilon}(x,t)&\geq \psi(x,t),\\
\bar{u}^{\varepsilon}(0)&=u^0.
  \end{aligned}
\right.
\end{equation}

\begin{lem}[Well-posedness of the controlled stochastic equation]\label{lem:well-posedness-controlled-stochastic}
Assume \textbf{Hypotheses (H1) and (H2)$'$}. 
For every $\varepsilon\in(0,1)$ and $h^\varepsilon\in\mathcal A_M$, \eqref{rs-1} admits a unique kinetic solution $(\bar u^\varepsilon,\bar\nu^\varepsilon)$, with an associated
kinetic measure $\bar q^\varepsilon$. 
Moreover, for every $p\geq2$,
\begin{align}\label{rs-22}
\sup_{0<\varepsilon<1}\mathbb E\bigg[&\underset{0\leq t\leq T}{\operatorname{ess\,sup}}\|\bar u^\varepsilon(t)\|_{L^p(\mathbb T^N)}^p+\bar\nu^\varepsilon(\mathbb T^N\times[0,T))\notag\\
&\quad+\int_{\mathbb T^N\times[0,T]\times\mathbb R}(1+|\xi|^p)\,\dd\bar q^\varepsilon(x,t,\xi)\bigg]\leq C_{p,M},
\end{align}
where $C_{p,M}$ depends only on $p$, $M$, $T$, $D_0$, the initial datum, and the obstacle, and is independent of $\varepsilon$ and of the chosen family
$(h^\varepsilon)_{\varepsilon\in(0,1)}\subset\mathcal A_M$.
\end{lem}

\begin{proof}
Fix $\varepsilon\in(0,1)$ and write $h=h^\varepsilon$. 
Set
\[
W^{\varepsilon,h}(t):=W(t)+\frac1{\sqrt\varepsilon}\int_0^t h(s)\dd s
\]
and
\[
Z_T^{\varepsilon,h}:=\exp\left[-\frac1{\sqrt\varepsilon}\int_0^T\langle h(s),\dd W(s)\rangle_U-\frac1{2\varepsilon}\int_0^T|h(s)|_U^2\dd s\right].
\]
By \eqref{rr-2}, Novikov's condition holds without localization.
Under the probability measure
\[
\dd\mathbb Q^{\varepsilon,h}=Z_T^{\varepsilon,h}\dd\mathbb P,
\]
the process $W^{\varepsilon,h}$ is a cylindrical Wiener process.
Therefore, using the measurable solution map for \eqref{eq:intro-small-noise},
\[
\bar u^\varepsilon:=\mathcal G^\varepsilon(W^{\varepsilon,h})
\]
solves \eqref{eq:intro-small-noise} driven by $W^{\varepsilon,h}$ under $\mathbb Q^{\varepsilon,h}$. 
Since
\[
\sqrt\varepsilon\,\dd W^{\varepsilon,h}=\sqrt\varepsilon\dd W+h\dd t,
\]
it satisfies \eqref{rs-1}. 
The measures $\mathbb Q^{\varepsilon,h}$ and $\mathbb P$ are equivalent, so the kinetic formulation also holds $\mathbb P$-almost surely.
Uniqueness follows in the same way from the uniqueness for \eqref{eq:intro-small-noise}.

The uniform estimate \eqref{rs-22} follows by repeating the viscous-penalization scheme (\textbf{App Equ (1)} and \textbf{App Equ (2)}) developed in Section \ref{sec-app}.
Since
\[
\int_0^T|h(t)|_U^2\,\dd t\leq M\qquad\mathbb P\text{-a.s.},
\]
combining the arguments of Lemma \ref{lem:priori estimates} together with \eqref{control term}, we obtain \eqref{rs-22} uniformly in $\varepsilon$ and $h$.
In addition, the estimates for $\bar\nu^\varepsilon$ and $\bar q^\varepsilon$, and the passage
to the weak initial trace, follow from the remaining steps of the same approximation argument in Section \ref{sec-app}.
\end{proof}

Let
\[
f^\varepsilon:=\one_{\{\bar u^\varepsilon>\xi\}},\qquad\bar\mu^\varepsilon_{x,r}(\dd\xi):=\delta_{\bar u^\varepsilon(x,r)}(\dd\xi).
\]
The proof of Proposition \ref{prop:trace-obstacle-scl} applies without change, since the additional control term is absolutely continuous in time and is absorbed into the continuous part of the scalar pairing. Then, for all $\varphi\in C^1_c(\mathbb{T}^N\times \mathbb{R})$,
\begin{equation}
\begin{aligned}
&\langle {f}^{\varepsilon,+}(t),\varphi\rangle -\langle f(0),\varphi\rangle -\int_0^t\!\langle {f}^{\varepsilon}(r),a(\xi)\cdot\nabla_x\varphi\rangle\dd r \\
&=\sqrt{\varepsilon}\sum_{k\geq1}\int_0^t\int_{\mathbb{T}^N}\int_{\mathbb{R}}g_k(x,\xi)\varphi(x,\xi)\dd \bar{\mu}^{\varepsilon}_{x,r}(\xi)\dd x\dd\beta_k(r)\\
&\quad + \sum_{k\geq1}\int_0^t\int_{\mathbb{T}^N}\int_{\mathbb{R}}g_k(x,\xi)\varphi(x,\xi)h^{\varepsilon,k}(r)\dd \bar{\mu}^{\varepsilon}_{x,r}(\xi)\dd x\dd r\\
&\quad +\frac{\varepsilon}{2}\int_0^t\int_{\mathbb{T}^N}\int_{\mathbb{R}}G^2(x,\xi)\partial_\xi\varphi(x,\xi)\dd \bar{\mu}^{\varepsilon}_{x,r}(\xi)\dd x\dd r\\
&\quad -\int_{\mathbb{T}^N\times[0,t]\times{\mathbb{R}}}\partial_{\xi}\varphi(x,\xi)\dd \bar{q}^{\varepsilon}(x,r,\xi)
+\int_{\mathbb{T}^N\times[0,t]}\varphi(x,\psi(x,r))\dd\bar{\nu}^{\varepsilon}(x,r).
\end{aligned}\label{P-10}
\end{equation}

The weak--strong continuity established by Proposition \ref{prp-3} implies that (ii) in Condition A holds. 
Thus, according to Theorem \ref{thm-7}, to obtain the LDP, it remains to show (i) in Condition A, which is the content of the following proposition.

\begin{prop}\label{prp-1}
For any $M>0$, let $\{h^{\varepsilon}; 0<\varepsilon<1\}\subset \mathcal{A}_M$. Then
\begin{align*}
        \lim_{\varepsilon\rightarrow 0}\bigg\Vert \mathcal{G}^{\varepsilon}\Big(W(\cdot)+\frac{1}{\sqrt{\varepsilon}}\int^{\cdot}_0h^{\varepsilon}(s)\dd s\Big)-\mathcal{G}^0\Big(\int^{\cdot}_0h^{\varepsilon}(s)\dd s\Big)\bigg\Vert _{L^1(0,T;L^1(\mathbb{T}^N))}
    =0,
    \end{align*}
    in probability.
\end{prop}
\begin{proof}
By Lemma \ref{lem:well-posedness-controlled-stochastic},
\[
\bar u^\varepsilon=\mathcal G^\varepsilon\left(W(\cdot)+\frac1{\sqrt\varepsilon}\int_0^{\cdot}h^\varepsilon(s)\dd s\right)
\]
is the unique kinetic solution of \eqref{rs-1}, with associated measures $\bar q^\varepsilon$ and $\bar\nu^\varepsilon$. 
Moreover,
\[
u^{h^\varepsilon}=\mathcal G^0\left(\int_0^{\cdot}h^\varepsilon(s)\dd s\right)
\]
is the kinetic solution of \eqref{eq:skeleton-obstacle-problem}, with associated measures $q^{h^\varepsilon}$ and $\nu^{h^\varepsilon}$.

For simplicity in this proof, let $f_1:=\one_{\bar{u}^{\varepsilon}>\xi}$ and $f_2:=\one_{u^{h^{\varepsilon}}>\xi}$, $\mu^1_{x,s}(\dd\xi):=\delta_{\bar{u}^{\varepsilon}(x,s)}(\dd\xi)$ and $\mu^2_{y,s}(\dd\zeta):=\delta_{u^{h^{\varepsilon}}(y,s)}(\dd\zeta)$.

Similarly to Theorem \ref{thm:uniqueness}, with the tensorization argument of \cite[proof of Proposition 9]{DV10-publish} and based on \eqref{P-10} and \eqref{eq:weak-kinetic-2}, taking $\varphi_{\gamma,\delta}(x,\xi,y,\zeta)=\rho_{\gamma}(x-y)\kappa_\delta(\xi-\zeta)$, we reach $\mathbb{P}$-a.s.,
\begin{align*}\notag
\bar{\mathcal{R}}(t)&:=\int_{(\mathbb{T}^N)^2}\int_{\mathbb{R}^2} (f^+_1(t)\bar{f}^+_2(t)+\bar{f}^+_1(t)f^+_2(t))\rho_{\gamma}(x-y)\kappa_\delta(\xi-\zeta)\dd x\dd\xi\dd y\dd\zeta \\
\notag
&\leq\bar{\mathcal{E}}_0(\gamma,\delta)+K_{\gamma,\delta}^{\rm{flux}}(t)+K_{\gamma,\delta}^h(t)+K_{\gamma,\delta}^{\rm{mart}}(t)+K_{\gamma,\delta}^{\rm{quad}}(t)+K_{\gamma,\delta}^{q}(t)+K_{\gamma,\delta}^{\nu}(t),
\end{align*}
where 
\[
\bar{\mathcal{E}}_0(\gamma,\delta):=\int_{(\mathbb{T}^N)^2}\int_{\mathbb{R}^2}(f_1(0)\bar{f}_2(0)+\bar{f}_1(0)f_2(0))\rho_{\gamma}(x-y)\kappa_\delta(\xi-\zeta)\dd x\dd\xi\dd y\dd\zeta,
\]
and $K_{\gamma,\delta}^{\rm{flux}}, K_{\gamma,\delta}^{q}(t)$ and $K_{\gamma,\delta}^{\nu}$ are the same as those in Theorem \ref{thm:uniqueness},
the term related to the control term is given by
\begin{align*}
K_{\gamma,\delta}^h(t)&:=\sum_{k\geq1}\int_0^t\int_{(\mathbb{T}^N)^2}\int_{\mathbb{R}^2}(\bar{f}_2-f_2)g_k(x,\xi)h^{\varepsilon,k}(r)\rho_{\gamma}(x-y)\kappa_\delta(\xi-\zeta)\dd \mu^1_{x,r}(\xi)\dd \zeta \dd x \dd y \dd r\\
  &\quad-\sum_{k\geq1}\int_0^t\int_{(\mathbb{T}^N)^2}\int_{\mathbb{R}^2}(f_1-\bar{f}_1)g_k(y,\zeta)h^{\varepsilon,k}(r)\rho_{\gamma}(x-y)\kappa_\delta(\xi-\zeta)\dd \mu^2_{y,r}(\zeta)\dd \xi\dd x \dd y\dd r,
\end{align*}
the term associated with the martingale term is
\begin{align*}
K_{\gamma,\delta}^{\rm{mart}}(t):=\sqrt{\varepsilon}\sum_{k\geq1}\int_0^t\int_{(\mathbb{T}^N)^2}(\bar{f}_2-f_2)g_k(x,\xi)\rho_{\gamma}(x-y)\kappa_{\delta}(\xi-\zeta)\dd \zeta\dd \mu^1_{x,r}(\xi)\dd x\dd y \dd\beta_k(r),
\end{align*}
and the term generated by the quadratic term is
\begin{align*}
K^{{\rm{quad}}}_{\gamma,\delta}(t):= \frac{\varepsilon}{2}\int^t_0\int_{(\mathbb{T}^N)^2} \int_{\mathbb{R}^2}(\bar{f}_2-f_2)\rho_{\gamma}(x-y)\partial_{\xi}\kappa_{\delta}(\xi-\zeta)G^2(x,\xi)\dd\mu^1_{x,r}(\xi)\dd \zeta \dd x\dd y \dd r.
\end{align*}

We will estimate them one by one.
Referring to \eqref{eq:estimates for Ia-1} in Theorem \ref{thm:uniqueness}, we have $\mathbb{P}$-a.s.,
\begin{align*}
K_{\gamma,\delta}^{\rm{flux}}&\leq C\delta \int^t_0\int_{(\mathbb{T}^N)^2}|\nabla_x \rho_{\gamma}(x-y)|(1+|\bar{u}^{\varepsilon}(x,r)|^p+|u^{h^{\varepsilon}}(y,r)|^p)\dd x\dd y\dd r\\
&\leq C\delta \gamma^{-1}\int^t_0(1+\|\bar{u}^{\varepsilon}\|^p_{L^p(\mathbb{T}^N)}+\|u^{h^{\varepsilon}}\|^p_{L^p(\mathbb{T}^N)})\dd r.
\end{align*}
With regard to $K_{\gamma,\delta}^{\nu}$, it follows from \eqref{eq:I-nu-final-estimate} that $\mathbb{P}$-a.s.,
\begin{align*}
K_{\gamma,\delta}^{\nu}(t)\leq C\frac{\gamma^2}{\delta}\Big[\bar{\nu}^\varepsilon(\mathbb T^N\times[0,t])+\nu^{h^\varepsilon}(\mathbb T^N\times[0,t])\Big].
\end{align*}
Furthermore, \eqref{eq:estimates for Iqvar} yields that $\mathbb{P}$-a.s.,
\begin{align*}
K_{\gamma,\delta}^{q}(t)\leq 0.
\end{align*}

By the same method as (\ref{rs-20}), using \eqref{rr-2}, we have $\mathbb{P}$-a.s.,
\begin{align*}\notag
  K_{\gamma,\delta}^h(t)
 &\leq 2\sqrt{D_0}\gamma \int^t_0|h^{\varepsilon}(r)|_U \dd r+ 4 \delta \sqrt{D_0}\int^t_0|h^{\varepsilon}(r)|_U \dd r+2\sqrt{D_0}\int^t_0|h^{\varepsilon}(r)|_U\bar{\mathcal{R}}(r)\dd r\\
&\leq 2\sqrt{D_0}\gamma \sqrt{M}\sqrt{T}+4 \delta \sqrt{D_0}\sqrt{M}\sqrt{T}
+ 2\sqrt{D_0}\int^t_0|h^{\varepsilon}(r)|_U\bar{\mathcal{R}}(r)\dd r.
\end{align*}
By integration by parts and (\ref{eq:assumption for g}), we have $\mathbb{P}$-a.s.,
\begin{align*}
K^{{\rm{quad}}}_{\gamma,\delta}&\leq D_0\varepsilon\int^t_0\int_{(\mathbb{T}^N)^2} \int_{\mathbb{R}^2}\rho_{\gamma}(x-y)\kappa_{\delta}(\xi-\zeta)(1+|\xi|^2)\dd\mu^1_{x,r}\otimes \mu^2_{y,r}(\xi,\zeta) \dd x\dd y\dd r\\
&\leq D_0\varepsilon\delta^{-1}\int^t_0\int_{(\mathbb{T}^N)^2}\rho_{\gamma}(x-y)(1+|\bar{u}^{\varepsilon}(x,r)|^2)\dd x\dd y\dd r\\
&\leq D_0\varepsilon\delta^{-1}\bigg(T+\int^t_0\|\bar{u}^{\varepsilon}(r)\|^2_{L^2(\mathbb{T}^N)}\dd r\bigg).
\end{align*}

For the martingale term, define
\[
\Theta_\delta(r,s):=\int_{\mathbb R}\left(\one_{\{s\leq\zeta\}}-\one_{\{s>\zeta\}}\right)\kappa_\delta(r-\zeta)\dd\zeta.
\]
Since $\kappa_\delta\geq0$ and $\int_{\mathbb R}\kappa_\delta(z)\dd z=1$, we have
\begin{equation}\label{eq:Theta-delta-bound}
|\Theta_\delta(r,s)|\leq1,\qquad r,s\in\mathbb R.
\end{equation}
Using the Young measures $\mu^1_{x,r}=\delta_{\bar u^\varepsilon(x,r)}$ and $\mu^2_{y,r}=\delta_{u^{h^\varepsilon}(y,r)}$, the martingale term can
be written as
\begin{equation*}
K_{\gamma,\delta}^{\rm mart}(t)=\sqrt\varepsilon\sum_{k=1}^\infty\int_0^t B_k(r)\dd\beta_k(r),
\end{equation*}
where
\[
\begin{aligned}
B_k(r):=\int_{(\mathbb T^N)^2}\rho_\gamma(x-y)g_k\bigl(x,\bar u^\varepsilon(x,r)\bigr)\Theta_\delta\bigl(\bar u^\varepsilon(x,r),u^{h^\varepsilon}(y,r)\bigr)\dd x\dd y.
\end{aligned}
\]
Since $\rho_\gamma(x-y)\dd x\dd y$ has uniformly bounded total mass, Jensen's inequality and \eqref{eq:Theta-delta-bound} yield
\begin{align*}
\sum_{k\geq1}|B_k(r)|^2&=\sum_{k\geq1}\left|\int_{(\mathbb T^N)^2}\rho_\gamma(x-y)g_k\bigl(x,\bar u^\varepsilon(x,r)\bigr)
\Theta_\delta\bigl(\bar u^\varepsilon(x,r),u^{h^\varepsilon}(y,r)\bigr)\dd x\dd y\right|^2\\
&\leq C\Big(\int_{(\mathbb T^N)^2}\rho_\gamma(x-y)G^2\bigl(x,\bar u^\varepsilon(x,r)\bigr)\dd x\dd y\Big)\Big(\int_{(\mathbb T^N)^2}\rho_\gamma(x-y)\dd x\dd y \Big)\\
&\leq C\left(1+\|\bar u^\varepsilon(r)\|_{L^2(\mathbb T^N)}^2\right).
\end{align*}
Therefore, the Burkholder--Davis--Gundy inequality gives
\begin{align}
\mathbb E\sup_{t\in[0,T]}\left|K_{\gamma,\delta}^{\rm mart}(t)\right|
&\leq C\sqrt\varepsilon\mathbb E\left[\int_0^T\sum_{k\geq1}|B_k(r)|^2\,\dd r
\right]^{1/2}\notag\\
&\leq C\sqrt\varepsilon\mathbb E\left[T+\int_0^T\|\bar u^\varepsilon(r)\|_{L^2(\mathbb T^N)}^2\dd r\right]^{1/2}\notag\\
&\leq C\sqrt\varepsilon.
\label{rs-26}
\end{align}

Denote by
\begin{align}\notag
  \Upsilon_{\varepsilon,\gamma,\delta}(t)&:=C\delta \gamma^{-1}\int^t_0(1+\|\bar{u}^{\varepsilon}\|^p_{L^p(\mathbb{T}^N)}+\|u^{h^{\varepsilon}}\|^p_{L^p(\mathbb{T}^N)})\dd r\\ \notag
  &+C\frac{\gamma^2}{\delta}
\Big[\bar{\nu}^\varepsilon(\mathbb T^N\times[0,t])+\nu^{h^\varepsilon}(\mathbb T^N\times[0,t])\Big]
+ (2\gamma +4 \delta) \sqrt{D_0}\sqrt{M}\sqrt{T}
\\
\label{rs-23}
&+ D_0\varepsilon\delta^{-1}\bigg(T+\int^t_0\|\bar{u}^{\varepsilon}\|^2_{L^2(\mathbb{T}^N)}\dd r\bigg).
\end{align}
Then, based on the above estimates, we arrive at $\mathbb{P}$-a.s.,
\begin{align*}
\bar{\mathcal{R}}(t)&\leq \bar{\mathcal{E}}_0(\gamma,\delta)+\Upsilon_{\varepsilon,\gamma,\delta}(t)+2\sqrt{D_0}\int^t_0|h^{\varepsilon}(r)|_U\bar{\mathcal{R}}(r)\dd r+\sup_{0\leq s\leq t}|K^{{\rm{mart}}}_{\gamma,\delta}(s)|,
\end{align*}
where $\Upsilon_{\varepsilon,\gamma,\delta}(t)$ is given by \eqref{rs-23}.

Applying Gr\"onwall's inequality, using \eqref{rr-2}, we deduce that $\mathbb{P}$-a.s.,
\begin{align}\notag
\bar{\mathcal{R}}(t)&\leq \exp\{2\sqrt{D_0}\int^t_0|h^{\varepsilon}(r)|_U\dd r\}\Big[ \bar{\mathcal{E}}_0(\gamma,\delta)+\Upsilon_{\varepsilon,\gamma,\delta}(t)+\sup_{0\leq s\leq t}|K^{{\rm{mart}}}_{\gamma,\delta}(s)|\Big]\\ 
&\leq \exp\{2\sqrt{D_0}\sqrt{MT}\}\Big[ \bar{\mathcal{E}}_0(\gamma,\delta)+\Upsilon_{\varepsilon,\gamma,\delta}(t)+\sup_{0\leq s\leq t}|K^{{\rm{mart}}}_{\gamma,\delta}(s)|\Big].\label{rs-24}
\end{align}
For any $t\in [0,T]$, define the error term
\begin{align*}
&\bar{\mathcal{E}}_t(\varepsilon,\gamma,\delta)\\
&:=\int_{\mathbb{T}^N}\int_{\mathbb{R}}(f^+_1(x,t,\xi)\bar{f}^+_2(x,t,\xi)+\bar{f}^{+}_1(x,t,\xi){f}^+_2(x,t,\xi))\dd x\dd\xi\\
&\ -\int_{(\mathbb{T}^N)^2}\int_{\mathbb{R}^2}\rho_\gamma (x-y)\kappa_{\delta}(\xi-\zeta)(f^{+}_1(x,t,\xi)\bar{f}^{+}_2(y,t,\zeta)+\bar{f}^{+}_1(x,t,\xi){f}^{+}_2(y,t,\zeta))\dd\xi \dd\zeta \dd x\dd y.
\end{align*}
Similarly to the proof of \eqref{eq:rs-14}-\eqref{eq:rs-17}, we have
\begin{align*}
&\bar{\mathcal{E}}_t(\varepsilon,\gamma,\delta)\\
&\leq 4\delta+ \int_{(\mathbb{T}^N)^2}\int_{\mathbb{R}^2}\rho_\gamma (x-y)\kappa_{\delta}(\xi-\zeta)(f^{+}_2(x,t,\xi)\bar{f}^{+}_2(y,t,\zeta)+\bar{f}^{+}_2(x,t,\xi){f}^{+}_2(y,t,\zeta))\dd\xi \dd\zeta \dd x\dd y.
\end{align*}
Observe that $f_2$ does not have martingale or quadratic terms. 
Applying the deterministic analogue of \eqref{rs-24} to two copies of $u^{h^\varepsilon}$, and observing that there are no martingale or It\^o-correction terms, we obtain $\mathbb P$-a.s.
\begin{align}\notag
&\bar{\mathcal{E}}_t(\varepsilon,\gamma,\delta)\\ \notag
&\leq 4\delta+ \exp\{2\sqrt{D_0}\sqrt{MT}\}\Big[ \bar{\mathcal{E}}_0(\gamma,\delta)
+C\delta \gamma^{-1}\int^t_0(1+2\|u^{h^{\varepsilon}}\|^p_{L^p(\mathbb{T}^N)})\dd r\\
\label{rs-25}
  &\quad+2C\frac{\gamma^2}{\delta}\nu^{h^\varepsilon}(\mathbb T^N\times[0,t])
+ (2\gamma+4 \delta) \sqrt{D_0}\sqrt{M}\sqrt{T}\Big].
\end{align}
It follows from \eqref{rs-24} and \eqref{rs-25} that $\mathbb{P}$-a.s.,
\begin{align*}
&\underset{0\leq t\leq T}{{\rm{ess\sup}}}\ \|\bar{u}^{\varepsilon}(t)-u^{h^{\varepsilon}}(t)\|_{L^1(\mathbb{T}^N)}\\
&\leq \sup_{t\in [0,T]}\bar{\mathcal{E}}_t(\varepsilon,\gamma,\delta)+\exp\{2\sqrt{D_0}\sqrt{MT}\}\Big[ \bar{\mathcal{E}}_0(\gamma,\delta)
+\sup_{t\in [0,T]}\Upsilon_{\varepsilon,\gamma,\delta}(t)+\sup_{t\in [0,T]}|K^{{\rm{mart}}}_{\gamma,\delta}(t)|\Big]\\ \notag
&\leq 4\delta+ \exp\{2\sqrt{D_0}\sqrt{MT}\}\Big[ \bar{\mathcal{E}}_0(\gamma,\delta)
+C\delta \gamma^{-1}\int^T_0(1+2\|u^{h^{\varepsilon}}\|^p_{L^p(\mathbb{T}^N)})\dd r\\
  &\quad+2C\frac{\gamma^2}{\delta}
\nu^{h^\varepsilon}(\mathbb T^N\times[0,T))
+ (2\gamma+4 \delta) \sqrt{D_0}\sqrt{M}\sqrt{T}\Big]\\
&\quad+\exp\{2\sqrt{D_0}\sqrt{MT}\}\Big[ \bar{\mathcal{E}}_0(\gamma,\delta)
+\sup_{t\in [0,T]}\Upsilon_{\varepsilon,\gamma,\delta}(t)+\sup_{t\in [0,T]}|K^{{\rm{mart}}}_{\gamma,\delta}(t)|\Big].
\end{align*}
Taking the expectation, then using \eqref{rs-26}, \eqref{eq:skeleton-moment-bound} and \eqref{rs-22},  we obtain
\begin{align*}
&\E \ \underset{0\leq t\leq T}{{\rm{ess\sup}}}\ \|\bar{u}^{\varepsilon}(t)-u^{h^{\varepsilon}}(t)\|_{L^1(\mathbb{T}^N)}\\
&\leq 4\delta+ \exp\{2\sqrt{D_0}\sqrt{MT}\}\Big[ 2\bar{\mathcal{E}}_0(\gamma,\delta)
+C(M,p,D_0,T)\delta \gamma^{-1}\\
 &\quad +C(M,p,D_0,T)\frac{\gamma^2}{\delta}
+C(M,D_0,T)(2\gamma+ 4\delta+\varepsilon \delta^{-1}+\sqrt{\varepsilon})\Big].
\end{align*}
Let $\gamma=\varepsilon^{\frac{1}{3}}$ and $\delta=\gamma^{\frac{3}{2}}$.
Using the continuity of spatial translations in
$L^1(\mathbb T^N)$, we reach
\begin{align*}
&\lim_{\varepsilon\rightarrow 0}\E \ \underset{0\leq t\leq T}{{\rm{ess\sup}}}\ \|\bar{u}^{\varepsilon}(t)-u^{h^{\varepsilon}}(t)\|_{L^1(\mathbb{T}^N)}=0.
\end{align*}
This, together with the Chebyshev inequality, yields
\begin{align*}
\|\bar{u}^{\varepsilon}-u^{h^{\varepsilon}}\|_{L^1(0,T;L^1(\mathbb{T}^N))}\longrightarrow0\qquad\text{in probability}.
\end{align*}
This completes the proof.
\end{proof}

\bibliographystyle{plain}
\bibliography{bi}
\end{document}